%% file: master-document.tex
\documentclass[a4paper,12pt,reqno,oneside]{amsart}
\input{declarations}

\title{On string topology of classifying spaces}
\date{\today}

\author{Richard Hepworth}
\address[R.~Hepworth]{Institute of Mathematics\\
University of Aberdeen\\
Aberdeen AB24 3UE\\
United Kingdom}
\email{r.hepworth@abdn.ac.uk}

\author{Anssi Lahtinen}
\address[A.~Lahtinen]{Department of mathematics\\
KTH Royal Institute of Technology\\
SE-10044  Stockholm\\
Sweden}
\email{aslahtin@kth.se}


\begin{document}

\begin{abstract}
Let $G$ be a compact Lie group.
By work of Chataur and Menichi,
the homology of the space of free loops in the classifying space of $G$
is known to be the value on the circle in a homological conformal field theory.
This means in particular that it admits operations parameterized by
homology classes of classifying spaces of diffeomorphism groups of surfaces.
Here we present a radical extension of this result,
giving a new construction in which 
diffeomorphisms are replaced with homotopy equivalences,
and surfaces with boundary are replaced
with arbitrary spaces homotopy equivalent to finite graphs. 
The result is a novel kind of field theory
which is related to both the diffeomorphism groups of surfaces
and the automorphism groups of free groups with boundaries.
Our work shows that the algebraic structures in string 
topology of classifying spaces can be brought into line with, 
and in fact far exceed, those available in string topology of manifolds.
For simplicity, we restrict to the characteristic 2 case.
The generalization to arbitrary characteristic will be addressed
in a subsequent paper.
\end{abstract}

\maketitle

\setcounter{tocdepth}{1}
\tableofcontents

\input{introduction}

\input{h-graphs-and-field-theories}

\input{overview-of-the-construction}

\input{categorical-background}

\input{the-push-pull-construction}

\input{constructing-the-umkehr-maps}

\input{comparison}

\input{higher-operations}

\appendix

\input{app-fibredspaces}

\input{app-Fd-fibrations}

\bibliographystyle{plain}
\bibliography{master-document.bib}

\end{document}

%% file: introduction.tex

\section{Introduction}

\subsection{Background}

String topology was introduced by Chas and Sullivan in \cite{ChasSullivan}.
They showed that if $M$ is a closed smooth manifold of dimension $d$,
then the homology $H_\ast(LM)$ of the free loop space of $M$ admits the structure of a 
Batalin--Vilkovisky algebra.
After Chas and Sullivan's seminal work, much effort has been made to
understand and extend the structure they discovered, and to obtain similar
structures from objects other than manifolds.  See for example
Cohen and Jones' homotopy-theoretic construction of the string
topology BV-algebra~\cite{CohenJones}, Cohen and Godin's construction
of a topological field theory~\cite{CohenGodin},
Godin's construction of a homological conformal field theory~\cite{Godin},
and Poirier and Rounds' construction of operations parameterized by a space
of string diagrams~\cite{PoirierRounds}, all with input a closed oriented manifold.
For constructions involving objects other than manifolds see
Behrend, Ginot, Noohi and Xu's paper on 
string topology for stacks~\cite{BGNX}, 
Chataur and Menichi's paper on
string topology for classifying spaces~\cite{ChataurMenichi},
Lupercio, Uribe and Xicotencatl's paper on 
orbifold string topology \cite{OrbifoldStringTopology}, 
and
Felix and Thomas' construction of string topology operations
for Gorenstein spaces~\cite{FelixThomas}.
The study of algebraic structures on Hochschild homology of
algebras with appropriate structure has developed in parallel
with string topology, and sometimes in direct connection with it.
See for example the work of  Cohen and Jones~\cite{CohenJones},
Costello~\cite{Costello}, Tradler and Zeinalian~\cite{TradlerZeinalianOne,TradlerZeinalianTwo},
F{\'e}lix, Thomas and Vigu{\'e}-Poirrier~\cite{FelixThomasVigue},
Wahl and Westerland~\cite{WahlWesterland} and Wahl~\cite{Wahl}.

The immediate background for the present paper is the work of
Godin~\cite{Godin} and Chataur and Menichi~\cite{ChataurMenichi}.
Godin showed that if $M$ is a closed oriented
manifold of dimension $d$, then $H_\ast(LM)$ is the value on 
$S^1$ of a homological conformal field
theory (or HCFT) of degree $d$.
Such an HCFT consists of vector spaces $\phi_\ast(X)$, one for each
compact $1$-manifold $X$, and operations
\begin{equation}\label{HCFTop1:eq}
	H_\ast(B\Diff(\Sigma);\partial_\Sigma^{\tensor d})\otimes\phi_\ast(X)
	\to
	\phi_\ast(Y),
\end{equation}
one for each cobordism $\Sigma$ from $X$ to $Y$.
Here $\Diff(\Sigma)$ is the topological 
group of orientation-preserving diffeomorphisms
of $\Sigma$ that fix $X$ and $Y$ pointwise,
and $\partial_\Sigma$ is the determinant coefficient system.
The operations are required to be compatible with the formation
of composites and disjoint unions of cobordisms, and with diffeomorphisms
of cobordisms.
(We note that under modest assumptions on $\Sigma$, 
the space $B\Diff(\Sigma)$ is homotopy equivalent to 
a moduli space of Riemann surfaces modelled on $\Sigma$, 
providing motivation for the word ``conformal'' in the terminology.)
On the other hand, Chataur and Menichi
showed that if $G$ is a compact Lie
group of dimension $d$, either connected or finite, then
$H_\ast(LBG)$ is the value on $S^1$ of a HCFT of degree $-d$.
However, the HCFT they constructed is of a more restricted type
than that constructed by Godin. For example, in their theory
Chataur and Menichi only allow \emph{closed} cobordisms 
$\Sigma\colon X \to Y$, meaning ones where $X$ and $Y$ consist
of circles (``closed strings'')
and together comprise the whole boundary of $\Sigma$;
whereas Godin's theory allows more general \emph{open-closed} 
cobordisms where $X$ and $Y$ consist of circles and intervals 
(``open strings''), and parts of the boundary of $\Sigma$
may be \emph{free}, that is, not belong to either $X$ or $Y$.

\subsection{Statement of the main result}
\label{statement:subs} 
In this paper, working in characteristic 2 throughout,
we extend Chataur and Menichi's  theory to
an entirely new kind of structure that is richer and more complex
than any HCFT.   We call this new structure a
\emph{homological h-graph field theory} or  \emph{HHGFT}. 
Roughly speaking, an HHGFT is a structure analogous to an HCFT, 
but with 1-manifolds and cobordisms replaced 
by spaces homotopy equivalent to finite graphs, and with
diffeomorphisms replaced by homotopy equivalences. 

More precisely, by an \emph{h-graph} we mean a space with the homotopy 
type of a finite graph,
and by an \emph{h-graph cobordism} $S\colon X\hto Y$ 
between h-graphs $X$ and $Y$ we mean a diagram
\[X \longincl S \longrevincl Y\]
satisfying certain conditions:
the inclusion of $X\sqcup Y$ into $S$ 
must be a closed cofibration,
and it must be possible to obtain $S$ up to homotopy equivalence 
by adding points and arcs to $Y$.
More generally, by a \emph{family} of h-graph cobordisms
$S/B\colon X\hto Y$ over a base space $B$ we mean 
a fibration $S\to B$ equipped with 
closed fibrewise cofibrations over $B$
\[ B\times X\longincl S \longrevincl B\times Y\]
restricting to an h-graph cobordism $X\hto Y$ in each fibre over $B$.
An h-graph cobordism $S\colon X\hto Y$ is called \emph{positive}
if the image of $X$ in $S$ meets every component,
and a family of h-graph cobordisms $S/B\colon X\hto Y$ is called \emph{positive}
if it restricts to a positive h-graph cobordism in every fibre over $B$.
A \emph{(positive) homological h-graph field theory} or \emph{HHGFT} $\Phi$
of degree $d$ over a field $\F$ of characteristic 2 
now consists of an $\F$-vector space $\Phi_\ast(X)$ for each h-graph $X$,
together with an operation
\begin{equation}\label{HHGFTop:eq}
	\Phi(S/B) 
	\colon
	H_{\ast-d\chi(S,X)}(B;\,\F)\otimes\Phi_\ast(X)
	\longto
	\Phi_{\ast}(Y)
\end{equation}
for each (positive) family of h-graph cobordisms $S/B\colon X\hto Y$.
Here $\chi(S,X)$ is given by the Euler characteristics of the fibres
of $S$ relative to $X$.
The operations \eqref{HHGFTop:eq}
are required to be compatible with compositions and disjoint unions,
and also with base-change of families of h-graph cobordisms.
These requirements are made precise in Definition~\ref{def:HHGFT}.
We can now state our main result.

\begin{theorem}\label{thm:main}
Let $G$ be a compact Lie group, and let $\F$ be a field of characteristic 2.
Then there is a positive homological h-graph field theory 
$\Phi^G$ over $\F$ whose value on
an h-graph $X$ is $\Phi^G_\ast(X)=H_\ast(BG^X;\,\F)$.
Here $BG^X$ is the space of maps $X\to BG$.
\end{theorem}

The theorem will be proved by giving an explicit construction
of the theory $\Phi^G$. We will compare our theory with 
Chataur and Menichi's and Godin's theories
in subsection~\ref{subsec:CM-Godin} below;
in particular, as we will see there, the HHGFT $\Phi^G$
restricts to an HCFT that is in a certain sense the closest 
possible analogue to Godin's HCFT in the context of 
string topology of classifying spaces. 
Moreover, we will show in section~\ref{sec:ChataurMenichi} that our $\Phi^G$
is an extension of Chataur and Menichi's HCFT.
For the rest of the paper, unless mentioned otherwise, 
we will work over a fixed field $\F$ 
of characteristic 2, so that our vector spaces are 
over $\F$, and our homology groups are taken with coefficients in $\F$.

To demonstrate the nontriviality of our theory, we explicitly compute
the operation that $\Phi^{\Z/2}$ associates to a  family of
h-graph cobordisms $S/B(\Z/2) \colon \pt \hto \pt$
defined in Definition~\ref{def:family-s}.

\begin{theorem}
\label{thm:calc}
The operation
\[
	\Phi^{\Z/2}\big(S/B(\Z/2)\big)
	\colon 
	H_\ast B(\Z/2) \tensor H_\ast B(\Z/2)
	\longto 
	H_\ast B(\Z/2)
\]
is the map given by
\[
	a\tensor b \longmapsto
		\begin{cases}
		 a \cdot b &\text{if the degree of $a$ is positive}\\
		 0 &\text{if the degree of $a$ is $0$}
		\end{cases}
\]
for homogeneous elements $a,b \in H_\ast B(\Z/2)$.
\end{theorem}

Here $a\cdot b$ denotes the Pontryagin product of 
$a,b\in H_\ast B(\Z/2)$
induced by the addition map $\Z/2\times\Z/2\to\Z/2$.  
In particular,
for each $i>0$ and for each nonzero $a\in H_i B(\Z/2)$,
we obtain a nontrivial operation
\[
	\Phi^{\Z/2}\big(S/B(\Z/2)\big)(a\otimes -)
	\colon
	H_\ast B(\Z/2)\longrightarrow H_{\ast+i}B(\Z/2)
\]
which for $i>1$ cannot correspond to any HCFT operation.
See Remark~\ref{rk:calc}.

Calculations of higher string topology operations have been few and far between
in the literature.   To our knowledge, Theorem~\ref{thm:calc}
gives the first examples
of non-trivial higher operations in string topology of classifying spaces
other than the Batalin--Vilkovisky $\Delta$-operator.  In particular, it gives the first
infinite family of such operations.  (See Wahl's recent paper \cite[section 4]{Wahl}
for the construction of infinite families of non-trivial higher string topology operations
in the manifold case.)

\subsection{Benefits of the extension}

We will now briefly highlight some of the benefits brought 
by the extension of an HCFT into an HHGFT. 
For further discussion as well as elaboration of some of the 
points made here, see subsection~\ref{sbs:compare}.

\subsubsection{New cobordisms and factorisations}
\label{subsubsec:new-cobordisms}
One advantage of passing from homological conformal field
theories to homological h-graph field theories
is that there are many more h-graph cobordisms
than ordinary open-closed cobordisms, even between $1$-manifolds.
One such example is the disk $D^2\colon S^1\hto I$ shown on the left.
\[\begin{tikzpicture}[scale=0.03,baseline=-2]
	\path[ARC, fill=gray!20] (0,12) .. controls (10,12) .. (20,6)
	-- (20,-6)
	.. controls (10,-12) .. (0,-12);
	\path[ARC, fill=gray!50] (0,0) ellipse (6 and 12);
\end{tikzpicture}
\qquad\qquad\qquad\qquad
\begin{tikzpicture}[scale=0.03,baseline=-2]
	\path[ARC, fill=gray!20, join=round] (0,18) .. controls (25,18) and (20,6) .. (40,6)
		arc [start angle=90, end angle=-90, x radius=3, y radius=6]
		.. controls (20,-6) and (25,-18) .. (0,-18)
		-- (0,-6)
		.. controls (5,-6) and (15,-5) .. (15,0)
		.. controls (15,5) and (5,6) .. (0,6)
		-- cycle;	
	\path[ARC, densely dotted] (40,-6) 
		arc [start angle=270, end angle=90, x radius=3, y radius=6];
	\path[ARC,fill=gray!50] (0,12) ellipse (3 and 6);
	\path[ARC,fill=gray!50] (0,-12) ellipse (3 and 6);
\end{tikzpicture}
\quad=\quad
\begin{tikzpicture}[scale=0.03,baseline=-2]
	\path[ARC, fill=gray!20, join=round]
	(-2,6)
	-- (-2,18)
	-- (15,18)
	.. controls (20,18) and (20,5) .. (15,0)
	.. controls (20,-5) and (20,-18) .. (15,-18)
	-- (-2,-18)
	-- (-2,-6)
	.. controls (5,-6) and (14.055,-5) .. (14.055,0)
	.. controls (14.055,5) and (5,6) .. (-2,6)
	-- cycle;
	\path[ARC, densely dotted]
		(15,0) .. controls (10,5) and (10,18) .. (15,18);
	\path[ARC, densely dotted]
		(15,0) .. controls (10,-5) and (10,-18) .. (15,-18);	
	\path[ARC, fill=gray!50] (-2,12) ellipse (3 and 6);
	\path[ARC, fill=gray!50] (-2,-12) ellipse (3 and 6);
\end{tikzpicture}
\circ
\begin{tikzpicture}[scale=0.03,baseline=-2]
	\path[ARC, fill=gray!20, join=round]
	(15,18)
	.. controls (30,18) and (25,6) .. (40,6)
	arc [start angle=90, end angle=-90, x radius=3, y radius=6]
	.. controls (25,-6) and (30,-18) .. (15,-18)
	-- cycle;	
	\path[ARC, densely dotted] (40,-6) 
		arc [start angle=270, end angle=90, x radius=3, y radius=6];
	\path[ARC, fill=gray!50]
	(15,18)
	.. controls (20,18) and (20,5) .. (15,0)
	.. controls (10,5) and (10,18) .. (15,18)
	-- cycle;
	\path[ARC, fill=gray!50]
	(15,-18)
	.. controls (20,-18) and (20,-5) .. (15,0)
	.. controls (10,-5) and (10,-18) .. (15,-18)
	-- cycle;
\end{tikzpicture}\]
The result is a host of new field theory operations.
In addition to being interesting in their own right, 
these new operations have implications for the structure 
of the underlying HCFT of the HHGFT.
For example, the disk $D^2\colon S^1 \hto I$
above is easily seen to induce a retraction of coalgebras from the 
value of the HCFT on $S^1$ onto its value on $I$; 
see Proposition~\ref{prop:retraction}.

The supply of new cobordisms also
allows us to find new factorisations of ordinary cobordisms,
as displayed above on the right.  By applying Theorem~\ref{thm:main},
the factorisation depicted shows that
the loop product
$H_\ast(LBG)\otimes H_\ast(LBG)\to H_{\ast+\dim G}(LBG)$
factors through the space $H_\ast(BG^{S^1\vee S^1})$.
This is an exact reflection, in the structure of the HHGFT,
of the usual way one defines the loop product
(as in, for example, Chas--Sullivan~\cite{ChasSullivan})
by first restricting to pairs of loops that can be concatenated,
and then concatenating those pairs.

\subsubsection{Homotopy automorphisms}
The extension from homological conformal field theories to
homological h-graph field theories brings another entirely 
new aspect into the theory,
namely that it replaces diffeomorphisms with homotopy equivalences,
as we now explain.

Given an h-graph cobordism $S\colon X\hto Y$, we write
$\hAut(S)$ for the topological monoid of self homotopy equivalences
of $S$ that fix $X$ and $Y$ pointwise,
and $\hAut^w(S)$ for the homotopy equivalent `whiskered' monoid obtained
by attaching an interval at the identity element of $\hAut(S)$; see
\cite[Definition~A.8]{MayGILS}.
Then the classifying space $B\hAut^w(S)$
parameterizes a family of h-graph cobordisms
$U\hAut^w(S)/B\hAut^w(S)\colon X\hto Y$ that in each fibre is homotopy equivalent
to $S$ itself.
The families $U\hAut^w(S)/B\hAut^w(S)$ satisfy a certain universal property,
one consequence of which is that as $S$ runs through all h-graph
cobordisms, the resulting operations
\begin{equation}\label{UniversalOperation:eq}
	H_{\ast-d\chi(S,X)}(B\hAut^w(S))\otimes\Phi_\ast(X)
	\longto
	\Phi_{\ast}(Y)
\end{equation}
determine all of the operations \eqref{HHGFTop:eq} in the HHGFT.
See subsection~\ref{subsec:universal-families} and 
Remarks~\ref{rk:universal-operations} and 
\ref{rk:families-vs-universal-families}.

If we take $S=\Sigma$ an open-closed cobordism, then there is a map 
\begin{equation}\label{eq:comparison}
	B\Diff(\Sigma)\longto B\hAut^w(\Sigma)
\end{equation}
via which the HCFT operations
\[
	H_{\ast-d\chi(\Sigma,X)}(B\Diff(\Sigma))\tensor\Phi_\ast(X)
	\longto
	\Phi_{\ast}(Y)
\]
 factor through
\eqref{UniversalOperation:eq}.
In the case where $\Sigma$ has no free boundary, 
this map \eqref{eq:comparison} is a homotopy equivalence.
On the other hand, when $\Sigma$ does have free boundary,
\eqref{eq:comparison} can be far from a homotopy equivalence,
opening up the possibility that there are more 
HHGFT operations than HCFT operations associated to $\Sigma$
(if \eqref{eq:comparison} fails to be surjective on homology)
and the prospect of obtaining vanishing results
for HCFT operations
(if \eqref{eq:comparison} fails to be injective on homology).
An example of the former situation is given in Remark~\ref{rk:calc}.

\subsubsection{Automorphisms of free groups}
Roughly speaking, an HCFT is a collection of operations parameterized
by the homology of mapping class groups.
The passage from an HCFT to an HHGFT adds brand new operations,
including ones parameterized by the homology of the automorphism groups
of free groups  with boundaries $A^s_{n,k}$ studied by Hatcher, Jensen and Wahl 
\cite{JensenWahl}, \cite{HatcherWahl}, \cite{WahlAutomorphisms}.
The family $A^s_{n,k}$ is highly interesting: 
for example,
the group $A^1_{n,0}$ agrees with the automorphism group $\Aut(F_n)$
of the free group on $n$ generators; 
the group $A^2_{n,0}$ is isomorphic to the holomorph 
$\Hol(F_n) = F_n\rtimes \Aut(F_n)$ of the free group on $n$ generators;
and the group $A^1_{0,k}$ is a central extension by $\Z^k$
of the pure symmetric automorphism group of the free group 
on $k$ generators \cite{JensenWahl}.
To illustrate a pattern for constructing elements in the homology groups
of $A^s_{n,k}$ (and other homology groups that parameterize HHGFT
operations), we will explain in section~\ref{sec:higher-operations}
how Theorem~\ref{thm:calc} leads to the  construction of  non-trivial
elements in the homology of  $\Hol(F_n)$.

To see why the groups $A^s_{n,k}$ appear in HHGFTs, recall that
$A^s_{n,k}$ is defined to be the set of path components of the space of
homotopy automorphisms (fixing $s$ distinguished points and
$k$ distinguished circles) of an appropriate graph $S$ \cite{HatcherWahl}.
Assuming that $s+k>0$, we can turn $S$ into a positive h-graph cobordism
by suitably dividing the $s$ points and the $k$ circles into
incoming and outgoing ones. Then $A^s_{n,k} = \pi_0\hAut(S)$.
The components of $\hAut(S)$ are contractible
(see Proposition~\ref{pr:cc}), so that $B\hAut^w(S) \simeq BA^s_{n,k}$,
and consequently \eqref{UniversalOperation:eq}
gives us operations parameterized by the homology of $A^s_{n,k}$.

\subsection{Remarks on the construction}
The HHGFT operations \eqref{HHGFTop:eq} in 
our theory arise from a push-pull construction that combines
ordinary induced maps with `umkehr' or `wrong-way' maps. 
This is a common method for constructing field theories,
especially in string topology, and in particular it is the method
used by both Chataur and Menichi~\cite{ChataurMenichi}
and Godin~\cite{Godin}.

For the sake of simplicity, let us consider the case of 
a single positive h-graph cobordism
$S \colon X \hto Y$, regarded as a family parameterized by the
one-point space $\pt$, and let us assume
that $G$ is  positive dimensional.  Write 
$r_\mathrm{in}\colon BG^{S} \to BG^{X}$ and 
$r_\mathrm{out}\colon BG^{S} \to BG^{Y}$
for the two restriction maps. Then the operation 
that our field theory associates to the family 
$S/\pt$ is given by the composite
\[
	H_\ast(BG^{X})
	\xrightarrow{\ (r_\mathrm{in})^!\ }
	H_{\ast-\dim(G)\cdot\chi(S,X)}(BG^{S})
	\xrightarrow{\ (r_\mathrm{out})_\ast\ }
	H_{\ast-\dim(G)\cdot\chi(S,X)}(BG^{Y})
\]
where $(r_\mathrm{out})_\ast$ is a standard induced map
and $(r_\mathrm{in})^!$ is an umkehr map.
If $S$ is one of the ordinary cobordisms
considered by Chataur and Menichi, then the operation
that their theory associates to the generator of
$H_0(B\Diff(S))$ arises as a composite of the same form.  
In both cases, the crucial step is the construction of the umkehr map.

Chataur and Menichi obtain their umkehr maps from the Serre spectral 
sequence. This is made possible by the restrictions
they place on the cobordism $S$, which are stringent enough to 
ensure that the fibres of the map $r_\mathrm{in}$ are
\emph{small} in the sense that they are homotopy equivalent
to closed manifolds.  In our case, where $S$ is a general positive
h-graph cobordism, the fibres of $r_\mathrm{in}$ need not satisfy
any such smallness condition, and so
a different construction of the umkehr maps is necessary.

The key idea in our construction of 
the umkehr maps is the realization that even though
the fibres of $r_\mathrm{in}$ for a positive h-graph 
cobordism $S$ need not be small in general,
by choosing a finite set $P \subset X$ meeting every 
component of $X$, we obtain a commutative diagram
\[\xymatrix{
	BG^S 
	\ar[dr]
	\ar[rr]^{r_\mathrm{in}}
	&&
	BG^X
	\ar[dl]
	\\
	&
	BG^P
}\]
where the fibres of both maps into $BG^P$ are small.
Indeed, the spaces $BG^S$ and $BG^X$ turn out to be 
fibrewise homotopy equivalent to 
fibrewise manifolds over $BG^P$, and the map $r_\mathrm{in}$
corresponds to a fibrewise smooth map. 
This makes it possible for us to 
obtain the umkehr map $(r_\mathrm{in})^!$ from  
a fibrewise Pontryagin--Thom construction.

The approach to constructing the umkehr maps outlined above
is complicated by the fact that it relies on the choice of the set $P$. 
Proving that the resulting umkehr map is independent
of this choice and satisfies all the properties
necessary for proving the HHGFT axioms 
turns out to be a surprisingly complicated task, and
meticulous organization is required to prevent
the argument from collapsing under the weight of 
a plethora of details.

\subsection{Relation to the work of Chataur--Menichi and Godin}
\label{subsec:CM-Godin}

A homological h-graph field theory $\Phi$
gives rise to a homological conformal field theory $\phi$ as follows.
A compact 1-manifold $X$ is in particular an h-graph, 
so we may define $\phi_\ast(X)=\Phi_\ast(X)$.
Moreover, an open-closed cobordism $\Sigma$ from $X$ to $Y$
determines a family of h-graph
cobordisms $S/B\colon X\hto Y$ with $B=B\Diff(\Sigma)$ and
$S=E\Diff(\Sigma)\times_{\Diff(\Sigma)}\Sigma$.
We may therefore define the HCFT operation
\begin{equation}
\label{eq:hcft-op}
	H_{\ast-d\chi(\Sigma,X)}(B\Diff(\Sigma))\otimes\phi_\ast(X)
	\longto
	\phi_{\ast}(Y)
\end{equation}
to be the HHGFT operation \eqref{HHGFTop:eq}
associated to $S/B$.
Let $\phi^G$ denote the HCFT obtained from the HHGFT
$\Phi^G$ of Theorem~\ref{thm:main} in this way.

We would now like to compare the HCFT $\phi^G$ 
with the HCFTs discussed earlier in the introduction.
The three HCFTs in question are:
\begin{enumerate}
	\item[(a)]
	The Godin HCFT of \cite{Godin}, which we denote $\phi^\Godin$.
	\item[(b)]
	The Chataur--Menichi HCFT of \cite{ChataurMenichi},
	which we denote $\phi^\CM$.
	\item[(c)]
	The HCFT $\phi^G$ obtained using Theorem~\ref{thm:main}.
\end{enumerate}
The input data from which these theories are constructed are:
\begin{enumerate}
	\item[(a)]
	A closed oriented manifold $M$.
	\item[(b)]
	A compact Lie group $G$, either finite or connected.
	\item[(c)]
	A compact Lie group $G$.
\end{enumerate}
The basic values of the theories are:
\begin{itemize}
	\item[(a)]
	$\phi_\ast^\Godin(S^1)=H_\ast(LM)$ and $\phi_\ast^\Godin(I)=H_\ast(M)$.
	\item[(b)]
	$\phi_\ast^\CM(S^1)=H_\ast(LBG)$.
	\item[(c)]
	$\phi_\ast^G(S^1)=H_\ast(LBG)$
	and $\phi_\ast^G(I)=H_\ast(BG)$.
\end{itemize}
The theories are
\begin{enumerate}
	\item[(a)]
	open-closed,
	\item[(b)]
	closed,
	\item[(c)]
	open-closed
\end{enumerate}
respectively.
Here an \emph{open-closed} theory is one in which 1-manifolds 
are compact but may have boundary, and cobordisms may have corners.
In this case the \emph{free} boundary of a cobordism
consists of those points of the boundary that are neither incoming nor 
outgoing.
A \emph{closed} theory is one in which the $1$-manifolds must be closed
and in which cobordisms have no free boundary.
See \cite[section~2.1.1]{MooreSegal}
or \cite[section~1.2]{Costello} for more details.
The three theories are subject to the following 
\emph{positive boundary} or \emph{noncompactness} conditions.
\begin{enumerate}
	\item[(a)]
		Each component of a cobordism  meets the 
		free or the outgoing boundary.
	\item[(b)]
		Each component of a cobordism meets the 
		incoming boundary (for $G$ finite); or
		each component of a cobordism meets both the incoming
		and the outgoing boundary (for $G$ connected).
	\item[(c)]
		Each component of a cobordism meets the incoming boundary.
\end{enumerate}

Observe that $\phi^G$ is defined for any 1-manifold or cobordism
for which $\phi^\CM$ is defined,
and so can be restricted to an HCFT of the same kind as $\phi^\CM$.
In section~\ref{sec:ChataurMenichi} we will show that 
this restriction of $\phi^G$ is exactly $\phi^\CM$.

In order to illustrate the differences between the three theories
$\phi^\Godin$, $\phi^\CM$ and $\phi^G$,
we consider the various  `unit'
and `counit' cobordisms admitted by the theories:
\medskip
\begin{center}
\begin{threeparttable}
\begin{tabular}{llcccc}
\noalign{\hrule height 0.08em\vspace{.65ex}}
\multicolumn{2}{c}{Theory}
&
\begin{tabular}{c}
\begin{tikzpicture}[scale=0.03,baseline=0]
	\clip (-23,-16) rectangle (12,16);
	\path[ARC, fill=gray!20] (0,0) ellipse (6 and 12);
	\path[ARC, fill=gray!20] (0,12) .. controls (-25,12) and (-25,-12) .. (0,-12);
	\begin{scope}
	\clip (-7,-12) rectangle (1,12);
	\path[ARC, densely dotted] (0,0) ellipse (6 and 12);	
	\end{scope}
\end{tikzpicture}
\\
$\eta_{S^1}$
\end{tabular}
&
\begin{tabular}{c}
\begin{tikzpicture}[scale=0.03,baseline=0]
	\clip (23,-16) rectangle (-12,16);
	\path[ARC, fill=gray!20] (0,12) .. controls (25,12) and (25,-12) .. (0,-12);
	\path[ARC, fill=gray!50] (0,0) ellipse (6 and 12);
\end{tikzpicture}
\\
$\varepsilon_{S^1}$
\end{tabular}
&
\begin{tabular}{c}
\begin{tikzpicture}[scale=0.03,baseline=0]
	\clip (-23,-16) rectangle (12,16);
	\path[ARC, fill=gray!20] (0,12) .. controls (-25,12) and (-25,-12) .. (0,-12) -- cycle;
\end{tikzpicture}
\\
$\eta_I$
\end{tabular}
&
\begin{tabular}{c}
\begin{tikzpicture}[scale=0.03,baseline=0]
	\clip (23,-16) rectangle (-12,16);
	\path[ARC, fill=gray!20] (0,12) .. controls (25,12) and (25,-12) .. (0,-12) -- cycle;
\end{tikzpicture}
\\
$\varepsilon_I$ 
\end{tabular}
\\
\noalign{\vspace{.4ex}\hrule height 0.05em\vspace{.65ex}}
(a)& $\phi^\Godin$ & \tickYes & \tickNo & \tickYes & \tickYes \\
(b)& $\phi^\CM$ & \tickNo & \tickYes\!/\tickNo\,\tnote{$\dagger$} & n/a & n/a \\
(c)& $\phi^G$ & \tickNo & \tickYes & \tickNo & \tickYes 
\\
\noalign{\vspace{.4ex}\hrule height 0.08em}
\end{tabular}
\begin{tablenotes}
{\footnotesize
\item[$\dagger$]\tickYes\ for $G$ finite, \tickNo\ for $G$ connected.}
\end{tablenotes}
\end{threeparttable}
\end{center}
\medskip
Here $\eta_{S^1}$, $\varepsilon_{S^1}$, $\eta_I$ and $\varepsilon_I$ 
denote the disk viewed as a cobordism $\emptyset \to S^1$, 
$S^1 \to \emptyset$, $\emptyset \to I$ and $I\to \emptyset$, respectively.

Observe that our theory $\phi^G$ admits the counit cobordism
$\varepsilon_{S^1}$ for all compact Lie groups $G$,
whereas the Chataur--Menichi theory $\phi^\CM$  does so 
only when $G$ is finite. In a preprint version~\cite{ChataurMenichiPreprint} of their 
article, Chataur and Menichi used an ad-hoc
method to construct a unit for the 
product on $H^\ast(LBG)$ in a cohomological version
of their theory even in the case of a connected $G$.
The existence of this unit in the cohomological case 
suggests that it should be possible to incorporate 
the counit cobordism $\varepsilon_{S^1}$ into the 
homological version of the theory even when $G$ is 
positive-dimensional.
Our results confirm this expectation.

It is natural to ask whether any of the `missing' units or counits
can be  incorporated into the HCFTs $\phi^G$ and $\phi^\Godin$.
It turns out that no such improvements are possible.
For if an HCFT $\phi$ admits $\eta_{S^1}$ and $\varepsilon_{S^1}$,
then $\phi_\ast(S^1)$ is finite-dimensional, but this is not in general
the case for $\phi^\Godin_\ast(S^1)=H_\ast(LM)$ or for
$\phi^G_\ast(S^1)=H_\ast(LBG)$.
Similarly, if an HCFT $\phi$ admits $\eta_{I}$ and $\varepsilon_{I}$,
then $\phi_\ast(I)$ is finite-dimensional.
Although $\phi^\Godin_\ast(I)=H_\ast(M)$ is finite-dimensional,
$\phi^G_\ast(I)=H_\ast(BG)$ is typically not.

In view of the above discussion, we regard the HCFT $\phi^G$
as the structure that is the closest analogue to Godin's HCFT 
possible in the context of string topology of classifying spaces.
It is natural to ask whether Godin's HCFT can be extended
to an HHGFT. We expect that to be the case. 
See Conjecture~\ref{conj:manifold-hhgft}.

\subsection{Other related work}
In~\ref{subsec:CM-Godin} above we discussed how the work in the present paper
is related to that of Chataur and Menichi~\cite{ChataurMenichi}
and Godin~\cite{Godin}.  Here we discuss two other pieces of related work.

Part of our extension of Chataur and Menichi's HCFT has already been obtained
by Guldberg in the master's thesis \cite{Guldberg}.
There, Guldberg constructs an extension of 
Chataur and Menichi's HCFT to an open-closed theory with boundary conditions.
The cobordisms in this theory are open-closed, 
they are subject to the condition that every component of every cobordism
meets the incoming and outgoing boundary,
and their free boundary components
are labelled by connected closed subgroups of the given
compact connected Lie group $G$.
By restricting to the case where every free boundary component is labelled by $G$
itself, one obtains an open-closed HCFT that extends the Chataur--Menichi theory,
and of which our HHGFT is an extension.

Behrend, Ginot, Noohi and Xu~\cite{BGNX}
have given a unified approach to string topology of manifolds and of classifying spaces.
Given an oriented Hurewicz stack $\mathfrak{X}$ of dimension $d$,
they construct an HCFT of degree $d$ whose value on $S^1$
is $H_\ast(L\mathfrak{X})$.  See \cite[Theorem~14.2]{BGNX}.
Taking $\mathfrak{X}$ to be a closed oriented manifold $M$ gives
$H_\ast(L\mathfrak{X})\isom H_\ast(LM)$ and $d=\dim(M)$,
while taking 
$\mathfrak{X}$ to be $[\pt/G]$ for a compact connected Lie group $G$
gives $H_\ast(L\mathfrak{X})\isom H_\ast(LBG)$ and $d=-\dim(G)$.
The HCFT constructed is, like Chataur and Menichi's, a closed theory
in which all components of a cobordism are assumed to meet the incoming
and outgoing boundary.
It seems likely that in the case $\mathfrak{X}=[\pt/G]$ the resulting HCFT
coincides with Chataur and Menichi's.

\subsection{Extension to arbitrary characteristic}
Theorem~\ref{thm:main} admits an extension
in which the ground field $\F$ has arbitrary characteristic.
The extended version states that if $G$ is a compact Lie group
whose adjoint action on its Lie algebra preserves orientations,
then the choice of an orientation for the Lie algebra of $G$
induces a positive HHGFT of degree $-\dim(G)$ whose value 
on an h-graph $X$ is $H_\ast(BG^X;\,\F)$.
See Remark~\ref{rem:arbitrary-characteristic} for a discussion
of (positive) HHGFTs over a ground field of arbitrary characteristic.
Observe that the `orientability' condition on $G$
holds automatically if $G$ is finite or connected.

The proof of the extended version of Theorem~\ref{thm:main} 
requires a significant amount of extra
detail that is in some sense orthogonal to the other aspects of the
construction, and which the decision to work in characteristic 2 
has allowed us to suppress. The extended version will be proved in 
a later paper.

\subsection{Organisation of the paper}
The paper is arranged as follows.

Section~\ref{sec:h-graphs} introduces h-graphs, h-graph cobordisms,
and families of such.
It relates them to open-closed cobordisms, and gives a classification of
families of h-graph cobordisms over sufficiently nice base spaces.
Then section~\ref{sec:hhgft} introduces homological h-graph field theories
and relates them to homological conformal field theories.
Section~\ref{overview:section} sketches the construction of the operation 
that our HHGFT associates to a family of positive h-graph cobordisms.

The next three sections of the paper are dedicated to the 
actual construction of the
HHGFT $\Phi^G$ of Theorem~\ref{thm:main}.
Our construction makes use of 
fibred categories, used for handling base-change issues,
and symmetric monoidal double categories, used to express our theory
of umkehr maps. We will recall these category-theoretical notions 
in section~\ref{categorical-background-section}.
Then section~\ref{sec:push-pull} formalises the push-pull construction
by explaining how a collection of umkehr maps for families of h-graphs,
expressed in terms of double categories, leads to an HHGFT.
Finally, in section~\ref{sec:umkehr-maps}
we construct the theory of umkehr maps necessary
to obtain our HHGFT.
This section is by far the longest section of the paper,
and it begins with a detailed outline of its own contents.

In section~\ref{sec:ChataurMenichi} we show that
the HHGFT operations in our theory $\Phi^G$ agree with 
the HCFT operations constructed by Chataur and Menichi
when the latter are defined.  And in section~\ref{sec:higher-operations}
we prove Theorem~\ref{thm:calc}, which demonstrates nontriviality
of our theory, and discuss an application to the homology of 
the holomorph $\Hol(F_n)$ of the free group on $n$ generators.

We conclude with two appendices.
Appendix~\ref{appendix:fibred-spaces} recalls some basic facts from the theory
of fibrewise topology, while
appendix~\ref{app:Fd-fibrations} gives the classification
of what we call $(F,\partial)$-fibrations.
This classification is used to 
obtain the classification of families of h-graph cobordisms
given in section~\ref{sec:h-graphs}.

\subsection{Notation and conventions}
\label{subsec:notation-conventions}
Here and for the rest of the paper, when working with topological spaces,
we will operate entirely within the category of $k$-spaces. 
See for example~\cite[section 1.1]{MaySigurdsson}. 
Furthermore, when working with fibred spaces,
we will follow May and Sigurdsson \cite{MaySigurdsson} and assume
that all our base spaces are weak Hausdorff in addition to being 
$k$-spaces. See \cite[section 1.3]{MaySigurdsson} for discussion.
The one-point space will be denoted $\pt$. We remind the 
reader that unless mentioned otherwise, 
we will work over a fixed field $\F$ of characteristic 2,
and homology is taken with coefficients in $\F$.
The category of graded $\F$-modules will be denoted $\grmod$.
For a topological group or monoid $H$, 
by $BH$ and $EH$ we always mean the bar constructions 
$B(\pt,H,\pt)$ and $B(\pt,H,H)$, respectively. 
See for example~\cite[section 7]{MayClassifying}.
Finally, by a symmetric monoidal functor $F\colon \calC \to \calC'$ 
between symmetric monoidal categories
we mean a \emph{strong} symmetric monoidal functor in the sense
of Mac~Lane~\cite[section XI.2]{MacLane}, meaning that 
the monoidality and unit constraints
\[
	F_\tensor \colon F(X)\tensor F(Y) \longto F(X\tensor Y)
	\quad\text{and}\quad
	F_I \colon I'\longto F(I)
\]
are required to be isomorphisms. Here $I$ and $I'$ denote 
the unit objects of $\calC$ and $\calC'$, respectively.
As usual, $F$ is called \emph{strict} if $F_\tensor$ and $F_I$
are identity maps.

\subsection{Acknowledgements}
We would like to thank Nathalie Wahl and Alexander Kupers 
for helpful comments and conversations.
Moreover, we gratefully acknowledge support
from the Danish National Research Foundation 
through the Centre for Symmetry and Deformation (DNRF92).

%% file: h-graphs-and-field-theories.tex

\section{H-graphs}
\label{sec:h-graphs}
The present section will introduce h-graphs and h-graph cobordisms
in detail.
It will also introduce families of h-graphs and h-graph cobordisms,
and will show that open-closed cobordisms and their universal
families fit into this new framework.

\subsection{H-graphs and h-graph cobordisms}
This subsection will introduce h-graphs and h-graph cobordisms,
and will show that open-closed cobordisms are examples of such.

\begin{definition}
An \emph{h-graph} is a space with the homotopy type of a finite CW-complex
of dimension at most $1$.
\end{definition}

In other words an h-graph is a space with the homotopy type of a finite graph.
In particular a compact $1$-manifold is an h-graph, as is a compact surface
whose components all have nonempty boundary.

\begin{definition}
A continuous map $f\colon X\to Y$ between h-graphs is called 
\emph{positive} if $\pi_0f\colon\pi_0X\to\pi_0Y$ is surjective.
It is called an \emph{h-embedding} if there is a homotopy cofibre square
\begin{equation}\label{eq:hembedding}
	\vcenter{\xymatrix{
		A\ar[r]\ar[d] & X\ar[d]^f\\
		B\ar[r] & Y
	}}
\end{equation}
in which $B$ is an h-graph and $A$ has the homotopy type of a finite set.
Observe that the existence of such a homotopy cofibre square
implies that $Y$ itself is an h-graph.
Observe also that homotopy equivalences between h-graphs are h-embeddings.
\end{definition}

Recall that a \emph{homotopy cofibre square} is a commutative square for which the
map from the double mapping cylinder on the top-left part 
to the lower-right term is a homotopy equivalence.
So in rough terms a map $f\colon X\to Y$ is an h-embedding if $Y$
is obtained up to homotopy by adding points and arcs to $X$.

The notion of homotopy cofibre square makes sense in any
left-proper model category.
For general definitions and results
we refer the reader to \cite[sections~13.5 and~13.3]{Hirschhorn}.
For our applications, the appropriate model structure is the
Str\o m model structure on $k$-spaces in which
fibrations are Hurewicz fibrations, cofibrations are closed cofibrations,
and weak equivalences are homotopy equivalences.
See \cite[section~4.4]{MaySigurdsson}.

\begin{lemma}\label{lm:pushout}
Suppose given a pushout square
\[\xymatrix{
	X\ar[d]_{f}\ar[r]^{g} & Y\ar[d]^{h}\\
	Z\ar[r]_{k} & W
}\]
in which $X$, $Y$ and $Z$ are h-graphs
and $f$ is both an h-embedding and a closed cofibration.
Then $W$ is an h-graph and $h$ is an h-embedding.
If $g$ is positive then so is $k$.
\end{lemma}

\begin{proof}
Choose a homotopy cofibre square \eqref{eq:hembedding}
that witnesses $f$ as an h-embedding.
By pasting it with the square in the statement one obtains a second homotopy cofibre
square of the form \eqref{eq:hembedding} in which the right-hand map is $h$.
It follows that $W$ is an h-graph and $h$ is an h-embedding.
The final claim is immediate.
\end{proof}

\begin{lemma}\label{lm:composites}
A composite of positive maps is positive
and a composite of h-embeddings is an h-embedding.
\end{lemma}

\begin{proof}
The first part is immediate.
To prove the second part, let $X\to Y$ and $Y\to Z$ be h-embeddings, and choose the two homotopy cofibre diagrams
on the left
\[
	\xymatrix{
		S\ar[d]\ar[r] & X\ar[d]\\
		T\ar[r] & Y
	}
	\qquad\qquad
	\xymatrix{
		U\ar[d]\ar[r] & Y\ar[d]\\
		V\ar[r] & Z
	}
	\qquad\qquad
	\xymatrix{
		S\ar[d]\ar[r] & X\ar[d]\\
		W\ar[r] & Z
	}
\]
in which $T$ and $V$ are h-graphs and $S$ and $U$ have the homotopy
type of finite sets.
By modifying these squares we may assume, in turn,
that $T\to Y$ is surjective on $\pi_0(-)$,
that $U\to Y$ factors through a map $U\to T$,
and that this last map is a cofibration.
Writing $W$ for the pushout of $V\leftarrow U\rightarrow T$,
we obtain the homotopy cofibre square on the right,
in which $W$ is an h-graph.
Thus $X\to Z$ is an h-embedding, as required.
\end{proof}

\begin{definition}
Let $X$ and $Y$ be h-graphs.
An \emph{h-graph cobordism} $S$ from $X$ to $Y$,
written $S\colon X\hto Y$, consists of an h-graph $S$ and a zig-zag
\[
	X\xrightarrow{\ \ i\ \ } S\xleftarrow{\ \ j\ \ } Y
\]
such that the map $j$ is an h-embedding and the map
$(i,j) \colon X\sqcup Y \to S$ is a closed cofibration. 
It is called \emph{positive} if $i$ is positive.

Let $S\colon X\hto Y$ and $T\colon Y\hto Z$ be h-graph cobordisms.
The \emph{composite} h-graph cobordism
\[T\circ S\colon X\longhto Z\]
is given by $T\circ S=T\cup_Y S$ and the zig-zag
\[
	X\rightarrow S\rightarrow T\cup_Y S \leftarrow T\leftarrow Z.
\]
This is indeed an h-graph cobordism, for Lemmas~\ref{lm:pushout}
and~\ref{lm:composites} guarantee that $T\cup_Y S$ is an h-graph
and that the right-hand map in the zig-zag is an h-embedding.
If $S$ and $T$ are positive then so is $T\circ S$.

Let $S_1\colon X_1\hto Y_1$ and $S_2\colon X_2\hto Y_2$
be h-graph cobordisms.
The \emph{disjoint union} h-graph cobordism
\[S_1\sqcup S_2\colon X_1\sqcup X_2\longhto Y_1\sqcup Y_2\]
is obtained in the evident way.
If $S_1$ and $S_2$ are positive then so is $S_1\sqcup S_2$.
\end{definition}

\begin{example}\label{ex:occobordism}
Let $\Sigma$ be an open-closed cobordism 
in which every component meets
$\partial_\rmin\Sigma\cup\partial_\rmfree\Sigma$. 
See Costello \cite[\S1.2]{Costello}.
Then the inclusions of the incoming and outgoing boundaries of $\Sigma$
determine an h-graph cobordism
\[\Sigma\colon \partial_\rmin\Sigma\longhto\partial_\rmout\Sigma.\]
If furthermore every component of $\Sigma$ meets
$\partial_\rmin\Sigma$, then the h-graph cobordism is positive.

To see this, observe that
$\partial_\rmin\Sigma$, $\partial_\rmout\Sigma$
and $\Sigma$ are certainly h-graphs,
the latter because every component of $\Sigma$ has nonempty boundary;
the two maps 
$\partial_\rmin\Sigma \to \Sigma$ and 
$\partial_\rmout\Sigma \to \Sigma$ 
are closed cofibrations;
and the inclusion of $\partial_\rmout\Sigma$ is an h-embedding,
because every component of $\Sigma$ has boundary that is not outgoing.

\end{example}

H-graph cobordisms, even when their source and target are compact 1-manifolds,
can be significantly more general than open-closed cobordisms.
The reader who is interested to see examples of this could briefly skip
forward to subsection~\ref{sbs:compare}.

\subsection{Families of h-graphs and h-graph cobordisms}
\label{hgraphfamilies:subsection}
Now we will introduce {families} of h-graphs and h-graph cobordisms.
This will take place in the setting of fibred spaces,
for which the relevant notions are recalled in appendix~\ref{appendix:fibred-spaces}.
It may suffice to know that if $B$ is a chosen base space,
then a space \emph{fibred over} $B$ is simply a map $X\to B$.
Recall that we assume
that all base spaces are weak Hausdorff in addition to being $k$-spaces.

\begin{definition}
\label{def:families-of-h-graphs}
A \emph{family of h-graphs} $X$ over a base space $B$
is a Hurewicz fibration $X\to B$ whose fibres are h-graphs.
By a map from a family of h-graphs $X$ over $B$
to a second family of h-graphs $Y$ over $C$, 
we mean a pair $(f,g)$ of continuous maps making the diagram
\[\xymatrix{
	X
	\ar[r]^f
	\ar[d]
	&
	Y
	\ar[d]
	\\
	B
	\ar[r]^g
	&
	C
}\]
commutative.
Such a map is called \emph{positive}
(respectively, an \emph{h-embedding})
if for each $b\in B$, the induced map $X_b \to Y_{g(b)}$
between fibres is positive (respectively, an h-embedding).
\end{definition}

\begin{definition}\label{def:famcob}
Let $X$ and $Y$ be h-graphs and let $B$ be a base space.
A \emph{family of h-graph cobordisms} $S$ over $B$ from $X$ to $Y$,
written $S/B\colon X\hto Y$, 
consists of a family of h-graphs $S$ over $B$ and a zig-zag of maps over $B$
\[
	\pi_B^\ast X\xrightarrow{\ \ i\ \ } S\xleftarrow{\ \ j\ \ }\pi_B^\ast Y
\]
such that $j$ is an h-embedding and such that the map
$(X\sqcup Y) \times B \to S$ induced by $i$ and $j$ 
is a closed fibrewise cofibration.
Here $\pi_B\colon B\to\pt$ is the constant map and 
$X$ and $Y$ are regarded as families of h-graphs over $\pt$.
The family of cobordisms is called \emph{positive} if $i$ is positive.
We regard a family of h-graph cobordisms $S$ as a collection of
h-graph cobordisms $S_b\colon X\hto Y$, varying continuously with $b\in B$.

Let $S/B\colon X\hto Y$ and $T/C\colon Y\hto Z$ be families of h-graph
cobordisms.  The \emph{external composite} family of h-graph cobordisms
\[(T\ucirc S)/(C\times B)\colon X\xhto{\quad} Z\]
is the family of h-graph cobordisms
whose fibre over $(c,b)$ is the composite $T_c\circ S_b\colon X\hto Z$.
More precisely, it is given by the zig-zag
\[
	\pi_{C\times B}^\ast X
	\rightarrow
	\pi_{C}^\ast S
	\rightarrow
	\pi_{C}^\ast S\cup_{\pi_{C\times B}^\ast Y}\pi_B^\ast T
	\leftarrow
	\pi_B^\ast T
	\leftarrow
	\pi_{C\times B}^\ast Z
\]
in which $\pi_{C\times B}$, $\pi_C$ and $\pi_B$ are the projections from
$C\times B$ to $\pt$, $B$ and $C$ respectively.
By \cite[Proposition~1.3]{Clapp} the map
$\pi_{C}^\ast S\cup_{\pi_{C\times B}^\ast Y}\pi_B^\ast T \to C\times B$
is a fibration,
and now it follows from Lemmas~\ref{lm:pushout} and~\ref{lm:composites}
that $(T\ucirc S)/(C\times B)$ is indeed a family of h-graph cobordisms.

Let $S_1/B_1\colon X_1\hto Y_1$ and $S_2/B_2\colon X_2\hto Y_2$
be families of h-graph cobordisms.
The \emph{external disjoint union} family of h-graph cobordisms
\[
	(S_1\usqcup S_2)/(B_1\times B_2)
	\colon
	X_1\sqcup X_2
	\longhto Y_1\sqcup Y_2
\]
is the family whose fibre over $(b_1,b_2)$
is the disjoint union $(S_1)_{b_1}\sqcup (S_2)_{b_2}$.
The task of formulating $S_1\usqcup S_2$ as a space fibred 
over $B_1\times B_2$ is left to the reader.
\end{definition}

\begin{example}
\label{ex:universal-family}
Let $\Sigma$ be an open-closed cobordism,
so that we have an h-graph cobordism 
$\Sigma\colon\partial_\rmin\Sigma \hto \partial_\rmout\Sigma$
as in Example~\ref{ex:occobordism}.
Let $\Diff(\Sigma)$ denote the topological group of diffeomorphisms
of $\Sigma$ that fix the incoming and outgoing boundaries pointwise.
The universal bundle $\pi\colon U\Diff(\Sigma)\to B\Diff(\Sigma)$
with 
\[
	U\Diff(\Sigma) = E\Diff(\Sigma) \times_{\Diff(\Sigma)} \Sigma
\]
determines a \emph{family} of h-graph cobordisms
from $\partial_\rmin\Sigma$ to $\partial_\rmout\Sigma$:
\[
	B\Diff(\Sigma)\times \partial_\rmin\Sigma
	\longrightarrow
	U\Diff(\Sigma)
	\longleftarrow
	B\Diff(\Sigma) \times \partial_\rmout\Sigma.
\]
We must verify that $\pi$ is a fibration and that the maps in the zig-zag
are closed fibrewise cofibrations.
By \cite[Theorem~8.2]{MayClassifying} $\pi$ is a numerable fibre bundle,
so that the properties hold over the elements of a numerable cover,
and now by Proposition~\ref{pr:locality} they hold globally.
\end{example}

A crucial aspect of our theory is that we are able to
adjust families of h-graph cobordisms: fibres can be replaced up to
homotopy equivalence, and by taking pullbacks of families we can
change the base space.
The following definition makes this precise.

\begin{definition}\label{def:two-cell}
Let $S/B\colon X\hto Y$ and $S'/B'\colon X'\hto Y'$ be families of
h-graph cobordisms.
A \emph{2-cell} $\varphi\colon S/B\Rightarrow S'/B'$,
written as the square on the left,
\[\vcenter{\xymatrix{
	X
	\ar|-@{|}[r]^{S/B}_{\phantom{.}}="A"
	\ar[d]
	&
	Y
	\ar[d]
	\\
	X'
	\ar|-@{|}[r]_{S'/B'}^{\phantom{.}}="B"
	&
	Y'
	\ar@{=>}"A";"B"^\varphi
}}
\qquad\qquad\qquad
\begin{array}{l}
	\varphi_X\colon X\to X'
	\\
	\varphi_Y\colon Y\to Y'
	\\
	\varphi_B\colon B\to B'
	\\
	\varphi_S\colon S\to S'
\end{array}\]
consists of four maps as on the right,
compatible in the sense that $\varphi_S$
lies under $(\varphi_X\sqcup \varphi_Y)\times\varphi_B$
and over $\varphi_B$,
and subject to the conditions that
$\varphi_X$ and $\varphi_Y$ be homotopy equivalences
and that $\varphi_S$ be a homotopy equivalence in each fibre.
\end{definition}

\subsection{Homotopy automorphisms and universal families}
\label{subsec:universal-families}

Now we will give a classification of families of h-graph cobordisms,
at least when the base has the homotopy type of a CW-complex.
The results here are specializations of those in
appendix~\ref{app:Fd-fibrations}.
The present section gives the definitions and the statements of results,
referring to the appendix for details.
Our classification theorem will be given in terms of the following monoid
of homotopy automorphisms, which can be regarded as the h-graph
analogue of the group $\Diff(\Sigma)$.

\begin{definition}
Let $S_0\colon X\hto Y$ be an h-graph cobordism.
The \emph{homotopy automorphism monoid} of $S_0$,
written $\hAut(S_0)$, is the topological monoid of homotopy equivalences
$f\colon S_0\to S_0$ that fix $X\sqcup Y$ pointwise.
\end{definition}

The monoid $\hAut(S_0)$ has contractible components so long as
every component of $S_0$ meets $X\sqcup Y$;
in particular this holds if $S_0$ is positive.
See Proposition~\ref{pr:cc}.

\begin{example}[Diffeomorphisms of surfaces]\label{ex:DNB}
	If $X$ and $Y$ are $1$-manifolds and $\Sigma$
	is an open-closed cobordism from $X$ to $Y$
	then there is an evident homomorphism
	$\Diff(\Sigma)\to \hAut(\Sigma)$.
	If $X$ and $Y$ are closed and $\Sigma$ has no free boundary
	and every component of $\Sigma$ meets $X$ or $Y$,
	then this homomorphism is in fact a homotopy equivalence.
	This result should be classical but, lacking a specific reference,
	we sketch a proof in the next paragraph.
	
	The components of $\Diff(\Sigma)$ are contractible by
	\cite[Theorem~1D]{EarleSchatz}, and the components of
	$\hAut(\Sigma)$ are contractible by Proposition~\ref{pr:cc}.
	It therefore remains to show that the induced map
	$\pi_0(\Diff(\Sigma))\to\pi_0(\hAut(\Sigma))$ is an isomorphism.
	To prove this we decompose the boundary of $\Sigma$
	as $\partial_0\sqcup\partial_1$,
	where $\partial_0$ has exactly one component,
	and we fix a basepoint $p_0\in\partial_0$.
	We denote by $\hAut'(\Sigma)$ the group of homotopy classes of
	homotopy automorphisms of $\Sigma$ that fix $\partial_1\sqcup \{p_0\}$
	pointwise and that fix the based homotopy class of the inclusion
	$\partial_0\hookrightarrow\Sigma$.
	A routine argument,
	based on the fact that $\Sigma$ has the form $K(\pi,1)$,
	shows that the natural map $\pi_0(\hAut(\Sigma))\to \pi_0(\hAut'(\Sigma))$
	is an isomorphism, so it will suffice to show that the composite
	\[
		\pi_0(\Diff(\Sigma)\longto \pi_0(\hAut(\Sigma))\longto \pi_0(\hAut'(\Sigma))
	\]
	is an isomorphism.
	This is proved in \cite[section~2]{JensenWahl}
	(the relevant map is the central isomorphism appearing
	on page~548).
\end{example}

\begin{example}[Automorphisms of free groups with boundary]
\label{ex:free-groups-with-boundary}
Let $n$, $k$ and $s$ be non-negative integers and consider the following
graph.
\[\begin{tikzpicture}[scale=0.03]
	\path[ARC, fill=black] (0,0) circle (3);
	
	\path[ARC] (30,60) circle (10) node (A) at +(225:10) {};
	\path[ARC] (0,0) .. controls (10,10) and (20,30) .. (A.center);
	\path[ARC] (70,15) circle (10);
	\path[ARC] (0,0) .. controls (20,0) and (40,15) .. (60,15);
	\path (30,60) -- (70,15) node [black,midway,rotate=-50]  {$\cdots$};
	\draw
	[
		decorate,
		decoration={brace,amplitude=10pt},
		xshift=0pt, yshift=0pt, line width=0.3 mm
	]
	(30,78) -- (88,12) node [black,midway,xshift=15, yshift=15]  {$k$};	

	\path[ARC] (0,0) .. controls (-10,80) and (-70,60) .. (0,0);
	\path[ARC] (0,0) .. controls (-70,-40) and (-100,20) .. (0,0);
	\draw
	[
		decorate,
		decoration={brace,amplitude=10pt},
		xshift=0pt, yshift=0pt, line width=0.3 mm
	]
	(-70,-10) -- (-30,60) node [black,midway,xshift=-15, yshift=10]  {$n$};
	\path (-25,45) -- (-55,-5) node [black,midway,rotate=55]  {$\cdots$};

	\path[ARC] (0,0) -- (0,-50) [fill=black] circle (3);
	\path[ARC] (0,0) -- (45,-20) [fill=black] circle (3);
	\draw
	[
		decorate,
		decoration={brace,amplitude=10pt,mirror},
		xshift=0pt, yshift=0pt, line width=0.3 mm,
	]
	(-5,-60) -- (55,-25) node [black,midway,xshift=10, yshift=-15]  {$s$};
	\path (0,-50) -- (45,-20) node [black,midway,rotate=42]  {$\cdots$};
\end{tikzpicture}\]
Denote by $G_{n,k}^s$ the monoid of homotopy equivalences of the graph
that fix the $k$ circles and the $s$ leaves pointwise,
and define $A_{n,k}^s=\pi_0(G_{n,k}^s)$.
These groups were introduced by Hatcher and Wahl in~\cite{HatcherWahl}.
They encompass the groups $\mathrm{Out}(F_n)$, $\mathrm{Aut}(F_n)$,
and Wahl's automorphisms of free groups with boundary \cite{WahlAutomorphisms}.
By partitioning the $s$ points and $k$ circles into incoming and outgoing parts,
the graph above becomes an h-graph cobordism $S_0$ with
$\hAut(S_0)=G_{n,k}^s$. Moreover, $\hAut(S_0)$ is homotopy equivalent
to $A_{n,k}^s$ as long as $s>0$ or $k>0$. See Proposition~\ref{pr:cc}.
\end{example}

\begin{definition}
Let $S_0\colon X\hto Y$ be an h-graph cobordism.
	An \emph{$S_0$-family} of h-graph cobordisms is
	a family $S/B\colon X\hto Y$ such that for all $b\in B$ there is
	a map $S_0\to S_b$ under $X\sqcup Y$ which is a homotopy 
	equivalence.
\end{definition}

\begin{example}
 Suppose $S/B$ is a family of h-graph cobordisms over a
 path-connected base space $B$, and suppose $b\in B$.
 Then $S$ is an $S_b$-family.
\end{example}

\begin{definition}
	Two $S_0$-families $S/B$ and $S'/B$ over a base space $B$
	are \emph{equivalent}
	if they can be related by a zig-zag of 2-cells $\varphi$,
	all such that $\varphi_X$, $\varphi_Y$ and $\varphi_B$ are the identity maps.
	Let $S_0\Fam(B)$ denote the collection of equivalence classes of $S_0$-families
	over $B$.
\end{definition}

\begin{remark}
	If $B$ is sufficiently nice (for example, a CW complex),
	then $S/B$ and $S'/B$ are equivalent
	if and only if there is a 2-cell $S\Rightarrow S'$.
	See Remark \ref{rk:fd-equiv-rel}.
\end{remark}

Let $\hAut^w(S_0)$ denote the \emph{whiskered monoid}
$\hAut(S_0)\cup [0,1]$ in which $0\in [0,1]$ is identified with 
$\id_{S_0}\in \hAut(S_0)$.
Our classification theorem identifies $S_0\Fam(B)$
in terms of homotopy classes of maps into the classifying space
$B\hAut^w(S_0)$.
It is an instance of Theorem~\ref{thm:fd-classification}.

\begin{theorem}\label{th:universal}
Let $S_0\colon X\hto Y$ be an h-graph cobordism.
There is a \emph{universal $S_0$-family $U\hAut^w(S_0)\to B\hAut^w(S_0)$}
with the property that for any base space $B$ homotopy equivalent to a CW-complex
the map
\[
	[B,{B}\hAut^w(S_0)]\longrightarrow S_0\Fam(B),
	\qquad
	f\longmapsto f^\ast(U\/\hAut^w(S_0))
\]
is a bijection. \qed
\end{theorem}

The whiskered monoid $\hAut^w(S_0)$ is used here for technical reasons,
and can be replaced with $\hAut(S_0)$ itself when the identity element
of the latter is strongly nondegenerate. See Remark~\ref{rk:whisker}.
In any case, if $X\sqcup Y$ meets every component of $S_0$ then
$\hAut(S_0)$ has contractible components, so that $\hAut^w(S_0)\to \pi_0(\hAut(S_0))$
is a homotopy equivalence between monoids 
with strongly nondegenerate basepoints,
and consequently $B\hAut^w(S_0)\to B\pi_0(\hAut(S_0))$ is a homotopy equivalence.

\begin{corollary}\label{cor:universal}
Let $S/B\colon X\hto Y$ be an $S_0$-family of h-graph cobordisms
over a CW complex.
Then there is a 2-cell 
$\varphi\colon S/B\Rightarrow U\hAut^w(S_0) / B\hAut^w(S_0)$
with $\varphi_X$ and $\varphi_Y$ the identity maps.
The map $\varphi_B$ in such a 2-cell is 
uniquely determined up to homotopy. \qed
\end{corollary}

\section{Homological h-graph field theories}
\label{sec:hhgft}
Now that we have defined h-graphs, h-graph cobordisms and families of these,
we are able to define homological h-graph field theories
(or HHGFTs).
Broadly speaking these are analogous to homological conformal field theories
(or HCFTs), but with 1-manifolds and open-closed
cobordisms replaced by h-graphs and h-graph cobordisms.
The precise definition is given in subsection~\ref{sbs:hhgft} below.
Then in subsection~\ref{sbs:HCFT} we recall the notion of HCFT and explain
how every HHGFT restricts to an HCFT.
There are diverse implications for an HCFT when it arises from
an HHGFT in this way,
and in subsection~\ref{sbs:compare} we explore some of the phenomena
that arise.
The section ends in subsection~\ref{sbs:hgraphs-stringtop} with some comments
on the place of h-graphs in string topology.
Recall that we are working over a field $\bbF$ of characteristic $2$ throughout.

\subsection{Homological h-graph field theories}\label{sbs:hhgft}

\begin{definition}
Given a space $B$ and an integer $m$, we will write $H_{\ast+m}(B)$
for the graded $\bbF$-vector space given in degree $q$ by $H_{q+m}(B)$.
We will also use the evident generalization to the case where
$m$ is a locally constant function $m\colon B\to\bbZ$.
\end{definition}

\begin{definition}
Let $S/B\colon X\hto Y$ be a family of h-graph cobordisms.
Then $\chi(S,X)$ denotes the locally constant function on $B$
whose value at $b$ is the relative Euler characteristic $\chi(S_b,X)$.
These functions respect (external) composition and disjoint union of families
of h-graph cobordisms, so
that in the situation of Definition~\ref{def:famcob} we have
\begin{align*}
	\chi(T\ucirc S,X)&=\chi(T,Y)+\chi(S,X),\\
	\chi(S_1\usqcup S_2,X_1\sqcup X_2)&=\chi(S_1,X_1)+\chi(S_2,X_2).
\end{align*}
(The right-hand-sides are to be interpreted as functions on $C\times B$
and $B_1\times B_2$ in the evident way.)
\end{definition}

\begin{definition}
\label{def:HHGFT}
Fix an integer $d$.
A \emph{(positive) degree $d$ homological h-graph field theory}
$\Phi$ over a field $\bbF$ of characteristic 2 consists of
the following data:
\begin{itemize}
	\item A symmetric monoidal functor $\Phi_\ast$
 	from the category of h-graphs and homotopy equivalences among them
	into the category of graded $\bbF$-vector spaces.
	Here the monoidal structures are disjoint union and tensor product, respectively.
	Recall from section~\ref{subsec:notation-conventions} that 
	we require symmetric monoidal functors to be strong
	in the sense of Mac~Lane~\cite[section XI.2]{MacLane}.
	
	\item For each (positive) family $S/B\colon X\hto Y$ of h-graph cobordisms,
	a map
	\[
		\Phi(S/B)\colon
		H_{\ast-d\cdot\chi(S,X)}(B)\otimes\Phi_\ast(X)
		\longrightarrow
		\Phi_\ast(Y).
	\]
	Here homology is taken with $\bbF$ coefficients.
\end{itemize}
These data are required to satisfy the following
\emph{base change},
\emph{gluing},
\emph{identity},
and \emph{monoidality} axioms.
\end{definition}

\begin{BaseChange}
Given a 2-cell (see Definition~\ref{def:two-cell}) as on the left,
\[\xymatrix{
	X
	\ar|-@{|}[r]^{S/B}_{\phantom{.}}="A"
	\ar[d]
	&
	Y
	\ar[d]
	\\
	X'
	\ar|-@{|}[r]_{S'/B'}^{\phantom{.}}="B"
	&
	Y'
	\ar@{=>}"A";"B"^\varphi
}
\qquad\qquad\qquad
\xymatrix@C=50 pt{
	H_{\ast-d\cdot\chi(S,X)}(B)\otimes\Phi_\ast(X)
	\ar[r]^-{\Phi(S/B)}
	\ar[d]_{(\varphi_B)_\ast\otimes\Phi_\ast(\varphi_X)}
	&
	\Phi_\ast(Y)
	\ar[d]^{\Phi_\ast(\varphi_Y)}
	\\
	H_{\ast-d\cdot\chi(S',X')}(B')\otimes\Phi_\ast(X')
	\ar[r]_-{\Phi(S'/B')}
	&
	\Phi_\ast(Y')
}\]
the resulting diagram on the right commutes.
\end{BaseChange}

\begin{Gluing}
Given 
\refstepcounter{dummy}\label{gluing}
families of h-graph cobordisms
$X\xhto{S/B} Y\xhto{T/C}Z$,
the diagram
\[\xymatrix@C=60 pt@R=45pt{
	H_{\ast-d\cdot\chi(T,Y)}(C)\otimes H_{\ast-d\cdot\chi(S,X)}(B)\otimes\Phi_\ast(X)
	\ar[r]^-{1\otimes\Phi(S/B)}
	\ar[d]_{\times\otimes 1}
	&
	H_{\ast-d\cdot\chi(T,Y)}(C)\otimes\Phi_\ast(Y)
	\ar[d]^{\Phi(T/C)}
	\\
	H_{\ast-d\cdot\chi(T\ucirc S,X)}(C\times B)\otimes\Phi_\ast(X)
	\ar[r]_-{\Phi(T\ucirc S/C\times B)}
	&
	\Phi_\ast(Z)
}\]
commutes.
\end{Gluing}

\begin{Identity}
Let $X$ be an h-graph and let $(X\times I)/\pt\colon X\hto X$ be the cylinder
h-graph cobordism.
Then the diagram
\[\xymatrix{
	\bbF\otimes\Phi_\ast(X)\ar[rr]^l\ar[dr]_{\theta\otimes 1}
	&
	{}
	&
	\Phi_\ast(X)
	\\
	{}
	&
	H_\ast(\pt)\otimes\Phi_\ast(X)
	\ar[ur]_{\ \ \Phi((X\times I)/\pt)}
	&
	{}
}\]
commutes.
Here $\theta\colon \bbF\to H_\ast(\pt)$
is the canonical isomorphism and $l$ denotes the
left unit constraint for the tensor product of
graded $\bbF$-vector spaces.
\end{Identity}

\begin{Monoidality}
Given families of h-graph cobordisms
\[
	S_1/B_1\colon X_1\hto Y_1
	\quad\text{and}\quad
	S_2/B_2\colon X_2\hto Y_2,
\]
the pentagon
\[\xymatrix@!0@R=6ex@C=9.7em{
	*{\left[\begin{array}{c}
		H_{\ast-d\cdot\chi(S_1,X_1)}(B_1)\otimes \Phi_\ast(X_1)
		\\
		\otimes
		\\
		H_{\ast-d\cdot\chi(S_2,X_2)}(B_2)\otimes \Phi_\ast(X_2)
	\end{array}\right]}
	\ar@{<->}[rr]^-\isom_-{\mathrm{perm}}
	\ar[ddd]_{\Phi(S_1/B_1)\otimes\Phi(S_2/B_2)}
	&
	{}
	&
	*{\left[\begin{array}{c}
		H_{\ast-d\cdot\chi(S_1,X_1)}(B_1)\otimes H_{\ast-d\chi(S_2,X_2)}(B_2)
		\\
		\otimes
		\\
		\Phi_\ast(X_1)\otimes \Phi_\ast(X_2)
	\end{array}\right]}
	\ar[ddd]^{\times\otimes\Phi_\otimes}
	\\
	{}
	&
	{}
	&
	{}
	\\
	{}
	&
	{}
	&
	{}
	\\
	\Phi_\ast(Y_1)\otimes \Phi_\ast(Y_2)
	\ar `d[ddr] [ddr]_(0.35){\Phi_\otimes}
	& 
	{}
	&
	*{\left[\begin{array}{c}
	H_{\ast-d\cdot\chi(S_1\usqcup S_2,X_1\sqcup X_2)}(B_1\times B_2)
	\\
	\otimes
	\\
	\Phi_\ast(X_1\sqcup X_2)
	\end{array}\right]}
	\ar `d[ddl] [ddl]^(0.25){\Phi((S_1\usqcup S_2)/(B_1\times B_2))} 
	\\
	{}
	&
	{}
	&
	{}
	\\
	{}
	&
	\Phi_\ast(Y_1\sqcup Y_2)
	&
	{}	
}\]
commutes.  Here $\Phi_\otimes$ is the monoidality constraint 
for the symmetric monoidal functor $\Phi_\ast$.
\end{Monoidality}

\begin{remark}[Universal operations]
\label{rk:universal-operations}
The universal family $U\hAut^w(S_0)/B\hAut^w(S_0)$
of Theorem~\ref{th:universal}
associated to an h-graph cobordism
$S_0\colon X\hto Y$ induces an operation
\[
		\Phi(U\hAut^w(S_0)/B\hAut^w(S_0))\colon
		H_{\ast-d\cdot\chi(S_0,X)}(B\hAut^w(S_0))\otimes\Phi_\ast(X)
		\longrightarrow
		\Phi_\ast(Y).
\]
These universal examples determine all operations in the field theory.
To see this, let $S/B\colon X\hto Y$ be a family of h-graph cobordisms,
and assume that $B$ is a connected CW-complex.
This is no loss, for by the base change axiom,
the operation $\Phi(S/B)$ is determined by the operations induced by the
restrictions of $S$ to the path components of $B$; and using the 
base change axiom again, one sees that $\Phi(S/B)$ can be read
off from the operation associated with the pullback of $S$ to a
CW approximation of $B$.
Applying Corollary~\ref{cor:universal}  to $S/B$ and then using base change
shows that the operation $\Phi(S/B)$
factors through $\Phi(U\hAut^w(S_0)/B\hAut^w(S_0))$
in a canonical way.
\end{remark}

\begin{remark}[Families \emph{vs} universal families]
\label{rk:families-vs-universal-families}
An important advantage of using arbitrary families and base change
in the data for an HHGFT,
rather than including only the universal examples above,
is that it allows us to detach issues about the construction and properties
of the universal fibrations from the definition and construction of HHGFTs.
Indeed, the universal families are difficult to construct and interrelate,
choices being required at every step.
Moreover, having the non-universal families of h-graph cobordisms
available is useful for performing calculations such as the
one we make in Theorem~\ref{thm:calc}.
\end{remark}

\begin{remark}[Homotopy invariance]
The above axioms imply that $\Phi_\ast$ is homotopy invariant in the 
sense that $\Phi_\ast(f) = \Phi_\ast(g)$ when $f$ and $g$
are homotopic homotopy equivalences between h-graphs.  
Thus in particular $\Phi_\ast(f)$ is an isomorphism for each $f$. 
To see the claimed homotopy invariance, 
suppose $f\colon X \to Y$ is a homotopy 
equivalence between h-graphs. 
Then the mapping cylinder $M_f$ of $f$ gives an h-graph cobordism 
from $X$ to $Y$, and there is an evident 2-cell
\[\xymatrix@C+1em{
	X
	\ar|-@{|}[r]^{M_f/\pt}_{\phantom{.}}="A"
	\ar[d]_f
	&
	Y
	\ar[d]^{\id}
	\\
	Y
	\ar|-@{|}[r]_{(Y\times I)/\pt}^{\phantom{.}}="B"
	&
	Y
	\ar@{=>}"A";"B"
}\]
Applying the base change and identity axioms, we see that the map
\[
	\Phi_\ast(f)\colon \Phi_\ast(X) \longto \Phi_\ast(Y)
\]
agrees with the one obtained from
\[
	\Phi(M_f/\pt)\colon H_\ast(\pt)\tensor\Phi_\ast(X) \longto \Phi_\ast(Y)
\]
by evaluating against the generator of $H_*(\pt)$. 
But if $g\colon X\to Y$ is homotopic to $f$, then $M_f$ and $M_g$
are related by a 2-cell
\[\xymatrix@C+1em{
	X
	\ar|-@{|}[r]^{M_f/\pt}_{\phantom{.}}="A"
	\ar[d]_{\id}
	&
	Y
	\ar[d]^{\id}
	\\
	X
	\ar|-@{|}[r]_{M_g/\pt}^{\phantom{.}}="B"
	&
	Y
	\ar@{=>}"A";"B"
}\]
whence the base change axiom implies that $\Phi(M_f/\pt) = \Phi(M_g/\pt)$.
\end{remark}

\begin{remark}
Expressed in the language of symmetric monoidal double categories
(see Shulman~\cite{ShulmanSMBicat} and 
subsections~\ref{subsec:double-categories} and 
\ref{subsec:symmetric-monoidal-double-categories}),
an HHGFT corresponds precisely to a  
symmetric monoidal double functor
\[
	\Phi\colon\cob \longto \pmor(\grmod)
\]
from a certain double category $\cob$ 
of families of h-graph cobordisms
to a double category $\pmor(\grmod)$
of parameterized morphisms in $\grmod$. 
More precisely, the double category $\cob$ has
as objects h-graphs; as vertical morphisms
homotopy equivalences of h-graphs; as horizontal
morphisms families of h-graph cobordisms; and as
2-cells the 2-cells of Definition~\ref{def:two-cell}, 
except that the map $\varphi_S$ is only considered up to 
appropriate homotopy. On the level of horizontal morphisms,
the symmetric monoidal structure of $\cob$ is given by
external disjoint union. The double category $\pmor(\grmod)$
is described as follows: the objects are objects of $\grmod$;
the vertical morphisms are morphisms of $\grmod$; horizontal
morphism from $A$ to $B$ consist of a 
`parameterizing' object $V$ of $\grmod$ together with a 
map $V\tensor A \to B$ in $\grmod$; and a 2-cell 
is a morphism between parameterizing modules compatible 
with the boundary morphisms of the 2-cell.
We have chosen to avoid the double categorical language
in our definition of HHGFTs for the sake of directness 
and simplicity, but it is perhaps reassuring to know that
HHGFTs as we defined them amount to certain kind 
of symmetric monoidal functors, as outlined above. 

Later in this paper, we will make use of symmetric monoidal
double categories to encode theories of umkehr maps. 
However, unlike in the symmetric monoidal double categories we are 
going to encounter then, the composition of horizontal 
morphisms in $\cob$ and $\pmor(\grmod)$ fails to be 
strictly associative. The definitions we put forward in 
subsection~\ref{subsec:double-categories} therefore do not 
suffice to capture the structure of $\cob$ and $\pmor(\grmod)$;
instead, one should make use of the more general notions
defined in \cite{ShulmanSMBicat}.
\end{remark}

\begin{remark}[HHGFTs in arbitrary characteristic]
\label{rem:arbitrary-characteristic}
The definition of HHGFT given in this subsection is specialised
to the case where the ground field $\bbF$ has characteristic $2$.
Indeed, that is all that we require for the present paper.
However, as discussed in subsection~\ref{statement:subs},
there is a version of Theorem~\ref{thm:main}
that holds in arbitrary characteristic.

An HHGFT $\Phi$ over a field $\bbF$ of arbitrary characteristic
consists of a symmetric monoidal functor $\Phi_\ast$
from the category of h-graphs and homotopy equivalences among them
into the category of graded $\bbF$-vector spaces,
together with operations
\[
	\Phi(S/B)
	\colon
	H_\ast(B;\partial_S^{\tensor d})
	\otimes\Phi_\ast(X)
	\longrightarrow
	\Phi_\ast(Y),
\]
one for each family of h-graph cobordisms $S/B\colon X\hto Y$.
Here $\partial_S$ is the local coefficient system whose fibre over $b\in B$
is the determinant line $\det(H^\ast(S_b,X))$.
It is analogous to the local coefficient system $\partial_\Sigma$
that appears in HCFT operations in arbitrary characteristic~\eqref{HCFTop1:eq}.
The functor $\Phi_\ast$ and the operations $\Phi(S/B)$
are subject to base change, gluing, identity and monoidality axioms 
that have the same general form as the ones listed above,
but now modified to take into account the coefficient systems $\partial_S$.
\end{remark}

\subsection{HCFTs from HHGFTs}\label{sbs:HCFT}
The present section will demonstrate how to obtain a
homological conformal field theory from a 
homological h-graph field theory.
The definition of homological conformal field theory
is phrased in various ways in the literature.
For our purposes the following is most convenient.
A \emph{degree $d$ homological conformal field theory}
(or \emph{HCFT})
$\phi$ over a field $\bbF$ of characteristic 2 consists of
the following data:
\begin{itemize}
	\item A symmetric monoidal functor $\phi_\ast$
 	from the category of 1-manifolds and diffeomorphisms
	into the category of graded $\bbF$-vector spaces.
	Here the monoidal structures are disjoint union and tensor product, respectively.

	\item For each open-closed cobordism $\Sigma$ from $X$ to $Y$,
	a map
	\[
		\phi(\Sigma)\colon
		H_{\ast-d\cdot\chi(\Sigma,X)}(B\Diff(\Sigma))\otimes\phi_\ast(X)
		\longrightarrow
		\phi_\ast(Y).
	\]
	Here homology is taken with $\bbF$ coefficients.
\end{itemize}
These data are required to be compatible with
disjoint unions and compositions of cobordisms
and diffeomorphisms of cobordisms.
The permitted classes of $1$-manifolds and cobordisms vary from example to example.
Compare with \cite{Godin}, \cite{ChataurMenichi} and \cite{Costello}.
The definition here is tailored to characteristic $2$.

\begin{example}[String topology of manifolds]
\label{ex:godin}
	Let $M$ be a closed manifold.
	Godin \cite{Godin} constructs an HCFT of degree $\dim(M)$
	for which
	\[I\mapsto H_\ast(M),\qquad  S^1 \mapsto H_\ast(LM).\]
	Here the $1$-manifolds may have boundary.
	The cobordisms $\Sigma$ are open-closed, 
	and are subject to the boundary condition
	that every component meets $\partial_\rmout\Sigma\cup\partial_\rmfree\Sigma$.
	(Godin proves a result in arbitrary characteristic, so long as $M$ is orientable.)
\end{example}

\begin{example}[String topology of classifying spaces]
\label{ex:chataur-menichi}
	Let $G$ be a connected compact Lie group or a finite group.
	Chataur and Menichi \cite{ChataurMenichi}
	construct an HCFT of degree $-\dim(G)$ for which
	\[S^1 \mapsto H_\ast(LBG).\]
	In this case the 1-manifolds must be closed,
	the cobordisms $\Sigma$ may not have
	free boundary, and every component must meet both $\partial_\rmin\Sigma$
	and $\partial_\rmout\Sigma$.
	(Chataur and Menichi prove a result valid in any characteristic.)
	In section~\ref{sec:ChataurMenichi}
	we will compare this result with our main theorem.
\end{example}
	
\begin{example}[Hochschild homology]\label{ex:wahl-westerland}
	Let $A$ be an $A_\infty$-Frobenius algebra of degree $d$.
	Following Costello~\cite{Costello},
	Wahl and Westerland~\cite{WahlWesterland}
	construct an HCFT of degree $d$ for which
	\[I \mapsto H_\ast(A),\qquad S^1 \mapsto HH_\ast(A).\]
	Here the $1$-manifolds may have boundary.
	The cobordisms $\Sigma$ are open-closed, and are subject to 
	the boundary condition that every component meets
	$\partial_\rmin\Sigma\cup\partial_\rmout\Sigma$.
	(Costello's result was in characteristic $0$;
	Wahl and Westerland prove a result valid in any characteristic.)
\end{example}

Suppose given a degree $d$ homological h-graph field theory $\Phi$.
Then one can obtain a degree $d$ homological conformal field theory $\phi$
whose value on a 1-manifold $X$ is $\phi_\ast(X)=\Phi_\ast(X)$,
and for which the graded-linear map
	\[
		\phi(\Sigma)\colon
		H_{\ast-d\cdot\chi(\Sigma,X)}(B\Diff(\Sigma))
		\otimes
		\phi_\ast(X)
		\longrightarrow
		\phi_\ast(Y),
	\]
is given by $\Phi(U\Diff(\Sigma) / B\Diff(\Sigma))$,
the value of $\Phi$ on the universal family
discussed in Example~\ref{ex:universal-family}.
In this theory all compact $1$-manifolds are permitted.
The cobordisms are open-closed, subject to the boundary condition
that every component meets $\partial_\rmin\Sigma\cup\partial_\rmfree\Sigma$
(or $\partial_\rmin\Sigma$ if $\Phi$ is positive).

The conditions required to make $\phi$ an HCFT
follow from the conditions we have imposed on $\Phi$.
Consider, for example, attempting to prove that $\phi$ respects composition of
cobordisms.
Once unwound, this condition amounts to the compatibility of
the values of $\Phi$ on the three families
\[
		U\Diff(\Sigma) / B\Diff(\Sigma),\quad
		U\Diff(\Sigma') / B\Diff(\Sigma'),\quad
		U\Diff(\Sigma'\circ\Sigma) / B\Diff(\Sigma'\circ\Sigma)
\]
whenever $\Sigma$ and $\Sigma'$ are composable open-closed cobordisms.
An application of the gluing axiom relates the first two 
of these to the external composite
\[(U\Diff(\Sigma')\ucirc U\Diff(\Sigma))/ (B\Diff(\Sigma')\times B\Diff(\Sigma)),\]
and then an application of base-change to the 2-cell
\[\xymatrix@C+5em{
	X
	\ar|-@{|}[r]
	^{U\Diff(\Sigma')\ucirc U\Diff(\Sigma)}
	_{\phantom{.}}="A"
	\ar@{=}[d]
	&
	Z
	\ar@{=}[d]
	\\
	X
	\ar|-@{|}[r]_{U\Diff(\Sigma'\circ\Sigma)}^{\phantom{.}}="B"
	&
	Z
	\ar@{=>}"A";"B"
}\]
relates this to the third.
The other compatibilities follow similarly.
In particular compatibility with diffeomorphisms of open-closed cobordisms
follows from the base change axiom.

\subsection{Comparison between HCFTs
and HHGFTs}\label{sbs:compare}
Let us fix a degree $d$ HCFT $\phi$
and a degree $d$ HHGFT $\Phi$,
and let us assume that
$\Phi$ restricts to $\phi$ as in subsection~\ref{sbs:HCFT}.
Such an extension $\Phi$ of $\phi$ produces a raft of new data beyond
what is contained in $\phi$, and its existence also imposes new
restrictions on $\phi$ itself.
Here we will highlight several such phenomena.
In particular we will see that not all HCFTs
admit such an extension.

\subsubsection{Degree-$0$ operations}
The \emph{degree-$0$ operation} associated to an open-closed cobordism
$\Sigma$ from $X$ to $Y$ is the linear map
\[\phi_\Sigma\colon \phi_\ast(X)\longrightarrow\phi_{\ast+d\cdot\chi(\Sigma,X)}(Y)\]
obtained from $\phi(\Sigma)$ by selecting the class of a single
point in $H_{\ast-d\cdot\chi(\Sigma,X)}(B\Diff(\Sigma))$.
Similarly, the \emph{degree-$0$ operation} associated to an h-graph
cobordism $S\colon X\hto Y$ is the linear map
\[\Phi_S\colon \Phi_\ast(X)\longrightarrow\Phi_{\ast+d\cdot\chi(S,X)}(Y)\]
obtained from $\Phi(S/\pt)$ by selecting the standard generator of
$H_{\ast-d\cdot\chi(S,X)}(\pt)$.
The operation $\phi_\Sigma$ depends only on the diffeomorphism
class of $\Sigma$ relative to $X\sqcup Y$, while $\Phi_S$ depends
only on the homotopy-equivalence class of $S$ relative to $X\sqcup Y$.
If we take $S=\Sigma$ then
an instance of the base change axiom shows that
the degree-$0$ operations $\phi_\Sigma$ and $\Phi_\Sigma$
coincide.

\begin{example}\label{ex:factorisation}
\emph{Existing cobordisms can be factorised in new ways.}
The pair of pants cobordism from two circles to one can be expressed
as a composite of h-graph cobordisms:
\[\begin{tikzpicture}[scale=0.05,baseline=-2]
	\path[ARC, fill=gray!20, join=round] (0,18) .. controls (25,18) and (20,6) .. (40,6)
		arc [start angle=90, end angle=-90, x radius=3, y radius=6]
		.. controls (20,-6) and (25,-18) .. (0,-18)
		-- (0,-6)
		.. controls (5,-6) and (15,-5) .. (15,0)
		.. controls (15,5) and (5,6) .. (0,6)
		-- cycle;	
	\path[ARC, densely dotted] (40,-6) 
		arc [start angle=270, end angle=90, x radius=3, y radius=6];
	\path[ARC,fill=gray!50] (0,12) ellipse (3 and 6);
	\path[ARC,fill=gray!50] (0,-12) ellipse (3 and 6);
	\node () at (20,-30) {$P$};
\end{tikzpicture}
\qquad=\qquad
\begin{tikzpicture}[scale=0.05,baseline=-2]
	\path[ARC, fill=gray!20, join=round]
	(-2,6)
	-- (-2,18)
	-- (15,18)
	.. controls (20,18) and (20,5) .. (15,0)
	.. controls (20,-5) and (20,-18) .. (15,-18)
	-- (-2,-18)
	-- (-2,-6)
	.. controls (5,-6) and (14.055,-5) .. (14.055,0)
	.. controls (14.055,5) and (5,6) .. (-2,6)
	-- cycle;
	\path[ARC, densely dotted]
		(15,0) .. controls (10,5) and (10,18) .. (15,18);
	\path[ARC, densely dotted]
		(15,0) .. controls (10,-5) and (10,-18) .. (15,-18);	
	\path[ARC, fill=gray!50] (-2,12) ellipse (3 and 6);
	\path[ARC, fill=gray!50] (-2,-12) ellipse (3 and 6);
	\node () at (7,-30) {$Q$};
\end{tikzpicture}
\:\:
\circ
\:\:
\begin{tikzpicture}[scale=0.05,baseline=-2]
	\path[ARC, fill=gray!20, join=round]
	(15,18)
	.. controls (30,18) and (25,6) .. (40,6)
	arc [start angle=90, end angle=-90, x radius=3, y radius=6]
	.. controls (25,-6) and (30,-18) .. (15,-18)
	-- cycle;	
	\path[ARC, densely dotted] (40,-6) 
		arc [start angle=270, end angle=90, x radius=3, y radius=6];
	\path[ARC, fill=gray!50]
	(15,18)
	.. controls (20,18) and (20,5) .. (15,0)
	.. controls (10,5) and (10,18) .. (15,18)
	-- cycle;
	\path[ARC, fill=gray!50]
	(15,-18)
	.. controls (20,-18) and (20,-5) .. (15,0)
	.. controls (10,-5) and (10,-18) .. (15,-18)
	-- cycle;
	\node () at (27,-30) {$R$};
\end{tikzpicture}\]
The gluing axiom for $\Phi$ then gives a new factorisation
$\phi_P = \Phi_R\circ\Phi_Q$
of the original degree-$0$ operation.
\end{example}

\begin{example}\label{ex:whistlepunchcard}
\emph{Open-closed cobordisms can be homotopy equivalent
but not diffeomorphic.}
Consider the following two open-closed cobordisms $U$ and $V$
from the interval $I$ to itself.
\[\begin{tikzpicture}[scale=0.05,baseline=0]
	\path[ARC, fill=gray!50] (0,10) -- (50,10) -- (50,-10) -- (0,-10) -- (0,10) -- cycle;
	\path[ARC, fill=white] (25,0) circle (6);
	\node () at (25,-20) {$U$};
\end{tikzpicture}
\qquad\qquad\qquad
\begin{tikzpicture}[scale=0.05,baseline=0]
	\path[ARC, fill=gray!20] (12,13) -- (38,13) -- (38,-13) -- (12,-13) -- cycle;
	\path[ARC, fill=gray!50]
	(0,0)
	-- (0,10)
	.. controls (10,10) and (8, 13) .. (12,13)
	.. controls (23,13) and (23,-13) .. (12,-13)
	.. controls (8,-13) and (10,-10) .. (0,-10)
	-- cycle;
	\path[ARC, fill=gray!50]
	(50,0)
	-- (50,10)
	.. controls (40,10) and (42, 13) .. (38,13)
	.. controls (27,13) and (27,-13) .. (38,-13)
	.. controls (42,-13) and (40,-10) .. (50,-10)
	-- cycle;
	\node () at (25,-20) {$V$};
\end{tikzpicture}\]
The two are not diffeomorphic relative to $I\sqcup I$,
but they \emph{are} homotopy equivalent relative to $I\sqcup I$.
It follows that $\phi_U=\phi_V$
as a simple consequence of the base-change axiom for $\Phi$.
This relation does not hold in a general HCFT.
For example, one can take a finite group $G$ and consider
the group-ring $\bbF[G]$ as a strict Frobenius algebra
(with trace $\sum a_g g\mapsto a_e$)
and so as an $A_\infty$-Frobenius algebra.
Applying the method of Wahl and Westerland
(Example~\ref{ex:wahl-westerland})
produces an HCFT with $\phi_\ast(I)=\bbF[G]$.
In this case the two operations above can be computed directly from
the Frobenius structure and for $g\in G$ we have
$\phi_U(g)=|G|\cdot g$ and $\phi_V(g)=\sum_{k\in G} kgk^{-1}$.
These are distinct in general unless $g$ is in the centre of $G$.
\end{example}

\begin{example}\label{ex:beanie}
\emph{New cobordisms between $1$-manifolds.}
Consider the disk as an h-graph cobordism
$D\colon S^1\hto I$
as shown.
\[\begin{tikzpicture}[scale=0.05,baseline=0]
	\path[ARC, fill=gray!20] (0,12) .. controls (10,12) .. (20,6)
	-- (20,-6)
	.. controls (10,-12) .. (0,-12);
	\path[ARC, fill=gray!50] (0,0) ellipse (6 and 12);
\end{tikzpicture}\]
This h-graph cobordism does not arise from an open-closed cobordism,
even up to homotopy equivalence relative to $S^1\sqcup I$.
The result is a new operation
\[\Phi_D\colon \phi_\ast(S^1)\longrightarrow \phi_{\ast+d}(I)\]
from the closed to the open part of the HCFT.
This operation leads to a new algebraic relationship between $\phi_\ast(S^1)$
and $\phi_\ast(I)$, as we will see in \ref{sbs:algebraic} below.
\end{example}

\subsubsection{Higher operations}
Let $\Sigma$ be an open-closed cobordism from $X$ to $Y$.
By assumption the \emph{higher operation}
\[
	\phi(\Sigma)\colon
	H_{\ast-d\cdot\chi(\Sigma,X)}(B\Diff(\Sigma))
	\otimes
	\phi_\ast(X)
	\longrightarrow
	\phi_\ast(Y)
\]
is obtained by applying $\Phi$ to the family
$U\Diff(\Sigma)/B\Diff(\Sigma)$ of Example~\ref{ex:universal-family}.
There is a 2-cell
\[\varphi\colon U\Diff(\Sigma)/B\Diff(\Sigma)
\Rightarrow
U\hAut^w(\Sigma)/B\hAut^w(\Sigma)\]
with target the universal family of Theorem~\ref{th:universal} 
and consequently a factorisation:
\[\xymatrix@C=60 pt{
	H_{\ast-d\cdot\chi(\Sigma,X)}(B\Diff(\Sigma))
	\otimes
	\phi_\ast(X)
	\ar[r]^-{\phi(\Sigma)}
	\ar[d]
	&
	\phi_\ast(Y)
	\\
	H_{\ast-d\cdot\chi(\Sigma,X)}(B\hAut^w(\Sigma))
	\otimes
	\phi_\ast(X)
	\ar[ur]_{\qquad \qquad \Phi(U\hAut^w(\Sigma)/B\hAut^w(\Sigma))}
	&
	{}
}\]
In this case $\hAut(\Sigma)$ has nondegenerate basepoint,
so that we may use it in place of $\hAut^w(\Sigma)$ throughout,
and then one can check that
$\varphi_B\colon B\Diff(\Sigma)\to B\hAut(\Sigma)$
is induced by the evident homomorphism.
If the resulting map
$(\varphi_B)_\ast\colon H_\ast(B\Diff(\Sigma))\to H_\ast(B\hAut(\Sigma))$
is not surjective we therefore obtain new higher operations associated to $\Sigma$,
and if it is not injective then we obtain new 
relations among existing higher operations associated to $\Sigma$.
When $\Sigma$ has no free boundary there is nothing
new to learn, for then $(\varphi_B)_\ast$
is an isomorphism as in Example~\ref{ex:DNB}.
This is not necessarily the case for open-closed cobordisms
with free boundary, as the next example demonstrates.

\begin{example}
\label{ex:diff-vs-haut}
Consider the open-closed cobordisms $U$ and $V$ of Example~\ref{ex:whistlepunchcard}.
For the diffeomorphism groups we have
\[\Diff(U)\simeq\pt,\qquad \Diff(V)\simeq\bbZ\]
but for homotopy automorphisms we have
\[\hAut(U)\simeq\bbZ\rtimes\bbZ/2,\qquad \hAut(V)\simeq\bbZ\rtimes\bbZ/2.\]
In each case the homomorphism $(\varphi_B)_\ast$
is injective but not surjective.
\end{example}

\subsubsection{Algebraic consequences}\label{sbs:algebraic}
The value $\phi_\ast(I)$ of our HCFT on the interval is a counital (but potentially not unital)
Frobenius algebra of degree $d$,
with product, coproduct and counit given by the degree-$0$ operations associated
to the three cobordisms on the left.
\[\begin{tikzpicture}[scale=0.03,baseline=0]
	\path[ARC, fill=gray!50]
	(0,18)
	.. controls (25,18) and (20,6) .. (40,6)
	-- (40,-6)
	.. controls (20,-6) and (25,-18) .. (0,-18)
	-- (0,-6)
	.. controls (5,-6) and (15,-5) .. (15,0)
	.. controls (15,5) and (5,6) .. (0,6)
	-- cycle;
\end{tikzpicture}
\qquad
\begin{tikzpicture}[scale=0.03,baseline=0, xscale=-1]
	\path[ARC, fill=gray!50]
	(0,18)
	.. controls (25,18) and (20,6) .. (40,6)
	-- (40,-6)
	.. controls (20,-6) and (25,-18) .. (0,-18)
	-- (0,-6)
	.. controls (5,-6) and (15,-5) .. (15,0)
	.. controls (15,5) and (5,6) .. (0,6)
	-- cycle;
\end{tikzpicture}
\qquad
\begin{tikzpicture}[scale=0.03,baseline=0]
	\path[ARC, fill=gray!50]
	(0,6)
	.. controls (15,6) and (15,-6) .. (0,-6)
	-- cycle;
\end{tikzpicture}
\qquad\qquad\qquad
\begin{tikzpicture}[scale=0.03,baseline=0]
	\path[fill=black] (0,12) circle (2);
	\path[fill=black] (0,-12) circle (2);
	\path[fill=black] (40,0) circle (2);
	\path[ARC] (0,12) .. controls (20,12) and (22,5) .. (22,0);
	\path[ARC] (0,-12) .. controls (20,-12) and (22,-5) .. (22,0);
	\path[ARC] (22,0) -- (40,0);
\end{tikzpicture}
\qquad
\begin{tikzpicture}[scale=0.03,baseline=0, xscale=-1]
	\path[fill=black] (0,12) circle (2);
	\path[fill=black] (0,-12) circle (2);
	\path[fill=black] (40,0) circle (2);
	\path[ARC] (0,12) .. controls (20,12) and (22,5) .. (22,0);
	\path[ARC] (0,-12) .. controls (20,-12) and (22,-5) .. (22,0);
	\path[ARC] (22,0) -- (40,0);
\end{tikzpicture}
\qquad
\begin{tikzpicture}[scale=0.03,baseline=0]
	\path[fill=black] (0,0) circle (2);
	\path[ARC, fill=gray!50]
	(0,0) -- (10, 0);
\end{tikzpicture}
\]
Similarly the value $\Phi_\ast(\pt)$ of our HHGFT on a point
is a counital (but possibly not unital)
Frobenius algebra of degree $d$.
The product, coproduct and counit are given by the degree-$0$
operations associated to the three h-graph cobordisms on the right.
In this case the product and coproduct are clearly commutative.

\begin{proposition}\label{prop:commutative}
$\phi_\ast(I)$ and $\Phi_\ast(\pt)$ are isomorphic as Frobenius algebras.
In particular $\phi_\ast(I)$ is commutative and cocommutative.
\end{proposition}

According to the definition of homological h-graph field theory,
$\Phi_\ast$ is functorial with respect to homotopy equivalences of h-graphs.
Now the required isomorphism is simply the map
$\phi_\ast(I)=\Phi_\ast(I)\to\Phi_\ast(\pt)$ induced
by the map $I\to\pt$.
That it respects the three algebraic structures follows using the base change
rule.  For example base change can be applied to the 2-cell
\[\xymatrix@C=70 pt{
	\begin{tikzpicture}[scale=0.02,baseline=-3]
		\path[ARC] (0,18)-- (0,6);
		\path[ARC] (0,-18)-- (0,-6);
	\end{tikzpicture}
	\ar|-@{|}[r]^{
		\begin{tikzpicture}[scale=0.02,baseline=0]
			\path[ARC, fill=gray!50]
			(0,18)
			.. controls (25,18) and (20,6) .. (40,6)
			-- (40,-6)
			.. controls (20,-6) and (25,-18) .. (0,-18)
			-- (0,-6)
			.. controls (5,-6) and (15,-5) .. (15,0)
			.. controls (15,5) and (5,6) .. (0,6)
			-- cycle;
		\end{tikzpicture}
	}_{\phantom{.}}="A"
	\ar[d]
	&
	\begin{tikzpicture}[scale=0.02,baseline=-3]
		\path[ARC] (40,6) -- (40,-6);
	\end{tikzpicture}
	\ar[d]
	\\
	\begin{tikzpicture}[scale=0.02,baseline=-3]
		\path[fill=black] (0,12) circle (3);
		\path[fill=black] (0,-12) circle (3);
	\end{tikzpicture}
	\ar|-@{|}[r]_{
		\begin{tikzpicture}[scale=0.02,baseline=0]
			\path[fill=black] (0,12) circle (3);
			\path[fill=black] (0,-12) circle (3);
			\path[fill=black] (40,0) circle (3);
			\path[ARC] (0,12) .. controls (20,12) and (22,5) .. (22,0);
			\path[ARC] (0,-12) .. controls (20,-12) and (22,-5) .. (22,0);
			\path[ARC] (22,0) -- (40,0);
		\end{tikzpicture}	
	}^{\phantom{.}}="B"
	&
	\begin{tikzpicture}[scale=0.02,baseline=-3]
		\path[fill=black] (40,0) circle (3);
	\end{tikzpicture}
	\ar@{=>}"A";"B"^\varphi
}\]
to show that the isomorphism respects products.

The proposition shows that not every homological conformal field theory
can be extended to an h-graph theory.
For example, applying the results of 
Wahl and Westerland~\cite{WahlWesterland}
to the group ring $\bbF[G]$ of a finite nonabelian group $G$
produces a homological conformal field theory $\phi$ with 
$\phi(I)=\bbF[G]$.

\begin{proposition}
\label{prop:retraction}
$\phi_\ast(I)$ is a retract of $\phi_\ast(S^1)$ in the category of counital coalgebras.
\end{proposition}

This proposition follows by letting
$\iota\colon\phi_\ast(I)\to\phi_{\ast-d}(S^1)$
and $\pi\colon\phi_\ast(S^1)\to\phi_{\ast+d}(I)$
denote the degree-$0$ operations determined by
the following h-graph cobordisms.
\[\begin{tikzpicture}[scale=0.04,baseline=0]
	\path[ARC, fill=gray!20, join=round]
	(12,12.0)
	-- (34,12.0) 
	arc [start angle=90, end angle=-90, x radius=6, y radius=12]
	-- (34,-12.0) -- (12,-12.0) -- cycle;
	\path[ARC, densely dotted]
	(34,12.0) arc [start angle=90, end angle=270, x radius=6, y radius=12];
	\path[ARC, fill=gray!50]
	(0,0)
	-- (0,6.0)
	.. controls (10,6.0) and (8, 12.0) .. (12,12.0)
	.. controls (21,12.0) and (21,-12.0) .. (12,-12.0)
	.. controls (8,-12.0) and (10,-6.0) .. (0,-6.0)
	-- cycle;
\end{tikzpicture}
\qquad\qquad\qquad
\begin{tikzpicture}[scale=0.04,baseline=0]
	\path[ARC, fill=gray!20] (0,12) .. controls (10,12) .. (20,6)
	-- (20,-6)
	.. controls (10,-12) .. (0,-12);
	\path[ARC, fill=gray!50] (0,0) ellipse (6 and 12);
\end{tikzpicture}\]
It is easy to see that these maps are morphisms of counital coalgebras
and that they satisfy $\pi\circ\iota=\id$.

\subsection{H-graphs and string topology}\label{sbs:hgraphs-stringtop}
\subsubsection{H-graphs and string topology of manifolds}
We regard HHGFTs as a more natural setting than HCFTs
for discussing string topology of classifying spaces.
For example, in the previous subsection, we demonstrated 
how certain properties of  Chataur and Menichi's
string topology HCFT (such as the loop product factoring as 
in Example~\ref{ex:factorisation} and the commutativity of the 
open sector proved in Proposition~\ref{prop:commutative})
follow easily from the fact that the HCFT extends to an HHGFT.
While these properties could be established by examining 
Chataur and Menichi's 
construction, they cannot be derived from the HCFT axioms alone.

The same phenomena (factorisation of the loop product, commutativity
of the open sector) also occur in string topology of manifolds
(Example~\ref{ex:godin}), and would follow directly from the 
conjecture below.
\begin{conjecture}
\label{conj:manifold-hhgft}
Let $M$ be a closed oriented manifold.
Then there is a reversed homological h-graph field theory
of degree $\dim(M)$ 
that extends the HCFT of Godin.
\end{conjecture}

Here the notion of \emph{reversed} HHGFT
is defined in the same way as an ordinary HHGFT 
after adjusting the definition of
h-graph cobordisms and families of such
so that now it is the inclusion of the incoming boundary,
and not the outgoing boundary, which is assumed to be an h-embedding.

\subsubsection{Vanishing results}

Tamanoi~\cite[Vanishing Theorem]{TamanoiStable}
has shown that in string topology of manifolds 
and string topology of classifying spaces
(Examples~\ref{ex:godin} and~\ref{ex:chataur-menichi})
all stable operations vanish.  More precisely, let $P$ denote the
cobordism from $S^1$ to itself obtained from a torus
by deleting two open disks.  Tamanoi shows that the
degree-$0$ operation associated to $P$ vanishes,
and concludes using Harer Stability that all 
HCFT operations in the stable range vanish.
The last statement means that if $\Sigma$ is a closed 
cobordism from $X$ to $Y$, admissible in the respective HCFT,
then the operation
\[
	\phi_\ast(X) \longto \phi_{\ast + k + d\chi(\Sigma,X)}(Y)
\]
associated to $x \in H_k(B\Diff(\Sigma))$
vanishes as long as $2k+1 \leq \mathrm{genus}(\Sigma)$.
In fact Tamanoi's result holds in a large class of HCFTs:

\begin{proposition}
Let $\phi$ be an HCFT for which $\phi_\ast(S^1)$ is concentrated
in non-negative degrees. If $\phi$ has positive degree
and admits the unit cobordism $\eta_{S^1}\colon \emptyset\to S^1$,
then $\phi_P=0$. The same conclusion also holds if
$\phi$ has negative degree and admits 
the counit cobordism $\varepsilon_{S^1}\colon S^1\to\emptyset$.
\end{proposition}
\begin{proof}
Let us prove the first statement.
If $\mu$ denotes the pair-of-pants cobordism
from two copies of $S^1$ to one, then
$P
\isom
\mu\circ (1\otimes (P\circ \eta_{S^1}))$,
and $\phi_{P\circ\eta_{S^1}}$ vanishes because it strictly decreases the degree.
\end{proof}

This vanishing result (and its proof) extend immediately to any HHGFT.
We anticipate that more vanishing results of the same flavour 
can be proved using the additional functoriality available 
in an HHGFT.
On the other hand, it is known that in the HHGFT provided by 
Theorem~\ref{thm:main}, many
operations must be nonzero. Such operations can be produced
by direct computation, as in Theorem~\ref{thm:calc},
but they also exist for abstract reasons. For example,
the presence of a counit guarantees that the pair-of-pants 
coproduct is non-zero. 

%% file: overview-of-the-construction.tex

\section{Overview of the construction}
\label{overview:section}
Let $S/B\colon X\hto Y$ be a family of h-graph cobordisms.
This section will sketch the construction of the operation 
\begin{equation}\label{HHGFTop1:eq}
	H_{\ast+\dim(G)\cdot\chi(S,X)}(B)\otimes H_\ast(BG^X)
	\longto
	H_{\ast}(BG^Y)
\end{equation}
that our HHGFT will associate to $S/B$.
Our aim is to give the reader a broad view of the construction,
without encumbering the discussion with all the additional 
details required to show that these operations do satisfy the 
axioms of a positive HHGFT.

\subsection{The push-pull construction}
\label{subsec:push-pull-construction}
The operation~\eqref{HHGFTop1:eq}
is obtained by a \emph{push-pull construction}, a general pattern for
constructing field theories that is discussed, for example,
in \cite[section~1]{FreedHopkinsTeleman}.
We consider the zig-zag
of inclusions
\[ X\times B\xrightarrow{\ \ i\ \ } S\xleftarrow{\ \ j\ \ } Y\times B,\]
and the resulting zig-zag of restrictions
\[ BG^{X\times B}\xleftarrow{\ \ i^\ast\ \ } BG^S\xrightarrow{\ \ j^\ast\ \ } BG^{Y\times B}.\]
Here, if $U\to B$ is a fibred space, then $BG^U$ denotes the 
\emph{fibrewise mapping space}
$\Map_B(U,BG\times B)$.  It is a space over $B$
whose fibre over $b\in B$ is the space of maps from $U_b$ into $BG$.
So $BG^{X\times B}$ and $BG^{Y\times B}$ are simply products
$BG^X\times B$ and $BG^Y\times B$, but $BG^S$ does not in general
factor in this way.
Now our operation \eqref{HHGFTop1:eq} is constructed as a composite of four maps,
\begin{equation}\label{composite:eq}
\vcenter{\xymatrix@R=1ex@C=-4em{
 *+[l]{H_{\ast+\dim(G)\cdot\chi(S,X)}(B)\otimes H_\ast(BG^X)}
 \ar[r]^-{\times}
 &
 *+[r]{H_{\ast+\dim(G)\cdot\chi(S,X)}(BG^{X\times B})}
 \\
 *+[l]{\phantom{H_{\ast-d\chi(S,X)}(B)\otimes H_\ast(BG^X)}}
 \ar[r]^-{i_!}
 &
 *+[r]{H_{\ast}(BG^{S})}
 \\
 *+[l]{\phantom{H_{\ast-d\chi(S,X)}(B)\otimes H_\ast(BG^X)}}
 \ar[r]^-{(j^\ast)_\ast}
 &
 *+[r]{H_\ast(BG^{Y\times B})}
 \\
 *+[l]{\phantom{H_{\ast-d\chi(S,X)}(B)\otimes H_\ast(BG^X)}}
 \ar[r]^-{\pi_\ast}
 &
 *+[r]{H_\ast(BG^Y).}
 }}
\end{equation}
Here $(j^\ast)_\ast\circ i_!$ can be regarded as a push-pull construction in families.
The external product $\times$ and the projection map $\pi_\ast$
are required to turn this push-pull map into an operation of the required form
\eqref{HHGFTop1:eq}.
The maps in the composite \eqref{composite:eq} are standard,
with the exception of $i_!$.

\subsection{Umkehr maps for fibrewise manifolds}
We must explain the definition of the nonstandard map $i_!$
appearing in the composite~\eqref{composite:eq}.
The main tool for constructing this map is
the theory of Gysin maps for fibrewise manifolds
developed by Crabb and James \cite{CrabbJames}.
Briefly, a fibrewise closed manifold $M$ over a base space $B$
is a fibre bundle over $B$
whose fibre is a smooth closed manifold and which 
satisfies a smoothness condition on transition maps.
Given a fibrewise smooth map
\[f\colon M\longto N\]
of fibrewise closed smooth manifolds over a finite CW-complex,
Crabb and James construct a {Gysin map}
\[
	f^\umk \colon N^{- \tau_N}
	\longto
	M^{- \tau_M}.
\]
By taking homology and using the Thom isomorphism, one obtains
an umkehr map
\[f^!\colon H_{\ast+\dim N}(N)\longto H_{\ast+\dim M}(M)\]
associated to $f$.
In the general case where the base space is not a finite CW-complex,
one defines the umkehr map $f^!$ by taking an appropriate colimit.

\subsection{Umkehr maps for fibrewise mapping spaces}
The nonstandard map $i_!$ appearing in composite~\eqref{composite:eq}
is obtained by applying a construction that associates to 
a map
\[k\colon U\longto V\]
between families of h-graphs over a base $B$
(satisfying certain conditions%
)
an {umkehr map}
\[ k_!\colon H_{\ast+\dim(G)\cdot\chi(V,U)}(BG^U) \longto H_\ast(BG^V).\]
The idea behind the construction of $k_!$
is to replace $BG^U$ and $BG^V$ with homotopy equivalent fibrewise manifolds,
in such a way that  the precomposition map $k^\ast\colon BG^V\to BG^U$ is replaced
with a fibrewise smooth map.
The map $k_!$ is then constructed by taking the umkehr map associated
with this fibrewise smooth map.

Here is how we replace the fibrewise mapping space $BG^U$
with a fibrewise manifold.
We assume that the map $k$ is positive.
We also assume that $U$ admits a choice of \emph{basepoints},
by which we mean a finite set $P$ and a map
$u\colon P\times B\to U$ over $B$
whose image meets every component of every fibre of $U\to B$.
Then we form the \emph{fibrewise fundamental groupoid} $\Pi_1(U,P)$.
This is a groupoid internal to topological spaces over $B$
whose fibre $\Pi_1(U_b,P)$ over $b\in B$ is the groupoid with objects $P$
in which a morphism from $p$ to $q$ is a homotopy class of paths in $U_b$
from the image of $p$ to that of $q$.
Next we form $G^{\Pi_1(U ,P)}$, another category
internal to spaces over $B$, whose fibre over $b\in B$
is the topological category of functors from $\Pi_1(U_b,P)$ to $G$.
Finally we form the classifying space $B(G^{\Pi_1(U,P)})$.
As long as the base space $B$ is sufficiently nice,
there is a natural zig-zag of homotopy equivalences
\[ BG^U \longleftrightarrow B(G^{\Pi_1(U,P)}).\]
In the case where $B=\pt$, $U=S^1$ and $P=\pt$, this zig-zag recovers
the familiar homotopy equivalence $LBG\simeq G^\mathrm{ad}\sslash G$.
Here $G^\mathrm{ad}$ denotes the group $G$ equipped with the 
conjugation action, and $X\sslash G$ for a $G$-space $X$
denotes the homotopy orbit space $EG \times_G X$.
There is a natural restriction map
\begin{equation}\label{FunFM:eq}
	B(G^{\Pi_1(U,P)}) \longrightarrow BG^P\times B
\end{equation}
whose fibre over $b\in B$ is the space of functors
$\Pi_1(U_b,P) \to G$.
Since $U_b$ is an h-graph, this space of functors
is simply a product of finitely many copies of $G$,
and so the fibres of \eqref{FunFM:eq}
are smooth manifolds.  In fact \eqref{FunFM:eq}
is itself a fibrewise closed manifold
whose fibre over $b\in B$ has dimension $\dim(G)\cdot(\chi(P)-\chi(U_b))$.

We have now seen that $BG^U$ and $BG^V$ can be replaced
with homotopy-equivalent fibrewise manifolds $B(G^{\Pi_1(U,P)})$ and
$B(G^{\Pi_1(V,P)})$ over $BG^P\times B$.
Under these equivalences, the precomposition map 
$k^\ast\colon BG^V\to BG^U$
corresponds to a fibrewise smooth map 
\[B(G^{\Pi_1(k)})\colon B(G^{\Pi_1(V,P)}) \longrightarrow B(G^{\Pi_1(U,P)}).\]
We therefore obtain an umkehr map
\[
	H_{\ast+\dim(G)\cdot\chi(V,U)}(B(G^{\Pi_1(U,P)}))
	\xto{B(G^{\Pi_1(k)})^!}
	H_{\ast}(B(G^{\Pi_1(V,P)})),
\]
and define the map $k_!$ by 
declaring that the square
\[\xymatrix@C=4em{
	H_{\ast +\dim(G)\cdot\chi(V,U)}(BG^U)
	\ar@{-->}[r]^-{k_!}
	\ar[d]_\isom
	&
	H_\ast(BG^V)
	\ar[d]^\isom
	\\
	H_{\ast+\dim(G)\cdot\chi(V,U)}(B(G^{\Pi_1(U,P)})
	\ar[r]^-{B(G^{\Pi_1(k)})^!}
	&
	H_\ast(B(G^{\Pi_1(V,P)}))
}\]	
commutes.
The resulting umkehr map $k_!$ is independent of the choice 
of basepoints $P$.

\subsection{Example: the counit}
Let us work out the steps of this construction in the case where
$S/B\colon X\hto Y$ is the counit cobordism
$S^1\hto\emptyset$
given by the disc $D^2$.
It is a family of h-graph cobordisms over a single point.
The resulting operation \eqref{HHGFTop1:eq}
has form
$H_{\ast+\dim(G)}(\pt)\otimes H_\ast(BG^{S^1})
\to
H_\ast(BG^\emptyset)$.
For brevity we will omit mention of the base space $\pt$,
so that what we have is an operation
\[
	H_\ast(BG^{S^1})
	\longto
	H_{\ast-\dim(G)}(BG^\emptyset)
\]
given according to~\eqref{composite:eq} by the composite
\begin{equation}\label{eq:counit}
	H_\ast(BG^{S^1})
	\xrightarrow{\ i_!\ }
	H_{\ast-\dim(G)}(BG^{D^2})
	\xrightarrow{\ (j^\ast)_\ast\ }
	H_{\ast-\dim(G)}(BG^\emptyset)
\end{equation}
in which $i\colon S^1\to D^2$ and $j\colon\emptyset\to D^2$
are the inclusions.

To construct $i_!$, we must replace $BG^{S^1}$
and $BG^{D^2}$ with fibrewise manifolds.
As basepoints for $S^1$ we choose $u\colon P\to S^1$
to be the inclusion of a single point.
Then $\Pi_1(S^1,P)$ has a single object and one free generator,
so that the category of functors $G^{\Pi_1(S^1,P)}$ is the quotient
groupoid $G^\ad/G$, and the classifying space $B(G^{\Pi_1(S^1,P)})$
is the homotopy orbit space $G^\ad\sslash G$.
We leave it to the reader to verify that $B(G^{\Pi_1(D^2,P)})=\pt\sslash G$.
The construction now replaces the restriction map
$i^\ast\colon BG^{D^2}\to BG^{S^1}$
with the fibrewise smooth map
\[
	\pt\sslash G \longrightarrow G^\ad\sslash G
\]
over $\pt\sslash G\homeom BG^P$,
given in fibres by the inclusion of the neutral element of $G$.
The resulting umkehr map
\[
	H_\ast(G^\ad\sslash G)\longto H_{\ast-\dim(G)}(\pt\sslash G)
\]
is then a fibrewise version of the standard umkehr map
$H_\ast(G)\to H_{\ast-\dim(G)}(\pt)$ sending the fundamental class to a generator.

The operation~\eqref{eq:counit} is now given,
after replacing
$H_\ast(BG^{S^1})$ with $H_\ast(G^\ad\sslash G)$,
$H_{\ast-\dim(G)}(BG^{D^2})$ with $H_{\ast-\dim(G)}(\pt\sslash G)$
and $H_{\ast-\dim(G)}(BG^\emptyset)$ with $H_{\ast-\dim(G)}(\pt)$,
by the composite
\begin{equation}\label{eq:counit2}
	H_{\ast}(G^\ad\sslash G)
	\longto
	H_{\ast-\dim(G)}(\pt\sslash G)
	\longto
	H_{\ast-\dim(G)}(\pt)
\end{equation}
of the fibrewise umkehr with the standard induced map.
In the case where $G$ is abelian this composite 
is especially simple, for then
$G^\ad\sslash G$ may be replaced by the product $G\times BG$,
and \eqref{eq:counit2} may be replaced with
\[
	H_\ast(G)\otimes H_\ast(BG)
	\longto
	H_\ast(G)
	\longto
	H_{\ast-\dim(G)}(\pt)
\]
where the first map is induced by the projection $G\times BG\to G$
and the second map sends the fundamental class to a generator.

%% file: categorical-background.tex

\section{Categorical background}
\label{categorical-background-section}

In this section, we will recall certain category-theoretical
concepts that we will use to organize the construction of our HHGFT. 
Many of our constructions are parameterized by a base space, and indeed
on many occasions we will need to work with multiple different 
base spaces at once. In subsections \ref{subsec:fibred-categories} through 
\ref{subsec:fibrewise-opposite} we discuss the language of 
fibred categories, which is ideal for handling the base change issues
that arise. After that, in subsections \ref{subsec:double-categories} and 
\ref{subsec:symmetric-monoidal-double-categories}, we will discuss
double categories. Our HHGFT operations are constructed by a 
push-pull construction that involves both ordinary induced maps 
and umkehr (or wrong-way) maps, and we will use 
double categories as a convenient formalism for capturing 
the interactions between these two kinds of induced maps.

\subsection{Fibred categories}
\label{subsec:fibred-categories}

This subsection recalls the notion of fibred category. For 
more detailed discussions we refer the reader to 
\cite[Expos\'e VI]{SGA1} and \cite[Chapter 8]{Borceux}.

\begin{definition}
Suppose $F \colon \calF \to \calE$ is a functor, and 
suppose $f$ is a morphism in $\calE$. 
A morphism $\alpha \colon Y \to X$ with $F(\alpha) = f$
is called \emph{cartesian}
over $f$ if for any morphisms $\beta \colon Z \to X$ of $\calF$
and $g \colon F(Z) \to F(Y)$ of $\calE$ such that 
$F(\beta) = f \circ g$, there exists a unique morphism
$\gamma \colon Z \to Y$ such that $\beta = \alpha\circ \gamma$ and 
$F(\gamma) = g$:
\[\xymatrix{
	Z
	\ar@{|->}[dd]
	\ar@{-->}[dr]^{\exists! \gamma}
	\ar@/^1em/[drrr]^\beta
	\\
	&
	Y
	\ar@{|->}[dd]
	\ar[rr]^\alpha
	&&
	X
	\ar@{|->}[dd]
	\\
	F(Z)
	\ar[dr]^{g}
	\ar@/^1em/[drrr]^{f\circ g }|!{[ur];[dr]}{\hole}
	\\
	&
	F(Y)
	\ar[rr]^{f}
	&&
	F(X)
}\]
The functor $F$ is called a \emph{fibration} 
and the category $\calF$ a \emph{category fibred over} $\calE$
if for every 
morphism $f \colon C \to B$ of $\calE$ and every 
object $X$ of $\calF$ such that $F(X) = B$, there exists
a cartesian morphism $\alpha \colon Y \to X$ over $f$.
A morphism $\alpha$ of $\calF$ is called \emph{vertical} if 
$F(\alpha)$ is an identity map. Given an object $C$ of $\calE$,
the \emph{fibre} of $F$ over $C$ is the subcategory $\calF_C$ of 
$\calF$ whose objects are the objects $Y$ of $\calF$ with 
$F(Y) = C$ and whose morphisms are the vertical maps of $\calF$
covering the identity map of $C$.
\end{definition}

Our interest in fibrations is explained by their usefulness
for discussing base change phenomena.
Suppose $F \colon \calF \to \calE$ is a fibration.
Given a morphism $f\colon C \to B$ in $\calE$
and an object $X$ of $\calF$ with $F(X) = B$, 
by a \emph{base change of} $X$ \emph{along} $f$ 
we mean an object $Y$ together with a cartesian morphism
$\alpha \colon Y \to X$
covering $f$. We will sometimes denote $Y$ by $f^\ast X$.
By the definition of fibration,
a base change of $X$ along $f$ exists for all $X$ and $f$,
and by the universal property of cartesian morphisms,
it is unique up to unique vertical isomorphism.
Moreover, suppose given two objects $X_1$ and $X_2$ 
in the fibre of $F$ over
$B$ and base changes $(Y_1,\alpha_1)$ and
$(Y_2,\alpha_2)$ of $X_1$ and $X_2$ along $f$, respectively.
Then for any vertical morphism $\varphi \colon X_1\to X_2$,
there is a unique vertical morphism $\psi \colon Y_1 \to Y_2$
making the top square in the diagram
\[\xymatrix@C-1em{
	&& Y_2
	\ar@{|->}[ddl]|!{[dr];[dl]}\hole
	\ar[rrr]^(0.50){\alpha_2}
	&&&
	X_2
	\ar@{|->}[ddl]	
	\\
	Y_1
	\ar@{|->}[dr]
	\ar@{-->}[urr]^(0.4){\exists!\psi}
	\ar[rrr]^(0.65){\alpha_1}
	&&&
	X_1
	\ar@{|->}[dr]
	\ar[urr]^(0.4)\varphi
	\\
	& C
	\ar[rrr]^f
	&&&
	B
}\]
commutative. We say that $\psi$ is \emph{obtained from}
$\varphi$ \emph{by base change along} $f$
(with respect to $\alpha_1$ and $\alpha_2$).

\begin{definition}[The categories $\calB$ and $\calU$]
\label{def:calB-calU}
We denote by $\calB$ the category of base spaces
(that is, weak Hausdorff $k$-spaces) and continuous maps.
By a \emph{good base space} we mean
any space which can be embedded as an open subset of some CW complex. 
We denote by $\calU$ the full subcategory of  $\calB$
whose objects are the good base spaces.
Most of the fibred categories we will
encounter will be fibred over $\calU$, which we regard as our 
preferred category of base spaces. However, for some purposes
the category $\calU$ will be insufficient, and we will need to
work over the more general category $\calB$.
\end{definition}

Good base spaces are paracompact (see eg.\ 
\cite[Corollary II.4.4]{LundellWeingram}) and locally
contractible. Moreover, an open subset of a good base space
is again a good base space, as is the product of a finite number of
good base spaces.

\begin{example}[The fibred categories $\calS$ and $\widehat\calS$]
\label{ex:calS-def}
Let $\widehat\calS$ be the category of parameterized spaces,
that is, the category of maps $X \to B$ of spaces where
$B$ is an object of $\calB$. 
A map in $\widehat\calS$ from $X\to B$ to $Y \to C$
is a pair $(f,g)$ of continuous maps making
the square
\[\xymatrix{
	X
	\ar[d]
	\ar[r]^f 
	&
	Y
	\ar[d]
	\\
	B
	\ar[r]^g
	&
	C
}\]
commutative. The functor
\[
	\widehat\calS \longto \calB,\quad (X\to B) \longmapsto B
\]
makes $\widehat\calS$ a category fibred over $\calB$.
A map is cartesian exactly when the corresponding
square is a pullback square. We denote by $\calS$ the
full subcategory of $\widehat\calS$ consisting of spaces over
good base spaces. It is a category fibred over $\calU$,
a morphism of $\calS$ being cartesian if and only if its
image in $\widehat\calS$ is so.
\end{example}

\begin{example}[The fibred category $\calH$]
\label{ex:calH-def}
Let us denote by $\calH$ the category whose objects are 
families of h-graphs admitting a choice of basepoints 
and whose base space lies in $\calU$.
The morphisms in $\calH$ are the maps of families of h-graphs.
(See Definition~\ref{def:families-of-h-graphs} for the 
definition of families of h-graphs and maps between them, 
and recall from section~\ref{overview:section} that a family of h-graphs
$U$ over $B$ admits basepoints if there is a finite set $P$ and a 
map $u\colon P\times B\to U$ over $B$ 
whose image meets every component of every fibre of $U\to B$.)
Then $\calH$ is a full subcategory of $\calS$,
and the composite functor $\calH \to \calS \to \calU$  
makes $\calH$ into a category 
fibred over $\calU$. The cartesian morphisms in $\calH$ are
again the morphisms corresponding to pullback squares.
We call $\calH$ the \emph{category of families of h-graphs}.
\end{example}

\subsection{Symmetric monoidal fibrations}
\label{subsec:symmetric-monoidal-fibrations}

Most of the fibrations we will encounter will  
have the structure of a symmetric monoidal fibration.
Following Shulman~\cite[section 12]{Shulman},
we make the following definitions.

\begin{definition}
A functor $F\colon \calF \to \calE$ is called 
a \emph{symmetric monoidal fibration}
if $\calF$ and $\calE$ are symmetric monoidal categories,
$F$ is both a fibration and a strict symmetric monoidal functor,
and 
the tensor product of any two cartesian morphisms in $\calF$
is again cartesian.
\end{definition}

\begin{definition}
Suppose 
$F \colon \calF \to \calE$ and $G \colon \calG\to \calE$ 
are symmetric monoidal fibrations. Then by a
\emph{symmetric monoidal morphism of fibrations from} 
$F$ \emph{to} $G$ \emph{over} $\calE$ 
we mean a symmetric monoidal functor 
$K\colon \calF \to \calG$
which preserves cartesian morphisms and makes the triangle
\[\xymatrix{
	\calF
	\ar[dr]_F
	\ar[rr]^K
	&&
	\calG
	\ar[dl]^G
	\\
	&\calE
}\]
of symmetric monoidal functors commutative.
\end{definition}

\begin{definition}
Given symmetric monoidal fibrations $F \colon \calF \to \calE$
and $G \colon \calG \to \calE$ and 
two symmetric monoidal morphisms of fibrations
\[	
	K,L \colon F \longto G
\]
over $\calE$, a 
\emph{symmetric monoidal transformation of fibrations from}
$K$ \emph{to} $L$ \emph{over} $\calE$ 
is a symmetric monoidal natural transformation
$\eta \colon K \Rightarrow L$ such that the natural transformation
$G\eta \colon F \Rightarrow F$ is the identity.
\end{definition}

Symmetric monoidal fibrations with target $\calE$,
symmetric monoidal morphisms of fibrations over $\calE$, 
and symmetric monoidal transformations of fibrations over $\calE$
assemble into a 2-category $\Fib_\tensor(\calE)$.

\begin{example}[The external disjoint union $\usqcup$]
\label{ex:sm-fibration-examples}
The fibrations
$\calS \to \calU$ and $\calH \to \calU$
become symmetric monoidal fibrations if we equip
$\calU$ with the symmetric monoidal structure given by 
direct product of spaces and 
we equip $\calS$ and $\calH$ with the symmetric monoidal
structure given by external disjoint union.
By definition, for parameterized spaces $X \to B$ and $Y \to C$,
the \emph{external disjoint union} of $X$ and $Y$ is the 
space
\[
	X \usqcup Y = (X \times C) \sqcup (B \times Y)
\]
over $B \times C$. The neutral object for the
symmetric monoidal structure on $\calS$
given by $\usqcup$ 
is the empty space over the one-point space.
The external disjoint union of objects of $\calH$ is
defined similarly.
\end{example}

\begin{example}[The direct product $\times$]
\label{ex:sm-calS-with-times}
Instead of the symmetric monoidal structure given by 
external disjoint union, we may equip $\calS$ with the 
symmetric monoidal structure given by the direct product
\[
	(X\to B)\times (Y\to C) = (X\times Y \to B \times C).
\]
This symmetric monoidal structure also makes $\calS \to \calU$
into a symmetric monoidal fibration. In a similar way, 
we may turn $\widehat\calS \to \calB$ into a symmetric monoidal
fibration by equipping $\widehat\calS$ with the direct product.
\end{example}

\subsection{The fibrewise opposite of a fibred category}
\label{subsec:fibrewise-opposite}
The construction of our HHGFT operations involves an assignment 
\begin{equation}
\label{assignment:BGX}
	X \mapsto BG^X
\end{equation}
where $X$ is a space fibred over a good base space $B$
and $BG^X$ denotes the fibrewise mapping space 
$\Map_B(X,BG\times B)$, a space over $B$ 
whose fibre over $b\in B$
is the space of maps from $X_b$ to $BG$. Observe that 
this assignment is contravariantly functorial with respect to vertical
morphisms of $\calS$ and covariantly functorial with respect to 
cartesian morphisms of $\calS$. Thus the above assignment
defines a functor on neither $\calS$ nor $\calS^\op$.
In this subsection, we will
discuss the construction of the fibrewise opposite $\calF^\fop$ of a 
fibred category $\calF$. This construction
has the property that the assignment 
\eqref{assignment:BGX} extends to a functor $\calS^\fop \to \calS$.

\begin{definition}
 Suppose $F\colon \calF \to \calE$ is a fibration. Then by the  
\emph{fibrewise opposite category} of $\calF$ we mean the 
category $\calF^\fop$ defined as follows.  The objects of 
$\calF^\fop$ are the objects of $\calF$, and the morphisms
in $\calF^\fop$ from $X$ to $Y$ are equivalence classes
of diagrams
\[
	X \xot{\ \alpha\ } U \xto{\ \beta\ } Y
\]
where $\alpha$ is vertical and $\beta$ is cartesian.
Two such diagrams
\[
	X \xot{\ \alpha_1\ } U_1 \xto{\ \beta_1\ } Y
	\quad \text{ and } \quad
	X \xot{\ \alpha_2\ } U_2 \xto{\ \beta_2\ } Y
\]
are equivalent
if there exists a (necessarily unique) vertical isomorphism
$U_1 \to U_2$ 
fitting into a commutative diagram
\[\xymatrix@R-4ex{
	& 
	U_1
	\ar[dd]^\isom
	\ar[dl]_{\alpha_1}
	\ar[dr]^{\beta_1}
	\\
	X
	&&
	Y.
	\\
	&
	U_2
	\ar[ul]^{\alpha_2}
	\ar[ur]_{\beta_2}
}\]
The composite of the morphisms 
$[X \xot{\alpha} U \xto{\beta} Y]$ and $[Y \xot{\gamma} V \xto{\delta} Z]$
is the morphism represented by the composites along the
two sides of the diagram
\[\xymatrix{
&&
W
\ar[dl]_{\tilde\gamma}
\ar[dr]^{\tilde\beta}
\\
&
U
\ar[dl]_{\alpha}
\ar[dr]^{\beta}
&&
V
\ar[dl]_\gamma
\ar[dr]^\delta
\\
X
&&
Y
&&	
Z
}\]
where $\tilde\beta$ is a cartesian morphism covering $F(\beta)$ 
and $\tilde\gamma$ is obtained from $\gamma$ by base change.

Defining $F^\fop \colon \calF^\fop \to \calE$ by 
setting $F^\fop(X) = F(X)$
on objects and 
\[
	F^\fop[X \gets U \xto{\beta}Y] = F(\beta)
\]
on morphisms, we obtain a new fibration 
$F^\fop \colon \calF^\fop \to \calE$, the 
\emph{fibrewise opposite fibration} of $F \colon \calF \to \calE$.
The morphism represented by a diagram $X\xot{\alpha} U \to Y$
is cartesian precisely when $\alpha$ is an isomorphism.
\end{definition}

\begin{remark}
The fibre $(\calF^\fop)_B$ over an object $B$ of $\calE$ is 
naturally isomorphic to the opposite category of the fibre $\calF_B$,
explaining the name ``fibrewise opposite category'' for $\calF^\fop$.
\end{remark}

Our next goal is to explain how the assignment 
$\calF \mapsto \calF^\fop$ extends to a 2-functor
\[
	(-)^\fop \colon \Fib_\tensor(\calE) \longto \Fib_\tensor(\calE)
\]
which reverses the direction of 2-morphisms.

\begin{definition}
\label{def:fop-objects}
If $F \colon \calF \to \calE$ is a symmetric monoidal fibration,
we give $\calF^\fop \to \calE$ the structure of a symmetric monoidal
fibration by equipping $\calF^\fop$ with the following symmetric
monoidal structure:
The tensor product in $\calF^\fop$ agrees on objects with
the tensor product of objects in $\calF$, and is defined on
morphisms by
\[
	[X_1 \xot{\alpha_1} U_1 \xto{\beta_1} Y_1]
		\tensor
	[X_2 \xot{\alpha_2} U_2 \xto{\beta_2} Y_2]	
	=
	[X_1\tensor X_2 
		\xot{\alpha_1\tensor\alpha_2}
	U_1\tensor U_2
		\xto{\beta_1\tensor\beta_2}
	Y_1\tensor Y_2].
\]
The neutral object for the tensor product in $\calF^\fop$ is 
the same as that for the tensor product in $\calF$.
The associativity constraint
\[
	(X\tensor Y) \tensor Z \xto{\,\isom\,} X \tensor (Y\tensor Z)
\]
for $\calF^\fop$ is the morphism represented by the diagram
\[
	(X\tensor Y) \tensor Z
	\xot{\quad\id\quad}
	(X\tensor Y) \tensor Z
	\xto[\isom]{\quad a\quad}
	X \tensor (Y\tensor Z)
\]
and similarly for the left and right unit constraints and
the symmetry constraint. 
\end{definition}
\begin{definition}
\label{def:fop-functors}
Suppose $F \colon \calF \to \calE$ and $G \colon \calG \to \calE$ 
are symmetric monoidal fibrations and $K \colon F \to G$
is a symmetric monoidal morphism of fibrations over $\calE$.
We define the symmetric monoidal morphism of 
fibrations over $\calE$
\[
	K^\fop \colon F^\fop \longto G^\fop
\]
on objects by setting $K^\fop (X) = K(X)$
and on morphisms by setting
\[
	K^\fop[X \xot{\alpha} U \xto{\beta} Y]
	 = 
	[K(X) \xot{K(\alpha)} K(U) \xto{K(\beta)} K(Y)].
\]
The monoidality isomorphism
\[
	K^\fop_\tensor \colon K^\fop(X)\tensor K^\fop(Y) \to K^\fop(X\tensor Y)
\]
is defined by 
\[
	[K(X) \tensor K(Y) 
		\xot{\ \id\ }
	K(X) \tensor K(Y) 
		\xto{\  K_\tensor \ }
	K(X\tensor Y)]
\]
and similarly for the unit isomorphism $K^\fop_I \colon I_\calG \to K^\fop(I_\calF)$. 
\end{definition}
\begin{definition}
\label{def:fop-transformations}
If $F \colon \calF \to \calE$
and $G \colon \calG \to \calE$ are symmetric monoidal fibrations,
$K,L \colon F \to G$
are symmetric monoidal morphisms of fibrations over $\calE$,
and $\eta\colon K \Rightarrow L$ 
is a symmetric monoidal transformation of fibrations
over $\calE$, we define the
symmetric monoidal transformation of fibrations over $\calE$ 
\[
	\eta^\fop \colon L^\fop \Rightarrow K^\fop
\]
by letting the 
component $\eta^\fop_X \colon L^\fop(X) \to K^\fop(X)$ 
of $\eta^\fop$ to be the morphism
\[
	[L(X) \xot{\ \eta_X\ } K(X) \xto{\ \id\ } K(X)].
\]
\end{definition}

\begin{proposition}
Definitions~\ref{def:fop-objects} through \ref{def:fop-transformations} 
give a 2-functor
\[
	(-)^\fop \colon \Fib_\tensor(\calE) \longto \Fib_\tensor(\calE)^\co
\]
where the superscript $(-)^\co$ means that the direction
of the 2-morphisms has been reversed. \qed
\end{proposition}
	
\begin{example}
\label{ex:BG-to-X}
The assignment \eqref{assignment:BGX} determines
a symmetric monoidal morphism of fibrations over $\calU$ 
\[\xymatrix{
 (\calS,\usqcup)^\fop 
 \ar[rr]^{BG^{(-)}}
 \ar[dr]
 &&
 (\calS,\times)
 \ar[dl]
 \\
 &
 \calU
}\]
from the fibrewise opposite of $\calS$ equipped with the 
symmetric monoidal structure given by external disjoint union to 
$\calS$ equipped with the symmetric monoidal structure given 
by direct product.
\end{example}

\subsection{Double categories}
\label{subsec:double-categories}
The goal of this subsection and the next one 
is to lay out the definition of 
symmetric monoidal double category together with the 
concomitant notions of symmetric monoidal double functor  
and symmetric monoidal transformation.
Along the way we also construct a number of symmetric monoidal 
double categories needed in the sequel. In our definitions,
we will follow
Shulman \cite{ShulmanSMBicat} with the major simplification
that we will only consider double categories where both 
the horizontal and vertical compositions are strict.

In this subsection, we discuss
non-monoidal double categories and the accompanying 
notions of double functor and transformation.
Intuitively, a double category is a structure consisting 
of objects, two kinds of 1-morphisms (horizontal and vertical), 
and 2-cells that are shaped like squares. In addition,
there is a composition law for each kind of 1-morphism, and 
two composition laws for 2-cells corresponding to horizontal 
and vertical pasting of the 2-cells. A concise definition is 
as follows.

\begin{definition}
A \emph{double category} is a category internal to categories. 
Thus a double category $\bbD$ consists of the following data:
\begin{itemize}
\item a category $\bbD_0$ (`objects')
\item a category $\bbD_1$ (`morphisms')
\item a functor $U\colon \bbD_0 \to \bbD_1$ (`units')
\item functors $S,T\colon \bbD_1 \to \bbD_0$ (`source and target')
\item a functor $\odot\colon \bbD_1 \times_{\bbD_0} \bbD_1 \to \bbD_1$
(`composition'), 
where the fibre product is formed by
mapping the first copy of $\bbD_1$ to $\bbD_0$ by
$S$ and the second one by $T$.
\end{itemize} 
We will write the value of the functor $U$ 
on an object $A$ and a morphism $f$ as
$U_A$ and $U_f$, respectively.
The above data are subject to the following axioms:
\begin{itemize}
\item $\odot$ is associative
\item $SU_A=TU_A=A$ and $S(U_f)=T(U_f)=f$
\item $U_{TM} \odot M = M\odot U_{SM} = M$ and 
	$U_{T\varphi} \odot \varphi = \varphi\odot U_{S\varphi} = \varphi$
\item $S(M\odot N) = SN$ and $S(\varphi\odot\psi) = S(\psi)$
\item $T(M\odot N) = TM$ and $T(\varphi\odot\psi) = T(\varphi)$
\end{itemize}
for all objects $A$ and morphisms $f$ of $\bbD_0$ and for all 
objects $M$, $N$ and morphisms $\varphi$, $\psi$ of $\bbD_1$
composable with respect to $\odot$.

The \emph{objects} of the double category $\bbD$
are the objects of $\bbD_0$,
and the \emph{vertical 1-morphisms} of $\bbD$ are the morphisms
of $\bbD_0$.
We call $\bbD_0$ the \emph{vertical category} of $\bbD$.
The \emph{horizontal 1-morphisms} of $\bbD$ are the objects
of $\bbD_1$.
We will often draw a horizontal 1-morphism
$M$ with $S(M) = A$ and
$T(M) = B$ as 
\[
A \xhto{\ M\ } B.
\]
The \emph{2-cells} of the double category $\bbD$
are the morphisms of $\bbD_1$.
We will often draw a 2-cell $\varphi$ with domain 
$A \xhto{M} B$, codomain $C\xhto{N} D$ and $S(\varphi)=f$ and 
$T(\varphi) = g$ as
\[\xymatrix{
	A
	\ar@{}[dr]|{\Downarrow \varphi}
	\ar[d]_f
	\ar[r]|-@{|}^M
	&
	B
	\ar[d]^g
	\\
	C
	\ar[r]|-@{|}_N	
	&
	D
}\]
The functor $\odot$ is called
the \emph{horizontal composition functor} of $\bbD$.
The horizontal composition of 2-cells can then be pictured as follows:
\[
\left(\vcenter{\xymatrix{
	B
	\ar@{}[dr]|{\Downarrow \psi}
	\ar[d]_g
	\ar[r]|-@{|}^N
	&
	C
	\ar[d]^h
	\\
	B'
	\ar[r]|-@{|}_{N'}
	&
	C'
}}\right)
\odot
\left(\vcenter{\xymatrix{
	A
	\ar@{}[dr]|{\Downarrow \varphi}
	\ar[d]_f
	\ar[r]|-@{|}^M
	&
	B
	\ar[d]^g
	\\
	A'
	\ar[r]|-@{|}_{M'}
	&
	B'
}}\right)
=
\left(\vcenter{\xymatrix{
	A
	\ar@{}[dr]|{\Downarrow \psi \odot \varphi}
	\ar[d]_f
	\ar[r]|-@{|}^{N\odot M}
	&
	C
	\ar[d]^h
	\\
	A'
	\ar[r]|-@{|}_{N'\odot M'}
	&
	C'
}}\right)
\]
\emph{Vertical composition} of vertical 1-morphisms and 2-cells
is given by composition in $\bbD_0$ and $\bbD_1$,
respectively.
\end{definition}

\begin{example}
\label{ex:sqcat}
Most of our double categories
will be constructed from the 
\emph{double category of squares} $\sq(\calC)$
associated to some ordinary category $\calC$.
The vertical category $\sq(\calC)_0$ of $\sq(\calC)$
is just $\calC$, while $\sq(\calC)_1$ 
is the category 
of morphisms in $\calC$, that is, the
functor category $\calC^{[1]}$ where 
$[1]$ denotes the poset $\{0<1\}$.
Thus the objects of $\sq(\calC)$ are the objects of $\calC$,
the vertical and horizontal 1-morphisms are both the 
morphisms of $\calC$, and the 2-cells are 
the commutative squares in $\calC$.
The horizontal composition $\odot$ is given on horizontal 1-morphisms
by composition in $\calC$ and on 2-cells by pasting commutative squares horizontally:
\[
\left(\vcenter{\xymatrix{
	B
	\ar[d]_g
	\ar[r]^v
	&
	C
	\ar[d]^h
	\\
	B'
	\ar[r]_{v'}
	&
	C'
}}\right)
\odot
\left(\vcenter{\xymatrix{
	A
	\ar[d]_f
	\ar[r]^u
	&
	B
	\ar[d]^g
	\\
	A'
	\ar[r]_{u'}
	&
	B'
}}\right)
=
\left(\vcenter{\xymatrix{
	A
	\ar[d]_f
	\ar[r]^{vu}
	&
	C
	\ar[d]^h
	\\
	A'
	\ar[r]_{v'u'}
	&
	C'
}}\right)
\]
Notice that $\sq(\calC)$ has the special property that 
a 2-cell is determined by its bounding 1-morphisms.
\end{example}

The following double category will later serve as the domain
of definition for our theory of umkehr maps
for fibrewise mapping spaces.

\begin{example}[The double category $\bbS^d(\calH^\fop)$]
\label{ex:dc-special-squares-h-graphs}
Let $d\in \Z$. 
We will construct \emph{the double category of special squares
in} $\calH^\fop$ \emph{with degree shifts},
denoted $\bbS^d(\calH^\fop)$,
as a sub-double category of the double category of squares
in a certain ordinary category.
Let $\calB^\ds$ denote the category of pairs 
$(B,m)$ where $B$ is an object of $\calB$ and $m\colon B \to \Z$
is a locally constant function. We refer to such an $m$ as
a \emph{degree shift}. A morphism in $\calB^\ds$ from $(B,m)$ to 
$(B',m')$ is simply a continuous map $B \to B'$; the degree
shifts $m$ and $m'$ do not play a role. Consider now the
fibre product $\calH^\fop\times_\calB \calB^\ds$. 
The objects of this fibre product can be identified with pairs
$(X,n)$, where $X$ is an object of $\calH$ over a good
base space $B$ and $n$ is a locally constant function $B \to \Z$. 
We call the objects of $\calH^\fop\times_\calB \calB^\ds$
\emph{families of h-graphs with degree shifts}.  A morphism 
in $\calH^\fop\times_\calB \calB^\ds$ from a family 
$(X,m)$ over $B$ to another family $(Y,n)$ over $C$
is a morphism in $\calH^\fop$ from $X$ to $Y$, that is,
an equivalence class of commutative diagrams of the form
\begin{equation}
\label{zigzag:morphism-representative}
\vcenter{\xymatrix@C-1em{
	X
	\ar[dr]
	&& 
	U 
	\ar[dl]
	\ar[ll]_\alpha
	\ar[rr]^\beta
	&&
	Y
	\ar[d]
	\\
	&
	B
	\ar[rrr]^f
	&&&
	C
}}
\end{equation}
where $U \to B$ is an object of $\calH$ and 
the map $\beta$ is cartesian over $f$.
We now let $\bbS^d(\calH^\fop)$ be the sub-double category of 
$\sq(\calH^\fop\times_\calB\calB^\ds)$ defined as follows.
\begin{itemize}
\item
The objects of $\bbS^d(\calH^\fop)$ are simply the 
objects of $\calH^\fop\times_\calB\calB^\ds$.
\item
A vertical morphism in $\bbS^d(\calH^\fop)$ 
is a morphism $(X,m)\to (Y,n)$ in $\calH^\fop\times_\calB\calB^\ds$ 
represented by a diagram of the form
\eqref{zigzag:morphism-representative}
in which $\alpha$ is an h-embedding
of families of h-graphs and $m$ and $n$ satisfy $m = n\circ f$.
\item
A horizontal morphism in $\bbS^d(\calH^\fop)$ is a morphism
$(X,m)\to (Y,n)$ in $\calH^\fop\times_\calB\calB^\ds$ 
represented by a diagram of the form
\eqref{zigzag:morphism-representative}
in which $\beta$ and $f$ are identity maps, 
$\alpha$ is a positive map of families of h-graphs, 
and $m$ and $n$ satisfy
\[
	m-n = -d(\chi(X)-\chi(Y)).
\]
Here $\chi(X) \colon B \to \Z$ is the locally constant
function whose value at a point $b \in B$ is the Euler
characteristic of the fibre $X_b$, and similarly for $\chi(Y)$.
\item
A 2-cell in $\bbS^d(\calH^\fop)$ is a 2-cell
\[\xymatrix{
	(X,m) 
	\ar[r]
	\ar[d]
	&
	(Y,n)
	\ar[d]
	\\
	(Z,k)
	\ar[r]
	&
	(W,l)
}\]
in $\sq(\calH^\fop\times_\calB\calB^\ds)$ 
whose vertical and horizontal morphisms satisfy the respective
above conditions and in which the
underlying commutative diagram of spaces
\[\xymatrix@R-2.6ex{
	X
	\ar[dr]
	&&
	Y
	\ar[ll]
	\ar[dl]
	\\
	& 
	B
	\ar[dd]^f
	\\
	f^\ast Z
	\ar[uu]
	\ar[ru]
	\ar[dd]
	&&
	f^\ast W
	\ar[uu]
	\ar[lu]
	\ar[dd]
	\\
	& C
	\\
	Z
	\ar[ur]
	&&
	W
	\ar[ll]
	\ar[ul]
}\]
is such that, for each $b \in B$, the square
\[\xymatrix{
	X_b
	&
	Y_b	
	\ar[l]
	\\
	Z_{f(b)}
	\ar[u]
	&
	W_{f(b)}
	\ar[u]
	\ar[l]
}\]
of fibres is a homotopy cofibre square. 
\end{itemize}
It is now easily
verified that $\bbS^d(\calH^\fop)$ does inherit 
a double category structure from 
$\sq(\calH^\fop\times_\calB\calB^\ds)$.
(Later in Example~\ref{ex:sm-dc-special-squares-h-graphs}
we will equip $\bbS^d(\calH^\fop)$ with the structure of a symmetric
monoidal double category.)
\end{example}

\begin{definition}\label{def:double-functor}
Suppose $\bbD$ and $\bbE$ are double categories.
A \emph{double functor} $F\colon \bbD \to \bbE$ 
is a functor $\bbD \to \bbE$ in the sense of internal 
category theory. Thus $F$ consists of 
\begin{itemize}
\item a functor $F_0\colon \bbD_0 \to \bbE_0$
\item a functor $F_1\colon \bbD_1 \to \bbE_1$
\end{itemize}
such that
\begin{itemize}
\item $SF_1 = F_0S$ and $TF_1 = F_0T$
\item $F_1 M \odot F_1 N = F_1(M\odot N)$ and 
	$F_1 \varphi \odot F_1 \psi = F_1(\varphi\odot \psi)$
	for all composable horizontal 1-morphisms $M$ and $N$ and 
	2-cells $\varphi$ and $\psi$
\item $U F_0 = F_1 U$.
\end{itemize}
\end{definition}

\begin{example}
\label{ex:sqfun}
Suppose $F\colon \calC \to \calD$ is a functor between ordinary 
categories.  Then there is a double functor
\[
	\sq(F)\colon \sq(\calC) \longto \sq(\calD)
\]
where $\sq(F)_0\colon \calC \to \calD$ is $F$ and 
$\sq(F)_1$ is the functor 
$\calC^{[1]} \to \calD^{[1]}$
induced by $F$.
\end{example}

\begin{definition}
Suppose $F,G\colon \bbD \to \bbE$ are functors of double categories.
A (vertical) \emph{transformation} $\alpha \colon F \to G$ 
consists of 
\begin{itemize}
\item a natural transformation $\alpha_0\colon F_0 \to G_0$
\item a natural transformation $\alpha_1\colon F_1 \to G_1$
\end{itemize}
(both of which we will usually denote by $\alpha$) satisfying
the following compatibility conditions:
\begin{itemize}
\item $S(\alpha_M) = \alpha_{SM}$ and $T(\alpha_M) = \alpha_{TM}$
\item $\alpha_{N \odot M} = \alpha_N \odot \alpha_M$
\item $\alpha_{U_A} = U_{\alpha_A}$.
\end{itemize}
for all composable horizontal 1-morphisms $M$ and $N$
and all objects $A$.
Notice that this notion of transformation 
does \emph{not} agree with the 
usual notion of natural transformation one obtains from 
internal category theory.
\end{definition}

\begin{example}
\label{ex:sqtransf}
Suppose $F,G\colon \calC \to \calD$ are ordinary functors,
and suppose $\alpha \colon F \to G$ is an ordinary
natural transformation.
Then there is a transformation 
\[
	\sq(\alpha) \colon \sq(F) \longto \sq(G)
\]
with 
\[
	\sq(\alpha)_0 = \alpha \colon F \longto G.
\]
The component of $\sq(\alpha)_1$ associated to a
horizontal 1-morphism $u \colon A \to B$ in $\sq(\calC)$
is the square
\[\vcenter{\xymatrix{
	FA
	\ar[d]_{\alpha_A}
	\ar[r]^{Fu}
	&
	FB
	\ar[d]^{\alpha_B}
	\\
	GA
	\ar[r]_{Gu}
	&
	GB
}}\]
\end{example}

\subsection{Symmetric monoidal double categories}
\label{subsec:symmetric-monoidal-double-categories}

Double categories, double functors and transformations assemble
into a 2-category $\Dbl$. Guided by the case of ordinary 
symmetric monoidal categories, which can be characterized as 
the symmetric pseudomonoids in the 2-category $\Cat$ of
ordinary categories, we make the following 
definition.

\begin{definition}
A \emph{symmetric monoidal double category} is a symmetric
pseudomonoid in the 2-category of double categories. 
This means that a symmetric monoidal double category 
$\bbD = (\bbD,\tensor,I,a,l,r,s)$
consists of:
\begin{itemize}
\item a double category $\bbD$
\item a double functor $\tensor\colon \bbD \times \bbD \to \bbD$
\item a double functor $I\colon * \to \bbD$
\item invertible transformations
	\begin{itemize}
 	\item $a\colon \mathord{\tensor} \circ (\mathord{\tensor} \times \Id)
	        \xto{\,\isom\,}
	        \mathord{\tensor} \circ (\Id \times \mathord{\tensor})$
	\item $l\colon \mathord{\tensor} \circ (I\times \Id) \xto{\,\isom\,} \Id$
	\item $r\colon \mathord{\tensor} \circ (\Id\times I) \xto{\,\isom\,} \Id$
	\item $s\colon \mathord{\tensor} \xto{\,\isom\,} \mathord{\tensor}\circ\tau$,
	where $\tau\colon \bbD \times \bbD \xto{\,\isom\,} \bbD\times \bbD$ 
	interchanges the two coordinates
	\end{itemize}
\end{itemize}
satisfying the usual axioms -- 
see eg.\ Mac~Lane \cite[section XI.1]{MacLane}.
\end{definition}

Unwinding definitions, the axioms alluded to in the above definition
amount simply to requirement that 
$\bbD_0 = (\bbD_0,\tensor_0,I_0,a_0,l_0,r_0,s_0)$
and 
$\bbD_1 = (\bbD_1,\tensor_1,I_1,a_1,l_1,r_1,s_1)$
be symmetric monoidal categories.
 
\begin{definition}\label{def:symmetric-monoidal-double-functor}
Given symmetric monoidal double categories $\bbD$ and $\bbD'$,
a \emph{symmetric monoidal double functor} 
$F \colon \bbD \to \bbD'$ consists of
\begin{itemize}
\item a double functor $F\colon \bbD \to \bbD'$
\item invertible transformations
	\begin{itemize}
	\item $F_\tensor \colon \mathord{\tensor}' \circ (F\times F) 
		   \xto{\,\isom\,}
		   F\circ \mathord{\tensor}$
	\item $F_I \colon I' \xto{\,\isom\,} F\circ I$
	\end{itemize}
\end{itemize}
such that 
$F_0 = (F_0,F_{\tensor,0},F_{I,0})$
and
$F_1 = (F_1,F_{\tensor,1},F_{I,1})$
are symmetric monoidal functors.
A symmetric monoidal double functor $F$ is called \emph{strict} 
if $F_\tensor$ and $F_I$ are 
identities instead of isomorphisms. 
\end{definition}

\begin{definition}
Let $F=(F,F_\tensor,F_I)$ and 
$G=(G,G_\tensor,G_I)$
be symmetric monoidal double functors $\bbD \to \bbD'$
between symmetric monoidal double categories. 
A \emph{symmetric monoidal transformation} $\theta\colon F\to G$
is a transformation $\theta\colon F\to G$ such that
$\theta_0$ and $\theta_1$ are symmetric monoidal natural transformations.
\end{definition}

Symmetric monoidal double categories, 
symmetric monoidal double functors, 
and symmetric monoidal transformations 
assemble into a 2-category $\Dbl_\tensor$. 
Examples \ref{ex:sqcat}, \ref{ex:sqfun}
and \ref{ex:sqtransf} give a 2-functor
\[
	\sq \colon \Cat \longto \Dbl
\]
from the 2-category $\Cat$ of ordinary categories to $\Dbl$,
and it is readily verified that this 2-functor induces 
a 2-functor
\[
	\sq \colon \Cat_\tensor \longto \Dbl_\tensor
\]
from the 2-category $\Cat_\tensor$ of 
ordinary symmetric monoidal categories,
symmetric monoidal functors and 
symmetric monoidal natural transformations 
to $\Dbl_\tensor$.

\begin{example}[The symmetric monoidal 
structure on $\bbS^d(\calH^\fop)$]
\label{ex:sm-dc-special-squares-h-graphs}
Consider the sub-double category 
$\bbS^d(\calH^\fop)$ of $\sq(\calH^\fop\times_\calB \calB^\ds)$
constructed in Example~\ref{ex:dc-special-squares-h-graphs}. 
The category $\calB^\ds$ has a symmetric monoidal structure given by 
\[
	(B,m)\times (C,n) = (B\times C,m+n)
\]
where the function
\[
	m+n \colon B \times C \longto \Z
\]
is obtained from the functions
$m\colon B \to \Z$ and $n \colon C \to \Z$
by adding together the composites
of $m$ and $n$ with the appropriate projection maps from $B\times C$.
Together with the symmetric monoidal structure on $\calH^\fop$,
this gives a symmetric monoidal structure on the fibre product
$\calH^\fop\times_\calB \calB^\ds$.
Concretely, the 
symmetric monoidal structure on $\calH^\fop\times_\calB \calB^\ds$
is given by 
\[
	(X,m) \usqcup (Y,n) = (X\usqcup Y, m+n).
\]
The symmetric monoidal structure on 
$\calH^\fop\times_\calB \calB^\ds$ gives
$\sq(\calH^\fop\times_\calB \calB^\ds)$ the structure of 
a symmetric monoidal double category, and the sub-double category
$\bbS^d(\calH^\fop)$ of $\sq(\calH^\fop\times_\calB \calB^\ds)$
inherits from $\sq(\calH^\fop\times_\calB \calB^\ds)$
a symmetric monoidal structure making the inclusion
\[
	\bbS^d(\calH^\fop) \longincl \sq(\calH^\fop\times_\calB \calB^\ds)
\]
into a strict symmetric monoidal double functor.
\end{example}

\begin{definition}
Given a double category $\bbD$, its \emph{horizontal opposite}
$\bbD^\hop$ is the double category with the same objects, 
horizontal 1-morphism, vertical 1-morphisms and 2-cells as $\bbD$, 
but with the direction of horizontal morphisms reversed. 
If $\bbD$ is symmetric monoidal, so is $\bbD^\mathrm{hop}$.
\end{definition}

%% file: the-push-pull-construction.tex

\section{The push-pull construction}
\label{sec:push-pull}

Let $d\in\Z$.
This section will explain how to obtain a homological
h-graph field theory of degree $d$ from the
data of a symmetric monoidal double functor
$U\colon \bbS^d(\calH^\fop)\to\sq(\grmod)^\hop$.
The functor $U$ encodes
all of the data required to perform the push-pull construction
discussed in section~\ref{overview:section}:
its effect on objects encodes the assignment
$X\mapsto H_\ast(BG^X)$ for $X$ a family of h-graphs,
while its effects on vertical and horizontal morphisms encode
the induced maps $(j^\ast)_\ast$ and 
umkehr maps $i_!$ appearing in~\eqref{composite:eq},
respectively.
The fact that $U$ is a symmetric monoidal double functor records all
of the properties required to obtain a positive HHGFT,
save for one property which we now define.

\begin{definition}\label{def:emptyfamilies}
Let $U\colon \bbS^d(\calH^\fop)\to\sq(\grmod)^\hop$ be a
symmetric monoidal double functor.
Let $\calU^\ds$ denote the symmetric monoidal category
in which an object is a pair $(B,m)$ with $B$ a good base space
and $m\colon B\to\Z$ a locally constant map, and in which a morphism
$(B,m)\to (C,n)$ is a continuous map $f\colon B\to C$ satisfying
$n\circ f=m$.
The symmetric monoidal structure is given by
$(B,m)\times (C,n) = (B\times C, m+n)$.
The assignment $(B,m)\mapsto (\emptyset_B,m)$
extends to a strict symmetric monoidal functor
$\emptyset\colon \calU^\ds\to \bbS^d(\calH^\fop)_0$.
We say that $U$ is \emph{homological} if the composite
\[
	\calU^\ds
	\xrightarrow{\ \ \emptyset\ \ }
	\bbS^d(\calH^\fop)_0
	\xrightarrow{\ \ U_0\ \ }
	\grmod
\]
is given by $(B,m)\mapsto H_{\ast-m}(B)$.
\end{definition}

\begin{theorem}\label{th:push-pull}
Given a homological symmetric monoidal double functor
\[U\colon \bbS^d(\calH^\fop)\longto\sq(\grmod)^\hop,\]
there is a positive homological h-graph field theory of degree $d$ whose
value on an h-graph $X$ is $U(X,0)$.
\end{theorem}

The rest of the section is given to the proof of this result.
We begin in subsection~\ref{subsection:cw} by reducing the problem
to one of constructing an HHGFT in which the base spaces are
CW complexes.
After making this reduction, we proceed to prove the theorem by construction,
defining the data for the required HHGFT in subsection~\ref{subsection:data}
and verifying that these data satisfy the required axioms
in subsection~\ref{subsection:proof}.

\subsection{Homological h-graph field theories defined on CW complexes}
\label{subsection:cw}

In the definition of HHGFT (Definition~\ref{def:HHGFT})
we worked with arbitrary base spaces,
but in a sense it is enough to consider base spaces which 
are CW complexes. By a \emph{(positive) HHGFT defined on CW complexes}
we mean a structure exactly analogous to a (positive) HHGFT, except
that only families of h-graph cobordisms over base spaces
admitting the structure of a CW complex are considered.
We then have the following proposition.

\begin{proposition}
Every HHGFT defined on CW complexes extends uniquely 
to an HHGFT, and similarly for positive HHGFTs.
\end{proposition}
\begin{proof}
Suppose $\Phi$ is an HHGFT defined on CW complexes. 
Let $\Gamma$ be  the functor sending a space $B$ to the
geometric realization of the singular simplicial set of $B$.
Then for each $B$, the space $\Gamma B$ has a natural CW structure,
and there is a natural weak equivalence $\gamma\colon \Gamma B \to B$.
Define an HHGFT $\tilde\Phi$ by taking the symmetric
monoidal functor $\tilde\Phi_\ast$ to be $\Phi_\ast$ and 
by defining $\tilde\Phi(S/B)$ for a family of 
h-graph cobordisms $S/B \colon X \to Y$ to be the unique map 
making the diagram
\[\xymatrix@C+2em{
	H_{\ast-d\cdot \chi(\gamma^\ast S,X)}(\Gamma B) \tensor \Phi_\ast(X)
	\ar[d]_{\gamma_\ast \tensor \id}^{\isom}
	\ar[r]^-{\Phi(\gamma^\ast S/\Gamma B)}
	&
	\Phi_\ast(Y)
	\ar[d]^{\id}
	\\
	H_{\ast-d\cdot \chi(S,X)}(B)\tensor \Phi_\ast(X)
	\ar@{-->}[r]^-{\tilde\Phi(S/B)}
	&
	\Phi_\ast(Y)
}\]
commutative. It is then readily verified that $\tilde\Phi$
satisfies the axioms of an HHGFT, and the 
base change axiom implies that $\tilde\Phi(S/B)$ agrees with
$\Phi(S/B)$ when $B$ admits the structure of a CW complex.
Thus $\tilde\Phi$ is an extension of $\Phi$. Moreover,
by the base change axiom, any extension $\tilde\Phi$ of 
$\Phi$ into an HHGFT must make the above 
square commutative. Thus the extension is unique.
The claim for positive HHGFTs is proved
in a similar way.
\end{proof}

\subsection{Defining the HHGFT}\label{subsection:data}
Let $U\colon \bbS^d(\calH^\fop)\to\sq(\grmod)^\hop$
be a homological symmetric monoidal double functor.
We will prove Theorem~\ref{th:push-pull} by
using $U$ to explicitly construct a positive
HHGFT $\Phi$ defined on CW complexes.
Recall from Definition~\ref{def:HHGFT} that this means
we must specify the following data.
\begin{itemize}
	\item A symmetric monoidal functor $\Phi_\ast$
 	from the category of h-graphs and homotopy equivalences among them
	into the category of graded $\bbF$-vector spaces.
	\item For each positive family $S/B\colon X\hto Y$ of h-graph 
	cobordisms over a space $B$ admitting the structure
	of a CW complex,
	a map
	\[
		\Phi(S/B)\colon
		H_{\ast-d\cdot\chi(S,X)}(B)\otimes\Phi_\ast(X)
		\longto
		\Phi_\ast(Y).
	\]
\end{itemize}
Moreover, we must check that these data satisfy the axioms of an HHGFT.
We begin in this subsection by defining the functor $\Phi_\ast$
and the operations $\Phi(S/B)$.
First we require the following lemma.

\begin{lemma}\label{inverse:lemma}
Let $f\colon X\to Y$ be a homotopy equivalence of h-graphs,
regarded as a vertical morphism $f\colon (Y,0)\to (X,0)$ in $\bbS^d(\calH^\fop)$.
Then $U(f)\colon U(Y,0)\to U(X,0)$ is an isomorphism.
\end{lemma}

\begin{proof}
It suffices to show that $U(f_0)=U(f_1)$
for homotopic homotopy equivalences $f_0,f_1\colon X\to Y$.
Let $F\colon X\times I\to Y$ be a homotopy from $f_0$ to $f_1$.
Consider the following vertical morphisms.
\begin{equation}\label{composites:equation}
	(Y,0)
	\isom
	(Y,0)\otimes(\emptyset_\mathrm{pt},0)
	\xrightarrow{\,1\otimes i_0,\,1\otimes i_1\,}
	(Y,0)\otimes (\emptyset_I,0)
	\xrightarrow{\ G\ }
	(X,0)
\end{equation}
Here $i_0$ and $i_1$ correspond to 
the morphisms $\mathrm{pt}\to I$ with image $\{0\}$ and $\{1\}$ respectively,
and $G$ corresponds to the zig-zag
\[
	Y\times I\xleftarrow{\ H\ }X\times I \xrightarrow{\ \mathrm{proj}\ } X
\]
with $H(x,t)=(F(x,t),t)$ for $x\in X$ and $t\in I$.
Then the two composites in \eqref{composites:equation}
are exactly $f_0$ and $f_1$, so it suffices to show that $U(i_0)=U(i_1)$.
But that follows from the fact that $U$ is homological.
\end{proof}

\begin{definition}\label{phiast:definition}
Let $\Phi_\ast$ be the symmetric monoidal functor from the 
category of h-graphs and homotopy equivalences into $\grmod$ 
obtained as follows.
For an h-graph $X$ we set
\[\Phi_\ast(X)=U(X,0)\]
where on the right $X$ is regarded as a family of h-graphs over a point,
and for a homotopy equivalence of h-graphs $f\colon X\to Y$ we set
\[\Phi_\ast(f)=U(f)^{-1}\]
where on the right $f$ is regarded as a vertical morphism
$f\colon (Y,0)\to (X,0)$ in $\bbS^d(\calH^\fop)$.
(Since $f$ is a homotopy equivalence, it is an h-embedding,
and Lemma~\ref{inverse:lemma} guarantees that $U(f)$ is invertible.)
We make $\Phi_\ast$ symmetric monoidal
by letting $(\Phi_\ast)_\otimes$
and $(\Phi_\ast)_I$ be given by $U_{\otimes,0}$ and $U_{I,0}$,
respectively.
\end{definition}

\begin{definition}\label{phisb:definition}
Let $S/B\colon X\hto Y$ be a family of positive h-graph cobordisms from an h-graph 
$X$ to an h-graph $Y$.
We define
\[
		\Phi(S/B)\colon
		H_{\ast-d\cdot\chi(S,X)}(B)\otimes\Phi_\ast(X)
		\longrightarrow
		\Phi_\ast(Y)
\]
as follows.
Consider the morphisms
\[\xymatrix{
	{}
	&
	(\emptyset_B,d\chi(S,X))\otimes (X,0)
	\ar[d]^{\alpha_S}
	\\
	(S,0)
	\ar[r]^-{\beta_S}
	\ar[d]_{\gamma_S}
	&
	(\pi_B^\ast X,d\chi(S,X))
	\\
	(Y,0)
	&
	{}
}\]
in $\bbS^d(\calH^\fop)$, where:
\begin{itemize}
	\item
	$\alpha_S$ is the vertical morphism corresponding
	to the cartesian morphism
	$\emptyset_B\usqcup X\to\pi_B^\ast X$ of $\calH$
	in which $\pi_B\colon B\to \pt$ denotes the constant map;
	\item
	$\beta_S$ is the horizontal morphism corresponding
	to the positive morphism
	$S\leftarrow\pi^\ast_BX$ in $\calH$;
	\item
	$\gamma_S$ is the vertical morphism corresponding to the zig-zag
	$S\leftarrow\pi_B^\ast Y\rightarrow Y$
	in $\calH$,
	where the left arrow is an h-embedding and the right arrow is cartesian.
\end{itemize}
Applying the symmetric monoidal functor $U$ to $\alpha_S$, $\beta_S$ and $\gamma_S$
gives a diagram of the same shape in $\sq(\grmod)^\hop$, which can be regarded as a 
sequence of three composable morphisms in $\grmod$.
Now we define $\Phi(S/B)$ to be the composite
\[
\xymatrix@R=1ex@C=-2em{
	 *+[l]{	H_{\ast-d\cdot\chi(S,X)}(B)\otimes U(X,0)}
	 \ar[r]^-{=}
	 &
	 *+[r]{	U(\emptyset_B,d\chi({S,X}))\otimes U(X,0)}
	 \\
	 *+[l]{\phantom{	H_{\ast+d\cdot\chi(S,X)}(B)\otimes U(X,0)}}
	 \ar[r]^-{U_\otimes}
	 &
	 *+[r]{U((\emptyset_B,d\chi(S,X))\otimes(X,0))}
	 \\
	 *+[l]{\phantom{	H_{\ast+d\cdot\chi(S,X)}(B)\otimes U(X,0)}}
	 \ar[r]^-{U(\gamma_S)U(\beta_S)U(\alpha_S)}
	 &
	 *+[r]{U(Y,0).}
}\]
\end{definition}

\subsection{Verifying the axioms}
\label{subsection:proof}
To complete the proof of Theorem~\ref{th:push-pull},
we must verify that the data defined above satisfy
the base change, gluing, identity and monoidality axioms of
Definition~\ref{def:HHGFT}.
In each case the verifications are lengthy, but not difficult,
and follow the same general pattern:
\begin{itemize}
	\item[(i)] Restate the axiom in terms of $U$ alone.
	\item[(ii)] Use monoidality of $U$ to reduce the problem to one taking
	place in $\bbS^d(\calH^\fop)$.
	\item[(iii)] Solve this final problem.
\end{itemize}
We will therefore restrict ourselves to proving the gluing axiom.
The remaining axioms can be proved by much the same method,
though the details are simpler,  and we leave the verifications to the reader.

Let us prove the gluing axiom.
Let $X\xhto{S/B} Y\xhto{T/C}Z$
be families of h-graph cobordisms. We are required to prove that the square 
appearing on page~\pageref{gluing},
which involves the morphisms $\Phi(S/B)$, $\Phi(T/C)$, $\Phi(T\ucirc S/C\times B)$
and the homology cross product, commutes.

We begin by rephrasing the problem in terms of $U$.
Replacing terms of the form $H_{\ast-n}(B)$ with $U(\emptyset_B,n)$,
replacing terms of the form $\Phi_\ast(A)$  with $U(A,0)$,
and unwinding the definitions of the morphisms $\Phi(T\ucirc S/C\times B)$,
$\Phi(S/B)$ and $\Phi(T/C)$ in terms of $U$,
commutativity of the square translates into the claim that the composites
\[\xymatrix@1@!0@C=4.5em{
	*!L{U\bigl(\emptyset_C,{d\chi(T,Y)}\bigr)
	\otimes
	U\bigl(\emptyset_B,d\chi(S,X)\bigr)
	\otimes
	U\bigl(X,0\bigr)}
	\\
	&
	\ar[rr]^{U_\otimes \otimes 1}
	&&
	*!L{\;U\bigl((\emptyset_C,d\chi(T,Y))\otimes( \emptyset_B,d\chi(S,X)\bigr)
	\otimes
	U\bigl(X,0\bigr)}
	\\
	&
	\ar[rr]^{U(\emptyset_\otimes)\otimes 1}
	&&
	*!L{\;U\bigl(\emptyset_{C\times B},d\chi(T\ucirc S,X)\bigr)
	\otimes
	U\bigl(X,0\bigr)}
	\\
	&
	\ar[rr]^{U_\otimes}
	&&
	*!L{\;U\bigl((\emptyset_{C\times B},d\chi(T\ucirc S,X))
	\otimes
	(X,0) \bigr)}
	\\
	&
	\ar[rr]^{U(\gamma_{T\ucirc S})U(\beta_{T\ucirc S})U(\alpha_{T\ucirc S})}
	&&
	*!L{\;U\bigl( Z,0 \bigr)}
}\]
and
\[\xymatrix@1@!0@C=4.5em{
	*!L{U\bigl(\emptyset_C,d\chi(T,Y)\bigr)
			\otimes
		U\bigl(\emptyset_B,d\chi(S,X)\bigr)
			\otimes
		U\bigl(X,0\bigr)}
	\\
	&
	\ar[rr]^{1\otimes U_\otimes}
	&&
	*!L{\;U\bigl(\emptyset_C,d\chi(T,Y)\bigr)
			\otimes
		U\bigl((\emptyset_B,d\chi(S,X)) \otimes (X,0)\bigr)}
	\\
	&
	\ar[rr]^{1\otimes U(\gamma_{S})U(\beta_{S})U(\alpha_{S})}
	&&
	*!L{\;U\bigl(\emptyset_C,d\chi(T,Y)\bigr)
			\otimes
		U\bigl(Y,0\bigr)}
	\\
	&
	\ar[rr]^{U_\otimes}
	&&
	*!L{\;U\bigl((\emptyset_C,d\chi(T,Y)) \otimes (Y,0)\bigr)}
	\\
	&
	\ar[rr]^{U(\gamma_{T})U(\beta_{T})U(\alpha_{T})}
	&&
	*!L{\;U\bigl(Z,0\bigr)}
}\]
coincide.
(Here we are writing
 $\emptyset_\otimes\colon \emptyset_C\otimes\emptyset_B\to \emptyset_{C\times B}$ 
to denote the monoidality isomorphism for the assignment $B\mapsto \emptyset_B$.
In fact $\emptyset_\otimes$ is the identity function,
but it is convenient retain it in the notation.)

Next, we use monoidality of $U$ to replace these with the composites
\[\xymatrix@1@!0@C=4.5em{
	*!L{U\bigl(\emptyset_C,d\chi(T,Y)\bigr)
	\otimes
	U\bigl(\emptyset_B,d\chi(S,X)\bigr)
	\otimes
	U\bigl(X,0\bigr)}
	\\
	&
	\ar[rr]^{U_\otimes\circ(U_\otimes \otimes 1)}
	&&
	*!L{\;U\bigl((\emptyset_C,d\chi(T,Y))
	\otimes
	(\emptyset_B,d\chi(S,X)) \otimes (X,0)\bigr)}
	\\
	&
	\ar[rr]^{U(\emptyset_\otimes\otimes 1)}
	&&
	*!L{\;U\bigl((\emptyset_{C\times B},d\chi(T\ucirc S,X))
	\otimes
	(X,0) \bigr)}
	\\
	&
	\ar[rr]^{U(\gamma_{T\ucirc S})U(\beta_{T\ucirc S})U(\alpha_{T\ucirc S})}
	&&
	*!L{\;U\bigl( Z,0 \bigr)}
}\]
and
\[\xymatrix@1@!0@C=4.5em{
	*!L{U\bigl(\emptyset_C,d\chi(T,Y)\bigr)
	\otimes
	U\bigl(\emptyset_B,d\chi(S,X)\bigr)
	\otimes
	U\bigl(X,0\bigr)}
	\\
	&
	\ar[rr]^{U_\otimes\circ(1\otimes U_\otimes)}
	&&
	*!L{\;
	U\bigl((\emptyset_C,d\chi(T,Y)) \otimes (\emptyset_B,d\chi(S,X)) \otimes (X,0)\bigr)
	}
	\\
	&
	\ar[rr]^{U(1\otimes\gamma_S)U(1\otimes\beta_S)U(1\otimes\alpha_S)}
	&&
	*!L{\;U\bigl((\emptyset_C,d\chi(T,Y))\otimes(Y,0)\bigr)}
	\\
	&
	\ar[rr]^{U(\gamma_{T})U(\beta_{T})U(\alpha_{T})}
	&&
	*!L{\;U\bigl(Z,0\bigr)}
}\]
respectively.
Since $U_\otimes\circ(U_\otimes\otimes 1)=U_{\otimes}\circ(1\otimes U_\otimes)$
it will be sufficient to show that
\[
	U(\gamma_{T\ucirc S})U(\beta_{T\ucirc S})U(\alpha_{T\ucirc S})U(\emptyset_\otimes\otimes 1)
	=
	U(\gamma_{T})U(\beta_{T})U(\alpha_{T})
	U(1\otimes\gamma_S)U(1\otimes\beta_S)U(1\otimes\alpha_S).
\]

The two sides of the last equation are obtained from the diagrams
\begin{equation}\label{stepsone}
\xymatrix{
	{}
	&
	\bullet
	\ar[d]^{\emptyset_\otimes\otimes 1}
	\\
	{}
	&
	\bullet
	\ar[d]^{\alpha_{T\ucirc S}}
	\\
	\bullet
	\ar|-@{|}[r]^{\beta_{T\ucirc S}}
	\ar[d]_{\gamma_{T\ucirc S}}
	&
	\bullet
	\\
	\bullet
	&
	{}
}
\qquad\qquad
\xymatrix{
	{}
	&
	{}
	&
	\bullet
	\ar[d]^{1\otimes\alpha_S}
	\\
	{}
	&
	\bullet
	\ar|-@{|}[r]^{1\otimes\beta_{S}}
	\ar[d]^{1\otimes\gamma_S}
	&
	\bullet
	\\
	{}
	&
	\bullet
	\ar[d]^{\alpha_T}
	&
	{}
	\\
	\bullet
	\ar|-@{|}[r]^{\beta_{T}}
	\ar[d]_{\gamma_T}
	&
	\bullet
	&
	{}
	\\
	\bullet
	&
	{}
	&
	{}
}\end{equation}
by first applying $U$ to obtain a diagram in $\sq(\grmod)^\hop$
and regarding the result as a sequence of composable morphisms in $\grmod$.
We will show that the two diagrams can be made equal by manipulating them
using moves of the following kind, and the reverses of these moves.
\begin{enumerate}
	\item
	Composing consecutive vertical morphisms.
	\item
	Composing consecutive horizontal morphisms.
	\item
	Replacing the top-left part of a 2-cell of $\bbS^d(\calH^\fop)$
	\[\xymatrix{
		\bullet\ar@{}[dr]|{\Downarrow}
		\ar[d]_f
		\ar[r]|-@{|}^e
		&
		\bullet
		\ar[d]^g
		\\
		\bullet
		\ar[r]|-@{|}_h
		&
		\bullet
	}
	\qquad
	\xymatrix{
		\bullet
		\ar[d]_f
		\ar[r]|-@{|}^e
		&
		\bullet
		\\
		\bullet
		&
		{}
	}
	\qquad
	\xymatrix{
		{}
		&
		\bullet
		\ar[d]^g
		\\
		\bullet
		\ar[r]|-@{|}_h
		&
		\bullet
	}
	\]
	with the bottom-right part.
\end{enumerate}
The claim will then follow, because these moves do
not affect the composite of morphisms in $\grmod$ obtained by applying $U$.

To show that the two diagrams \eqref{stepsone}
can be transformed into one another using the moves (1), (2) and (3) above,
let $\alpha'$, $\alpha''$, $\gamma'$ and $\gamma''$ be the evident  
vertical 1-morphisms of $\bbS^d(\calH^\fop)$ of the following kind.
(In what follows, the subscript of a projection map $\pi_{(-)}$ indicates a factor
that is being projected away; thus $\pi_C^\ast S$ denotes the pullback of
of $S$ (which is a family over $B$) to the new base $C\times B$.)
\[\xymatrix@1@!0@C=1.2em{
	*!R{(\emptyset_C,d\chi(T,Y))
	\otimes
	(\pi_B^\ast X,d\chi(S,X))\;}
	\ar[rr]^{\alpha'}
	&&
	*!L{\;(\pi_{C\times B}^\ast X,d\chi(T\ucirc S,X))}
	\\
	*!R{(\emptyset_C,d\chi(T,Y)) \otimes (S,0)\;}
	\ar[rr]^{\alpha''}
	&&
	*!L{\;(\pi_{C}^\ast S,d\chi(T,Y))}
	\\
	*!R{(\pi_{C}^\ast S,d\chi(T,Y))\;}
	\ar[rr]^{\gamma'}
	&&
	*!L{\;(\pi_C^\ast Y,d\chi(T,Y))}
	\\
	*!R{(T\ucirc S,0)\;}
	\ar[rr]^{\gamma''}
	&&
	*!L{\; (T,0)}
}\]
Moreover, let $\beta'$ and $\beta''$ be the evident horizontal 1-morphisms
of the following kind.
\[\xymatrix@1@!0@C=1.2em{
	*!R{(\pi_{C}^\ast S,d\chi(T,Y))\;}
	\ar|-@{|}[rr]^{\beta'}
	&&
	*!L{\;(\pi_{C\times B}^\ast X,d\chi(T\ucirc S,X))}
	\\
	*!R{(T\ucirc S,0)\;}
	\ar|-@{|}[rr]^{\beta''}
	&&
	*!L{\;(\pi_{C}^\ast S,d\chi(T,Y))}
}\]
Observe the identities
\begin{eqnarray*}
	\alpha_{T\ucirc S}\circ(\emptyset_\otimes\otimes 1)
	&=&
	\alpha'\circ(1\otimes\alpha_S),
	\\
	\gamma_{T\ucirc S}
	&=&
	\gamma_T\circ\gamma'',
	\\
	\alpha_T\circ(1\otimes\gamma_S)
	&=&
	\gamma'\circ\alpha'',
	\\
	\beta'\odot\beta''
	&=&
	\beta_{T\ucirc S}.
\end{eqnarray*}
Using moves of the first and second kind, we may
therefore replace the diagrams \eqref{stepsone} with the following.
\begin{equation}\label{stepstwo}
\vcenter{
\xymatrix{
	{}
	&
	{}
	&
	\bullet
	\ar[d]^{1\otimes\alpha_S}
	\\
	{}
	&
	{}
	&
	\bullet
	\ar[d]^{\alpha'}
	\\
	\bullet
	\ar|-@{|}[r]^{\beta''}
	\ar[d]_{\gamma''}
	&
	\bullet
	\ar|-@{|}[r]^{\beta'}	
	&
	\bullet
	\\
	\bullet
	\ar[d]_{\gamma_T}
	&
	{}
	&
	{}
	\\
	\bullet
	&
	{}
	&
	{}
}}
\qquad\qquad
\vcenter{\xymatrix{
	{}
	&
	{}
	&
	\bullet
	\ar[d]^{1\otimes\alpha_S}
	\\
	{}
	&
	\bullet
	\ar|-@{|}[r]^{1\otimes\beta_{S}}
	\ar[d]^{\alpha''}
	&
	\bullet
	\\
	{}
	&
	\bullet
	\ar[d]^{\gamma'}
	&
	{}
	\\
	\bullet
	\ar|-@{|}[r]^{\beta_{T}}
	\ar[d]_{\gamma_T}
	&
	\bullet
	&
	{}
	\\
	\bullet
	&
	{}
	&
	{}
}}
\end{equation}
Next observe that there are 2-cells in $\bbS^d(\calH^\fop)$
of the following form.
\[
\xymatrix@=35 pt{
	{}
	&
	\bullet
	\ar@{}[dr]|{\Downarrow \phi}
	\ar[d]_{\alpha''}
	\ar|-@{|}[r]^{1\otimes\beta_S}
	&
	\bullet
	\ar[d]^{\alpha'}
	\\
	\bullet
	\ar@{}[dr]|{\Downarrow \psi}
	\ar|-@{|}[r]^{\beta''}
	\ar[d]_{\gamma''}
	&
	\bullet
	\ar|-@{|}[r]_{\beta'}	
	\ar[d]^{\gamma'}
	&
	\bullet
	\\
	\bullet
	\ar|-@{|}[r]_{\beta_T}
	&
	\bullet
	&
	{}
}\]
To verify this claim, we must check that each square
qualifies as a 2-cell in $\bbS^d(\calH^\fop)$,
namely that in each fibre we obtain a homotopy cofibre square of h-graphs.
In the case of $\phi$ this is immediate,
and in the case of $\psi$ we are required to show that 
for $b\in B$ and $c\in C$ the square
\[\xymatrix{
	T_c\circ S_b
	&
	S_b
	\ar[l]
	\\
	T_c
	\ar[u]
	&
	Y
	\ar[u]
	\ar[l]
}\]
is a homotopy cofibre square.
But the square is a pushout and the arrows are all cofibrations,
so this claim follows.
We may therefore use moves of the third kind to identify the two diagrams
of \eqref{stepstwo}.
This completes the proof of the gluing axiom.

%% file: constructing-the-umkehr-maps.tex
\section{Constructing the umkehr maps}
\label{sec:umkehr-maps}

Let $G$ be a compact Lie group and let $d=-\dim(G)$.
In this section we will complete the proof of Theorem~\ref{thm:main}
by constructing the required positive homological h-graph
field theory $\Phi^G$. In view of Theorem~\ref{th:push-pull}, 
it will suffice to construct a symmetric monoidal double functor
\[
	U^G \colon\bbS^d(\calH^\fop) \longrightarrow \sq(\grmod)^\hop
\]
whose vertical part is the symmetric monoidal functor
$(X,n)\mapsto H_{\ast-n}(BG^X)$.
We will construct $U^G$ as follows.

\begin{step}
\label{step1}
	 We construct a symmetric monoidal double functor
	\[
		U_\mfld \colon \bbS^\ds(\calM) \longto \sq(\grmod)^\hop
	\]
	where $\calM$ is a category of fibrewise manifolds
	and $\bbS^\ds(\calM)$ is a double category of
	special squares in $\calM$. 
	This double functor will encode the umkehr maps
	$f^!\colon H_{\ast-n}(N)\to H_{\ast-m}(M)$ associated to fibrewise
	smooth maps $f\colon M\to N$ between fibrewise manifolds over the
	same base
	that were mentioned in section~\ref{overview:section}.
	We will define $\calM$
	and construct $\bbS^\ds(\calM)$ 
	in subsection~\ref{subsec:fibrewise-manifolds},
	and we will construct $U_\mfld$ in
	subsection~\ref{subsec:umkehr-maps-for-fibrewise-manifolds}.
\end{step}
\begin{step}
\label{step2}
	We give a construction that obtains a fibrewise manifold
	from a family of h-graphs with basepoints.
	More precisely, 
	given a family of h-graphs with basepoints $(X,P)$
	over a good base space $B$, we construct a fibrewise
	manifold $B(G^{\Pi_1(X,P)})\to B \times BG^P$
	whose total space is homotopy equivalent
	to the fibrewise mapping space $BG^X$ over $B$.
	We will achieve this by constructing the following diagram of 
	symmetric monoidal categories and functors and 
	transformations of such,
	where the transformations restrict to homotopy
	equivalences on total spaces.
	\begin{equation}\label{diagram-one}
	\vcenter{
	\xymatrix@!0@C=6em@R=18ex{
		(\calHbp,\usqcup)^\fop
		\ar[d]_{\forget}
		\ar[rr]^{\Pi_1^\fop}
		\ar[drr]^(0.34){(WQ)^\fop}^{}="uwmid"_{}="dwmid"
		&&
		(\calG,\usqcup)^\fop
		\ar[d]^{B^\fop}
		\ar[rr]^{B(G^{(-)})}
		\ar@{=>}[];"uwmid"^-\homot
		&&
		(\calM,\times)
		\ar[d]^{\forget}
		\ar@{=>}[dll]^{\homot}
		\\
		(\calH,\usqcup)^\fop
		\ar[rr]_{\inclusion}
		\ar@{=>}[];"dwmid"_-\homot
		&&
		(\calS,\usqcup)^\fop
		\ar[r]_(0.53){BG^{(-)}}
		&
		(\calS,\times)
		\ar[r]_{\inclusion}
		&
		(\widehat\calS,\times)
	}}\end{equation}
	Here the composite across the bottom represents the
	passage from a family of h-graphs $X$
	(admitting basepoints and over a good base)
	to the fibrewise mapping space $BG^X$,
	while the composite across the top
	represents the passage from a family of h-graphs with basepoints $(X,P)$
	(over a good base $B$)
	to the fibrewise manifold $B(G^{\Pi_1(X,P)})\to BG^P\times B$.
	The remaining functors and transformations in the diagram
	encode the fact that $B(G^{\Pi_1(X,P)})$ and $BG^X$ are homotopy equivalent.

	Several of the items appearing
	in~\eqref{diagram-one} have been introduced earlier:
	\begin{itemize}
		\item
		$(\calH,\usqcup)$ is the category of families of h-graphs
		(over good base spaces and admitting basepoints)
		equipped with external disjoint union.
		(See Examples~\ref{ex:calH-def} and~\ref{ex:sm-fibration-examples}.)
		
		\item
		$(\calS,\usqcup)$ is the category of fibred spaces
		(over good base spaces)
		equipped with external disjoint union.
		(See Examples~\ref{ex:calS-def} and~\ref{ex:sm-fibration-examples}.)
		
		\item		
		$(\calS,\times)$ and $(\widehat\calS,\times)$ 
		are the categories of fibred spaces
		(over good base spaces and arbitrary base spaces, respectively)
		equipped with direct product.
		(See Examples~\ref{ex:calS-def} and~\ref{ex:sm-calS-with-times}.)
		
		\item
		 $(\calM,\times)$ is the symmetric monoidal fibration over $\calB$
		consisting of fibrewise manifolds constructed in 
		step~\ref{step1} above.
	
		\item
		The symbol $(-)^\fop$ denotes the fibrewise opposite.
		(See subsection~\ref{subsec:fibrewise-opposite}.)
		
		\item
		The functor $BG^{(-)}$ is the fibrewise mapping space functor.
		(See Example~\ref{ex:BG-to-X}.)
		
	\end{itemize}
	The remaining parts of \eqref{diagram-one} are constructed as follows.
	The left part of the diagram is obtained by applying $(-)^\fop$ to a
	diagram
	\begin{equation}
	\label{diag:left-part}
	\vcenter{\xymatrix@!0@C=12em@R=18ex{
		(\calHbp,\usqcup)
		\ar[d]_{\forget}
		\ar[r]^{\Pi_1}
		\ar[dr]^(0.34){WQ}^{}="uwmid"_{}="dwmid"
		&
		(\calG,\usqcup)
		\ar[d]^{B}
		\ar@{=>}"uwmid";[]_-\homot
		\\
		(\calH,\usqcup)
		\ar[r]_{\inclusion}
		\ar@{=>}"dwmid";[]^-\homot
		&
		(\calS,\usqcup)
	}}
	\end{equation}
	of symmetric monoidal fibrations over $\calU$ and functors and 
	transformations of such in which:
	\begin{itemize}
		\item
		$(\calHbp,\usqcup)$
		is a variant of $(\calH,\usqcup)$ in which the objects
		are equipped with an explicit choice of basepoints.
		It is defined in subsection~\ref{subsec:h-graphs-with-basepoints}.
		\item
		$(\calG,\usqcup)$ is a symmetric monoidal fibration over $\calU$
		whose objects consist of `fibrewise finite free groupoids'.
		It is defined in subsection~\ref{subsec:groupoids}.
		\item
		The functor $\Pi_1$ is a fibrewise fundemental groupoid functor.
		It is constructed in subsection~\ref{subsec:h-graphs-to-groupoids}.
		\item
		The remaining functors and transformations in \eqref{diag:left-part}
		are constructed in subsection~\ref{subsec:X-to-B-Pi-One-X}.
		Furthermore, the transformations 
		in diagram \eqref{diag:left-part} (and hence in the left part of 
		diagram \eqref{diagram-one}) in fact consist of fibrewise homotopy 
		equivalences.
	\end{itemize}
	In contrast with the left-hand part, the right-hand square 
	in diagram~\eqref{diagram-one} only consists of symmetric monoidal
	categories and functors and transformations of symmetric monoidal 
	categories; it is not fibred over any base category.
	\begin{itemize}
		\item
		The functor $B(G^{(-)})$ sends a fibrewise finite free groupoid $\sP$
		over $B$
		to a fibred manifold of functors from $\sP$ into $G\times B$.
		It is constructed in
		subsection~\ref{subsec:groupoids-to-manifolds}.
		\item
		The natural transformation in the right-hand square of
		\eqref{diagram-one} is constructed in
		subsection~\ref{subsec:functors-to-mapping-spaces}.	
	\end{itemize}
\end{step}
\begin{step}
\label{step3}
	The final step is to use diagram~\eqref{diagram-one} to construct the umkehr
	functor $U^G$.
	More precisely, we will construct the following diagram
	of symmetric monoidal double categories
	and functors and transformations of such.
	\begin{equation}\label{diagram-two}
	\vcenter{
	\xymatrix@!0@R=11ex@C=8em{
		\bbS^d(\calHbp^\fop)
		\ar[r]^{\bbS(\Pi_1)}
		\ar@/_6 ex/[rrr]_{\widetilde U^G}^{}="2"
		\ar[d]_{\bbS(\forget)}
		&
		\bbS^d(\calG^\fop)
		\ar[r]^{\bbS(B(G^{(-)}))}_{}="1"
		&
		\bbS^\ds(\calM)
		\ar[r]^-{U_\mfld}
		&
		\sq(\grmod)
		\ar@{=>}"1";"2"
		\\
		\bbS^d(\calH^\fop)
		\ar@/_5ex/[rrru]_{U^G}
	}}\end{equation}
	First, in subsection~\ref{subsec:special-squares}, we construct the double
	categories $\bbS^d(\calHbp^\fop)$ and $\bbS^d(\calG^\fop)$,
	the other double categories in diagram~\eqref{diagram-two}
	having already been constructed 
	($\bbS^d(\calH^\fop)$ in
	Examples~\ref{ex:dc-special-squares-h-graphs} and 
	\ref{ex:sm-dc-special-squares-h-graphs};
	$\bbS^\ds(\calM)$ in step~\ref{step1} above; and
	$\sq(\grmod)$ in Example~\ref{ex:sqcat} and
	subsection~\ref{subsec:symmetric-monoidal-double-categories}).
	Then in subsection~\ref{subsec:double-functors},
	the upper and left edges of~\eqref{diagram-one}
	will be used to construct the corresponding parts of~\eqref{diagram-two}.
	Finally, in subsection~\ref{subsec:UG},
	the natural transformations in \eqref{diagram-one}
	will be used to construct 
	a symmetric monoidal natural isomorphism from the top composite to a symmetric
	monoidal double functor $\widetilde U^G$ whose vertical part is exactly
	$(X,P,n)\mapsto H_{\ast-n}(BG^X)$.
	This property of $\widetilde U^G$ will allow us to factor $\widetilde U^G$ through
	a symmetric monoidal double functor $U^G$ of the required kind.
\end{step}

\subsection{Fibrewise manifolds}
\label{subsec:fibrewise-manifolds}
Our first aim is to construct our category $\calM$ of fibrewise
manifolds. Unlike $\calH$ and $\calS$, which are fibred over 
the category $\calU$ of good base spaces, we will construct $\calM$
as a category fibred over the category $\calB$ of all base spaces.
\begin{definition}
By a \emph{fibrewise closed manifold} $M$ over a base space $B$,
we mean a fibrewise smooth fibre bundle $M\to B$
in the sense of Crabb and James \cite[p.~245]{CrabbJames} 
with the property that the fibres of $M\to B$ are closed
manifolds. Unlike Crabb and James, however, we do not require
the base space $B$  to be an ENR or the total space $M$ 
to be second countable, those assumptions being unnecessary
for the most basic aspects of the theory. 
\end{definition}

\begin{definition}
By a \emph{fibrewise smooth map} from a fibrewise closed 
manifold $M_1 \xto{\pi_1} B_1$ 
to a fibrewise closed manifold $M_2 \xto{\pi_2} B_2$,
we mean
a pair of continous maps $(g,f)$ making the diagram
\begin{equation}
\label{diag:calM-morphism}
\vcenter{\xymatrix{
	M_1
	\ar[r]^g
	\ar[d]_{\pi_1}
	&
	M_2
	\ar[d]^{\pi_2}
	\\
	B_1
	\ar[r]^f
	&
	B_2	
}}
\end{equation}
commutative and having the property that the induced map
$M_1 \to f^*M_2$ over $B_1$ is a fibrewise smooth map in the sense 
of Crabb and James \cite[p.~244]{CrabbJames}.
\end{definition}

\begin{definition}[The fibred category $\calM$]
We let $\calM$ be the category of fibrewise closed
manifolds and fibrewise smooth maps and equip it 
with the symmetric monoidal structure given by 
the direct product 
\[
	(M_1\xto{\pi_1} B_1) \times (M_2 \xto{\pi_2} B_2)
	 = (M_1\times M_2 \xto{\pi_1\times\pi_2} B_1 \times B_2).
\]
The forgetful functor $\calM \to \calB$ sending a fibrewise 
manifold $M$ over $B$ to $B$ is then a symmetric monoidal fibration. 
\end{definition}

\begin{example}
Let $(F,G,P)$ be a triple consisting of a smooth closed manifold
$F$, a Lie group $G$ acting smoothly on $F$,
and a principal $G$-bundle $P$ over a base $B$.
Then the induced fibre bundle $P\times_G F\to B$
admits the structure of a fibrewise closed manifold.
See \cite[p.~246]{CrabbJames}.
Given a second such triple $(F',G',P')$ and a triple of maps
\[
	f_F\colon F\longto F',
	\qquad
	f_G\colon G\longto G',
	\qquad
	f_P\colon P\longto P',
\]
with $f_F$ smooth, $f_G$ a smooth homomorphism,
and $f_F$ and $f_P$ both equivariant with respect to $f_G$,
then the induced map
\[
	P\times_G F\longto  P'\times_{G'}F'
\]
is fibrewise smooth.
\end{example}

Having defined $\calM$, we now proceed to construct the symmetric monoidal
double category $\bbS^\ds(\calM)$ following a strategy similar to the
one we employed in Examples~\ref{ex:dc-special-squares-h-graphs}
and \ref{ex:sm-dc-special-squares-h-graphs} to construct the 
symmetric monoidal double category $\bbS^d(\calH^\fop)$.
First observe that the fibre product $\calM \times_\calB \calB^\ds$
of $\calM$ with the category $\calB^\ds$ defined in 
Example~\ref{ex:dc-special-squares-h-graphs}
can be described as follows.
\begin{itemize}
\item The objects of $\calM \times_\calB \calB^\ds$ are pairs 
	$(M \to B, m)$, where $M\to B$
	is an object of $\calM$ and $m\colon B \to \Z$
	is a locally constant function.
\item A morphism in $\calM \times_\calB \calB^\ds$ 
	from $(M_1 \to B_1,m_1)$ to $(M_2 \to B_2,m_2)$ 
	is simply a morphism from $M_1 \to B_1$  to $M_2\to B_2$ in $\calM$. 
\end{itemize}
Also observe that the symmetric monoidal structures on $\calM$ and $\calB^\ds$ 
induce on $\calM \times_\calB \calB^\ds$ a symmetric 
monoidal structure given by the product
\[
	(M_1\xto{\pi_1} B_1, m_1) \times (M_2\xto{\pi_2} B_2, m_2)
		=
	(M_1\times M_2 \xto{\pi_1\times \pi_2} B_1\times B_2, m_1+m_2).
\]
\begin{definition}[The double category $\bbS^\ds(\calM)$]
\label{def:sm-dc-special-squares-manifolds}
We let the  
\emph{double category} $\bbS^\ds(\calM)$ 
\emph{of special squares in $\calM$ with degree shifts}
be the sub-double category of 
$\sq(\calM \times_\calB \calB^\ds)$ defined as follows. For a 
fibrewise closed manifold $M \to B$, let $|M|$ denote the 
locally constant function $B \to \Z$ sending a point $b\in B$
to the dimension of the fibre $M_b$. Then 
\begin{itemize}
\item The objects of $\bbS^\ds(\calM)$ are the objects of
	 $\calM \times_\calB \calB^\ds$.
\item A morphism $(g,f)$ from $(M_1 \to B_1,m_1)$ to $(M_2 \to B_2,m_2)$ 
	in $\calM \times_\calB \calB^\ds$ as depicted in diagram 
	\eqref{diag:calM-morphism} qualifies as a horizontal 
	1-morphism in $\bbS^\ds(\calM)$ if $f$ is the identity map 
	and $m_1$ and $m_2$ satisfy 
	\[
		m_1 + |M_1| = m_2 + |M_2|.
	\]
\item The morphism $(g,f)$ from $(M_1 \to B_1,m_1)$ to $(M_2 \to B_2,m_2)$ 
	in $\calM \times_\calB \calB^\ds$ qualifies
	as a vertical 1-morphism if 
	$m_1 = m_2 \circ f$ and the restriction of $g$ to fibres
	\[
		g| \colon (M_1)_b \longto (M_2)_{f(b)}
	\]
	is a submersion for all $b \in B_1$. 
\item A commutative square in $\calM \times_\calB \calB^\ds$ 
	qualifies as a 2-cell in $\bbS^\ds(\calM)$
	if its horizontal and vertical morphisms are, respectively, 
	horizontal and vertical 1-morphisms as defined above,  
	and the commutative square formed by the maps between 
	total spaces is a pullback square of spaces.
\end{itemize}
We equip the double category $\bbS^\ds(\calM)$ with the symmetric monoidal 
structure inherited from $\sq(\calM \times_\calB \calB^\ds)$.
\end{definition}

\subsection{Umkehr maps for fibrewise manifolds}
\label{subsec:umkehr-maps-for-fibrewise-manifolds}
In this subsection we will construct the symmetric monoidal 
double functor
\[
	U_\mfld \colon \bbS^\ds(\calM) \longto \sq(\grmod)^\hop
\]
relying on Crabb and James \cite[section II.12]{CrabbJames}.
The vertical part
\[
	(U_\mfld)_0 \colon \bbS^\ds(\calM)_0 \longto \grmod
\]
of $U_\mfld$ is easy to construct: it is simply the functor that 
sends an object $(M\to B,m)$ to the homology $H_{\ast-m}(M)$.
The main difficulty is the construction of $U_\mfld$ on 
horizontal 1-morphisms of $\bbS^\ds(\calM)$. Given a 
horizontal 1-morphism of $\bbS^\ds(\calM)$, that is to say a map
\begin{equation}
\label{diag:umkehrfodder}
\vcenter{\xymatrix{
	M 
	\ar[rr]^f
	\ar[dr]
	&&
	N
	\ar[dl]
	\\
	&
	B	
}}
\end{equation}
of fibrewise closed manifolds over a base space $B$ 
equipped with degree shifts $m,n \colon B \to \Z$ 
satisfying $m+|M|= n+|N|$, the task is 
to construct an umkehr map
\begin{equation}
\label{eq:umkehrmap}
 f^! \colon H_{\ast -n}(N) \longto H_{\ast-m}(M)
\end{equation}
so that setting $U_\mfld(f) = f^!$ makes $U_\mfld$
a symmetric monoidal double functor.

To construct $f^!$, we will first consider the case where 
$B$ is a finite CW complex (and hence in particular a compact ENR).
In this situation, if $\beta$ is a virtual bundle over $N$,
Crabb and James \cite[p.~266]{CrabbJames} use a fibrewise
Pontryagin--Thom construction to construct a Gysin map
\begin{equation}
\label{eq:gysinmap}
	f_\beta^\umk \colon N^{\beta - \tau_N} \longto M^{f^\ast\beta - \tau_M} 
\end{equation}
between Thom spectra. Here $\tau_N$ and $\tau_M$ denote the 
fibrewise tangent bundles of $N$ and $M$, respectively.
(Crabb and James denote this map by $f^!$,
conflicting the notation we have chosen here.)
\begin{definition}[Definition of $f^!$ when $B$ is finite CW]
\label{def:f-shriek-CW}
Suppose the space $B$ in \eqref{diag:umkehrfodder} is a finite CW complex. 
To construct the map $f^!$ of \eqref{eq:umkehrmap} in this situation,
we choose a virtual bundle $\beta$ over $N$ with virtual dimension 
$|\beta|$ equal to $n + |N|$, and define $f^!$ as the unique
map making commutative the diagram
\[\xymatrix{
	H_{\ast -n}(N)
	\ar@{-->}[r]^{f^!}
	&
	H_{\ast-m}(M)
	\\
	\tilde H_\ast(N^{\beta - \tau_N})
	\ar[r]^-{(f_\beta^\umk)^{\phantom{\gamma}}_{\ast}}
	\ar[u]^{\mathrm{Thom}}_\isom
	&
	\tilde H_\ast(M^{f^\ast\beta - \tau_M})
	\ar[u]_{\mathrm{Thom}}^\isom
}\]
where the vertical maps are given by Thom isomorphisms.
(Recall that we are working in characteristic 2, so any
virtual bundle has a canonical orientation.)
\end{definition}
In the preceding definition, we could of course choose as 
$\beta$ the trivial 
virtual bundle of the required dimension, but for the purpose
of verifying that $U_\mfld$ satisfies the required axioms,
it is important to have the flexibility to use other choices of 
$\beta$ as well. Of course, we should verify that $f^!$ does
not depend on the choice of the virtual bundle $\beta$. 
We will postpone the proof of
the following lemma until the end of this subsection.
\begin{lemma}
\label{lm:umkehrwelldefined}
The map $f^!=f^!_\beta$ constructed in 
Definition~\ref{def:f-shriek-CW} 
is independent of the choice of $\beta$.
\end{lemma}

We will now generalize the construction of $f^!$ to the case where 
the base space $B$ is not necessarily a finite CW complex.

\begin{definition}[Definition of $f^!$ for arbitrary $B$]
Suppose the space $B$ in \eqref{diag:umkehrfodder} is 
an arbitrary base space. Let $\Gamma B$ denote 
the geometric realization of the singular simplicial
set of $B$, and let $\{B_\lambda\}_\lambda$ be the poset of finite 
CW subcomplexes of $\Gamma B$. For each $\lambda$ we have a map
$B_\lambda \to B$ obtained by composing the inclusion of $B_\lambda$
into $\Gamma B$ with the natural map $\Gamma B \to B$. Let
$M_\lambda$ and $N_\lambda$ denote the pullbacks of $M$ and $N$
over $B_\lambda$, and let $f_\lambda$ denote the map
\[
	f_\lambda \colon M_\lambda \longto N_\lambda
\]
induced by $f$. Then the maps 
\[
	H_{\ast-m}(M_\lambda) \longto H_{\ast-m}(M)
\]
exhibit $H_{\ast-m}(M)$ as a colimit
\[
	H_{\ast-m}(M) = \colim_\lambda H_{\ast-m}(M_\lambda)
\]
and similarly for $H_{\ast-n}(N)$.
It follows from \cite[Proposition II.12.16]{CrabbJames}
that the maps $(f_\lambda)^!$ of Definition~\ref{def:f-shriek-CW}
are compatible with 
each other as $\lambda$ varies
in the sense that for $B_\lambda \subset B_\mu$, 
the diagram
\[\xymatrix{
	H_{\ast-n}(N_\lambda)
	\ar[r]^{(f_\lambda)^!}
	\ar[d]
	&
	H_{\ast-m}(M_\lambda)
	\ar[d]
	\\
	H_{\ast-n}(N_\mu)
	\ar[r]^{(f_\mu)^!}
	&
	H_{\ast-m}(M_\mu)	
}\]
commutes. 
We now define $f^!$ to be the colimit
\[
	f^! = (\colim_\lambda (f_\lambda)^! )
		\colon H_{\ast-n}(N) \longto H_{\ast-m}(M).
\]
It is easily verified that the map $f^!$ so defined agrees with the 
previously defined $f^!$ when $B$ happens to be a finite CW complex.
\end{definition}

Using Lemma~\ref{lm:umkehrwelldefined}, 
\cite[Proposition~II.12.16]{CrabbJames}
and the fibrewise versions \cite[p.~266]{CrabbJames} of 
\cite[Propositions II.12.7 and II.12.11]{CrabbJames},
it is now easy to check that 
\[
	U_\mfld \colon \bbS^\ds(\calM) \longto \sq(\grmod)^\hop
\]
is a double functor. For example, the fact that $U_\mfld$
takes 2-cells to 2-cells follows from Lemma~\ref{lm:umkehrwelldefined}
together with \cite[Proposition II.12.16]{CrabbJames} and the 
fibrewise analogue of \cite[Propositions II.12.11]{CrabbJames}.

To promote $U_\mfld$ into a symmetric monoidal double functor,
we will now explicitly specify the monoidality and identity constraints 
$U_{\mfld,\tensor}$ and $U_{\mfld,I}$. The constraint 
\[
	U_{\mfld,\tensor,0} 
	\colon 
	U_\mfld(M,m) \tensor U_\mfld(N,n) 
	\longto 
	U_\mfld(M\times N,m+n)
\]
is simply the cross product
\[
	H_{\ast-m}(M) \tensor H_{\ast-n}(N) 
	\xto{\ \times\ }
	H_{\ast-m-n}(M\times N),
\]
while for horizontal 1-morphisms $f_i\colon (M_i,m_i) \to (N_i,n_i)$,
$i=1,2$ the constraint $U_{\mfld,\tensor,1}$ is given by 
the square
\[\xymatrix@C+1.2em@R+3ex{
	H_{\ast-n_1}(N_1) \tensor H_{\ast-n_2}(N_2)
	\ar[d]_\times
	\ar[r]^-{f_1^!\tensor f_2^!}
	&
	H_{\ast-m_1}(M_1) \tensor H_{\ast-m_2}(M_2)
	\ar[d]^\times
	\\
	H_{\ast-n_1-n_2}(N_1 \times N_2)
	\ar[r]^-{(f_1\times f_2)^!}
	&
	H_{\ast-m_1-m_2}(M_1 \times M_2)
}\]
whose commutativity follows from the fibrewise version of
\cite[Proposition II.12.8]{CrabbJames}.
The identity constraint $U_{\mfld,I,0}$ is given by 
the isomorphism
\[
	(\F,0) \xto{\ \isom\ } H_\ast(\pt)
\]	
where $\F$ is our coefficient field of characteristic 2,
and the identity constraint $U_{\mfld,I,1}$ is given by the 
square
\[\xymatrix{
	(\F,0)
	\ar[r]^\id
	\ar[d]_\isom
	&
	(\F,0)
	\ar[d]^\isom
	\\
	H_\ast(\pt)
	\ar[r]^\id
	&
	H_\ast(\pt)
}\]
It is now readily verified that these data make $U_\mfld$ into
a symmetric monoidal double functor.

The work of this subsection is now finished except for the 
proof of Lemma~\ref{lm:umkehrwelldefined} which we postponed
earlier.

\begin{proof}[Proof of Lemma~\ref{lm:umkehrwelldefined}]
It is enough to verify that for any choice of $\beta$, the diagram
\begin{equation}
\label{sq:umkehrcomp}
\vcenter{\xymatrix@C+1em@R+2ex{
	H_{\ast -n}(N)
	\ar[r]^{f^!_\beta}
	&
	H_{\ast-m}(M)
	\\
	\tilde H_{\ast-n-|N|}(N^{-\tau_N})
	\ar[r]^-{(f^\umk)_\ast}
	\ar[u];[]_{\mathrm{Thom^{-1}}}^\isom
	&
	\tilde H_{\ast-m-|M|}(M^{-\tau_M})
	\ar[u];[]^{\mathrm{Thom^{-1}}}_\isom
}}
\end{equation}
commutes, where
\[
	f^\umk \colon N^{-\tau_N} \longto M^{-\tau_M}
\]
is the map \eqref{eq:gysinmap} for $\beta$ the trivial 0-dimensional 
virtual bundle. Working component\-wise, we may assume that the 
base space $B$ is connected, in which case all locally constant 
functions on $B$ are constant. 
Using the fibrewise versions \cite[p.~266]{CrabbJames}
of \cite[Proposition II.12.11]{CrabbJames} and 
\cite[Proposition II.12.16]{CrabbJames}, the diagram
\[\xymatrix{
	M
	\ar[r]^f
	\ar[d]_{(f,\id)}
	&
	N
	\ar[d]^{\Delta}
	\\
	N\times M
	\ar[r]^{\id\times f}
	&
	N\times N
}\]
gives rise to a commutative diagram
\begin{equation}
\label{diag:frobexpand}
\vcenter{\xymatrix@C+4em@R+3ex{
	N^{\beta -\tau_N}
	\ar[r]^{f^\umk_\beta}
	\ar[d]_{\Delta}
	&
	M^{f^\ast\beta - \tau_M}
	\ar[d]^{(f,\id)}
	\\
	(N \times N)^{\pr_1^\ast(\beta+\tau_N) - \tau_{N\times N}}
	\ar[r]^{(\id \times f)^\umk_{\pr_1^\ast(\beta+\tau_N)}}
	&
	(N \times M)^{\pr_1^\ast(\beta+\tau_N) - \tau_{N\times M}}
}}
\end{equation}
where $\pr_1$ denotes projection onto the first factor.

We have the following diagram
\[\xymatrix@!0@C=6.7em@R=9ex{
	&
	H_{\ast-n}(N)
	\ar[rrr]^{f^!_\beta}
	&&&
	H_{\ast-m}(M)
	\\
	&
	\tilde H_{\ast}(N^{\beta-\tau_N})
	\ar[rrr]^{(f^\umk_\beta)_\ast}
	\ar@{<-}[u]^{\mathrm{Thom}^{-1}}_\isom
	\ar[d]_{\Delta_\ast}
	&&&
	\tilde H_{\ast}(M^{f^\ast\beta-\tau_M})
	\ar@{<-}[u]_{\mathrm{Thom}^{-1}}^\isom
	\ar[d]^{(f,\id)_\ast}
	\\
	&
	\tilde H_\ast\big((N \times N)^{\pr_1^\ast(\beta+\tau_N) - \tau_{N\times N}}\big)
	\ar[rrr]^{\big((\id \times f)^\umk_{\pr_1^\ast(\beta+\tau_N)}\big)_\ast}
	&&&
	\tilde H_\ast\big((N \times M)^{\pr_1^\ast(\beta+\tau_N) - \tau_{N\times M}}\big)
	\\
	&
	\tilde H_\ast(N^\beta \smashprod N^{-\tau_N})
	\ar[rrr]^{(\id\smashprod f^\umk)_\ast}
	\ar@{<-}[u]_\isom
	&&&
	\tilde H_\ast(N^\beta \smashprod M^{-\tau_M})
	\ar@{<-}[u]^\isom
	\\
	&
	\tilde H_\ast(N^\beta)\tensor \tilde H_\ast(N^{-\tau_N})
	\ar[rrr]^{\id\tensor (f^\umk)_\ast}
	\ar@{<-}[u]_\isom^{(\times)^{-1}}
	\ar[d]_{\mathrm{Thom}\tensor \id}^\isom
	&&&
	\tilde H_\ast(N^\beta)\tensor \tilde H_\ast(M^{-\tau_M})
	\ar@{<-}[u]^\isom_{(\times)^{-1}}
	\ar[d]^{\mathrm{Thom}\tensor \id}_\isom	
	\\
	&
	H_{\ast-n-|N|}(N) \tensor \tilde H_\ast(N^{-\tau_N})
	\ar[rrr]^{\id\tensor (f^\umk)_\ast}
	\ar[d]_{\pr_\ast\tensor \id}
	&&&
	H_{\ast-n-|N|}(N) \tensor \tilde H_\ast(M^{-\tau_M})
	\ar[d]^{\pr_\ast\tensor \id}
	\\
	&
	H_{\ast-n-|N|}(\pt) \tensor \tilde H_\ast(N^{-\tau_N})
	\ar[rrr]^{\id\tensor (f^\umk)_\ast}
	\ar[d]_\times^\isom
	&&&
	H_{\ast-n-|N|}(\pt) \tensor \tilde H_\ast(M^{-\tau_M})
	\ar[d]^\times_\isom
	&
	\\
	&
	\tilde H_{\ast-n-|N|}(N^{-\tau_N})
	\ar[rrr]^{(f^\umk)_\ast}
	\ar@{<-}	`l[l]
		`[uuuuuuu]^\isom_{\mathrm{Thom}^{-1}}
		[uuuuuuu]
	&&&
	\tilde H_{\ast-n-|N|}(M^{-\tau_M})	
	\ar@{<-}	`r[ru]
		`[uuuuuuu]_\isom^{\mathrm{Thom}^{-1}}
		[uuuuuuu]
	&
}\]
Observe that the square \eqref{sq:umkehrcomp} appears as the outer
square of the preceding diagram, so it is enough to show that 
the diagram commutes.
The top middle square of the diagram
commutes by the definition of $f^!_\beta$,
and the second square from top commutes by the commutativity of 
diagram \eqref{diag:frobexpand}. The commutativity of the third 
square from top follows from the fibrewise version 
\cite[p.~266]{CrabbJames} of \cite[Proposition II.12.8]{CrabbJames}.
The commutativity of the remaining squares is immediate.
The commutativity of the flanking octagon on the right
follows from the commutativity of the diagram
\[\xymatrix@!0@C=5em@R=11ex{
	\tilde H_\ast(M^{f^\ast\beta-\tau_M})
	\ar@{<-}[rr]^-{\mathrm{Thom}^{-1}}_\isom
	\ar[dd]_{(f,\id)_\ast}
	&&
	H_{\ast-m}(M)
	\ar[rr]^(0.47)\id
	\ar[d]_{(f,\id)_\ast}
	&&
	H_{\ast-m}(M)
	\\
	&&
	H_{\ast-m}(N\times M)
	\ar[urr]^(0.47)*[@]=-!U(1.2){\labelstyle\pr_\ast}
	\\
	\tilde H_\ast\big((N \times M)^{\pr_1^\ast(\beta+\tau_N) - \tau_{N\times M}}\big)
	\ar[urr]^(0.47)*[@]=-!U(1.2){\labelstyle\mathrm{Thom}}_-\isom
	&&
	&&
	H_{\ast-n-|N|}(\pt)\tensor H_{\ast+|M|}(M)
	\ar[uu]_\times^\isom
	\\
	\tilde H_\ast(N^\beta \smashprod M^{-\tau_M})
	\ar@{<-}[u]^\isom
	&&
	H_{\ast-n-|N|}(N) \tensor H_{\ast+|M|}(M)
	\ar@{<-}[uu]^\isom_{(\times)^{-1}}
	\ar[urr]^(0.47)*[@]=-!U(1.2){\labelstyle\pr_\ast\tensor\id}
	&&
	&
	\tilde H_{\ast-n-|N|}(M^{-\tau_M})
	\ar@{<-}@/_3.4pc/[uuul]^{\mathrm{Thom}^{-1}}_\isom
	\\ 
	H_\ast(N^\beta) \tensor \tilde H_\ast(M^{-\tau_M})
	\ar@{<-}[u]^(0.55)\isom_(0.55){(\times)^{-1}}
	\ar[rru]^(0.47)*[@]=-!U(1.2){\labelstyle\mathrm{Thom}\tensor\mathrm{Thom}}_\isom
	\ar[drr]^(0.47)*[@]=-!U(1.2){\labelstyle\mathrm{Thom}\tensor\id}_\isom
	&&
	&&
	H_{\ast-n-|N|}(\pt)\tensor \tilde H_\ast(M^{-\tau_M})
	\ar[uu]^(0.45){\id\tensor\mathrm{Thom}}_(0.45)\isom
	\ar[ur]^\times_\isom
	\\ 
	&&
	H_{\ast-n-|N|}(N) \tensor \tilde H_{\ast}(M^{-\tau_M})
	\ar[uu]^(0.45){\id\tensor\mathrm{Thom}}_(0.45)\isom
	\ar[urr]^(0.47)*[@]=-!U(1.2){\labelstyle\pr_\ast\tensor\id}
}\]
and the commutativity of the octagon on the left can be established
in a similar way.
\end{proof}

\subsection{H-graphs with basepoints}
\label{subsec:h-graphs-with-basepoints}

We now begin the task of constructing diagram~\eqref{diag:left-part} by
constructing the symmetric monoidal fibration 
$\calHbp \to \calU$
appearing there. This fibration is a variant of 
$\calH \to \calU$ in which the objects are families of h-graphs
equipped with a choice of basepoints.

\begin{definition}
\label{def:h-graphs-with-basepoints-and-maps-of-such}
By a \emph{set of basepoints} for an h-graph $X$ we mean
a finite set $P$ together with a map $u\colon P \to X$ 
(not necessarily injective) such that
the induced map $P \to \pi_0 X$ is surjective. We will denote
$X$ equipped with this set of basepoints by $(X,P,u)$ or simply $(X,P)$,
leaving the map $u$ implicit. By a \emph{map of h-graphs with basepoints}
from $(X,P,u)$ to $(Y,Q,v)$ we mean a pair $(f,i)$ consisting 
of a continuous map $f\colon X \to Y$ and
an injective map $i\colon P\to Q$ fitting into a commutative diagram
\[\xymatrix{
	P
	\ar[d]_u
	\ar[r]^i
	&
	Q
	\ar[d]^v
	\\
	X
	\ar[r]^f
	&Y
}\]
\end{definition}

\begin{definition}
\label{def:families-with-basepoints-and-maps-of-such}
More generally, a set of basepoints for a family of h-graphs $X$ 
over a base space $B$ consists of a finite set $P$ and a 
continuous map $u\colon B \times P \to X$ over $B$ 
such that the induced map $P\to X_b$
is a set of basepoints of $X_b$ for each $b\in B$. Again,
we denote $X$ equipped with these basepoints by $(X,P,u)$ or $(X,P)$.
A \emph{map of families of h-graphs with basepoints} from a family 
$(X,P,u)$ over $B$ to a family $(Y,Q,v)$ over $C$ is 
a triple $(f,g,i)$ consisting of continuous maps
$f \colon X \to Y$  and $g \colon B\to C$ and an injective map
$i\colon P \to Q$ making the diagram
\begin{equation}
\label{diag:mapoffamilieswbp}
\vcenter{\xymatrix{
	B\times P 
	\ar[r]^{g \times i}
	\ar[d]_u
	&
	C \times Q
	\ar[d]^v
	\\
	X
	\ar[r]^f
	\ar[d]
	&
	Y
	\ar[d]
	\\
	B
	\ar[r]^g
	&
	C
}}
\end{equation}
commutative.
\end{definition}

\begin{definition}[The fibred category $\calHbp$]
\label{def:calHbp}
Let $\calHbp$ be the category of families of h-graphs with basepoints 
over good base spaces, and maps of such.
There is a forgetful functor $\calHbp \to \calH$,
and the composite $\calHbp \to \calH \to \calU$ makes
$\calHbp$ a category fibred over $\calU$. A morphism $(f,g,i)$
of $\calHbp$ as in diagram~\eqref{diag:mapoffamilieswbp}
is cartesian exactly when $(f,g)$ is a cartesian
morphism of $\calH$ and $i$ is an isomorphism.
We equip $\calHbp$ with the symmetric monoidal structure $\usqcup$
given by
\[
	(X,P) \usqcup (Y,Q) = (X \usqcup Y, P \sqcup Q).
\]
for families of h-graphs with basepoints $(X,P)$  and $(Y,Q)$.
This makes $\calHbp\to\calU$ into a symmetric monoidal fibration.
The neutral object for $\calHbp$ is the pair
$(\emptyset_\pt,\emptyset)$ consisting of the empty family of h-graphs
over a single point and the empty set of basepoints.
\end{definition}

\subsection{Fibrewise finite free groupoids}
\label{subsec:groupoids}
We continue the construction of diagram~\eqref{diag:left-part}
by constructing the symmetric monoidal fibration $\calG \to \calU$.
We call $\calG$ the \emph{category of fibrewise finite free groupoids}.
We start the construction of $\calG$ by first discussing 
finite free groupoids in an unparameterized setting.

\begin{definition}
We call a small groupoid $\sP$ a \emph{finite free groupoid} 
if the set of objects $\Ob\sP$ is finite and 
there exists a finite indexed subset
$\Gamma = \{\gamma_i\colon x_i\to y_i\}_{i \in I}$ 
of $\Mor\sP$  with the following property: given a groupoid $\sG$,
a function $f_0 \colon \Ob\sP \to \Ob\sG$,
and morphisms $\tilde{\gamma}_i\colon f_0(x_i)\to f_0(y_i)$ of $\sG$ for $i \in I$,
there is a unique groupoid homomorphism
$f\colon \sP \to \sG$ extending $f_0$ and sending
each $\gamma_i$ to $\tilde{\gamma}_i$.
Such a $\Gamma$ is called a \emph{basis} for $\sP$.
A \emph{morphism of finite free groupoids} is simply a groupoid 
homomorphism.
\end{definition}

\begin{example}
A free group on a finite number of generators is a finite free groupoid
with basis given by the generators. In particular, the fundamental
group of any h-graph equipped with a distinguished basepoint
is a finite free groupoid.
\end{example}

\begin{example}
\label{ex:h-graph-subgroupoid}
More generally, if $X$ is an h-graph and $P$ is a finite subset of $X$,
then the full subgroupoid of the fundamental groupoid of $X$ spanned
by $P$ is a finite free groupoid. If $P$ intersects only a single
path component of $X$, 
a basis for this groupoid can be constructed
by choosing a point $p_0\in P$ and adding to a basis of 
the fundamental group $\pi_1(X,p_0)$ 
a single morphism from $p_0$ to every other point of $P$.
In general, if $P$ intersects several path components of $X$,
a basis can be constructed by performing this procedure componentwise.
\end{example}

\begin{definition}
The \emph{rank} of a finite free groupoid $\sP$ is defined 
to be the number of elements in any basis of $\sP$.  
This is well-defined by the following lemma.
\end{definition}

\begin{lemma}
Any two bases of a finite free groupoid have the same cardinality.
\end{lemma}
\begin{proof}
Let $\sP$ be a finite free groupoid.
Let $k$ be a field, regarded as a one-object groupoid via addition.
The set of groupoid morphisms from $\sP$ to $k$ admits
the structure of a $k$-vector space, and by the defining property of 
basis its dimension is equal to the cardinality of any basis of $\sP$.
\end{proof}

\begin{lemma}
\label{lm:basis-words}
Let $\Gamma$ be a basis for a finite free groupoid $\sP$.
Then any morphism of $\sP$ can be written as a word in the elements
of $\Gamma$ and their inverses.
\end{lemma}

\begin{proof}
Let $\sP_0$ denote the subgroupoid of $\sP$ consisting of all objects
and exactly those morphisms
which can be written as words in the elements of $\Gamma$ and their inverses.
Since $\Gamma$ is a basis,
any functor from $\sP_0$ to a groupoid $\sG$ extends uniquely
to a functor from $\sP$ to $\sG$.
It follows from Yoneda's lemma that the inclusion of $\sP_0$ into
$\sP$ is an isomorphism, and so $\sP_0=\sP$.
\end{proof}

We will construct the category $\calG$ 
as a subcategory of a larger category we now define.

\begin{definition}
By a \emph{fibrewise category internal to spaces}
parameterized by a good base space $B$ we mean a category internal to the
category $\calS_B$ of spaces over $B$. If $\sE_1$ and $\sE_2$ 
are such categories over good base spaces $B_1$ and $B_2$, respectively,
a morphism $\varphi$ from $\sE_1$ to $\sE_2$ consists of a map 
$\varphi_B\colon B_1 \to B_2$ and maps
\[
	\varphi_{\sE,0} \colon \Ob \sE_1 \longto \Ob \sE_2
	\quad\text{and}\quad
	\varphi_{\sE,1} \colon \Mor \sE_1 \longto \Mor \sE_2
\]
covering $\varphi_B$ such that for each $b\in B_1$,
$\varphi_{\sE,0}$ and $\varphi_{\sE,1}$ restrict to define
a functor 
$(\sE_1)_b \to (\sE_2)_{\varphi_B(b)}$
of categories internal to spaces. 
The \emph{external fibrewise disjoint union} 
of $\sE_1$ and $\sE_2$, denoted $\sE_1\usqcup\sE_2$,
is the evident fibrewise category internal to spaces
parameterized by $B_1 \times B_2$ whose fibre over a point 
$(b_1,b_2)\in B_1 \times B_2$ is the disjoint union
$(\sE_1)_{b_1} \sqcup (\sE_2)_{b_2}$.
We denote by $\calC$ the category of 
fibrewise categories internal to spaces. 
$\calC$ is a category fibred over $\calU$, a morphism 
$\varphi$ of $\calC$ being cartesian if and only if 
$\varphi_{\sE,0}$ and $\varphi_{\sE,1}$
are cartesian morphisms of $\calS$. Moreover, 
external fibrewise disjoint union makes the
fibration $\calC\to \calU$ into a symmetric monoidal fibration.
\end{definition}

We are now ready to construct the symmetric monoidal fibration
$\calG \to \calU$.

\begin{definition}
A \emph{fibrewise finite free groupoid} over a good base space $B$
is an object $\sP$ of $\calC$ over $B$ having the 
property that for each point $b\in B$, 
there exists a neighborhood $U$ of $b$
and a finite free groupoid $\sP_0$ such that 
the restriction of $\sP$ to $U$ is isomorphic in $\calC_U$
to $\sP_0 \times U$. 
\end{definition}

\begin{definition}[The fibred category $\calG$]
We denote by $\calG$ the full subcategory of $\calC$ spanned
by fibrewise finite free groupoids, equipped with the
symmetric monoidal structure inherited from $\calC$.
The desired symmetric monoidal fibration $\calG\to \calU$
is now given by the composite  $\calG \to \calC \to \calU$.
\end{definition}

\subsection{From families of h-graphs to fibrewise groupoids}
\label{subsec:h-graphs-to-groupoids}
We continue the construction of diagram~\eqref{diag:left-part}
by constructing the map
\[\xymatrix{
	\calHbp
	\ar[rr]^{\Pi_1}
	\ar[rd]
	&& 
	\calG
	\ar[ld]
	\\
	&
	\calU
}\]
of symmetric monoidal fibrations over $\calU$.
We call $\Pi_1$ the \emph{fibrewise fundamental groupoid} functor.
Again, we start by first discussing $\Pi_1$ in the 
unparameterized setting, building on 
Example~\ref{ex:h-graph-subgroupoid}.
In general, a set of basepoints $u\colon P \to X$ for 
an h-graph $X$ is not an injective map, so it does not quite 
make sense to speak about the full subgroupoid of $X$ spanned
by $P$. However, we may remedy this problem by first replacing
$X$ by the mapping cylinder of $u$.

\begin{definition}
Let $X$ be an h-graph with basepoints $u\colon P\to X$,
and let $X'$ denote the mapping cylinder of $u$.
The \emph{fundamental groupoid} $\Pi_1(X,P,u)$ of $(X,P,u)$
is the full subgroupoid spanned by $P$ of the fundamental groupoid of $X'$.
(Here we identify $P$ with a subset of $X'$ as usual.)  
We will often drop the map $u$ from notation 
and denote $\Pi_1(X,P,u)$ by $\Pi_1(X,P)$.
\end{definition}

For any $p_0,p_1\in P$,
the collapse map $X' \to X$ induces a bijection from the 
set of homotopy classes of paths in $X'$ from $p_0$ to $p_1$
to the set of homotopy classes of paths in $X$ from $u(p_0)$
to $u(p_1)$. Thus we may alternatively regard $\Pi_1(X,P)$
as the groupoid where objects are 
elements of $P$, where maps from $p_0$ to $p_1$ are  
the homotopy classes of paths in $X$ from $u(p_0)$ to $u(p_1)$,
and where composition of maps is given by concatenation of paths.

A map of h-graphs with basepoints (see 
Definition~\ref{def:h-graphs-with-basepoints-and-maps-of-such})
\[
	(f,i)\colon (X,P,u) \longto (Y,Q,v)
\]
induces in an obvious way a morphism of finite free groupoids 
$\Pi_1(X,P)\to \Pi_1(Y,Q)$, so that $\Pi_1$ gives a functor from 
the category of h-graphs with basepoints into the category 
of finite free groupoids. Notice that 
if $i$ is a bijection and $f$ is a homotopy equivalence, 
the map induced by $(f,i)$ is an isomorphism.

\begin{definition}
The \emph{fibrewise fundamental groupoid} of a family of 
h-graphs with basepoints $(X,P,u)$ over a good base space $B$ is the
category $\Pi_1(X,P) = \Pi_1(X,P,u)$ internal to spaces over $B$ 
whose fibre over $b\in B$ is $\Pi_1(X_b,P)$. The space of objects
in $\Pi_1(X,P)$ is topologized by identifying it with 
the product $P\times B$, while the space of morphisms is 
given a topology as follows. Let $X'$ be the mapping cylinder of 
$u \colon P\times B \to X$. Then the space of morphisms
in $\Pi_1(X,P)$ is topologized as a quotient of the 
subspace of the fibrewise mapping space $\Map_B(I\times B, X')$
consisting of those fibrewise paths in $X'$ whose endpoints 
lie in $P\times B$.
\end{definition}

The construction of fibrewise fundamental groupoid extends
in an evident way to a map of symmetric monoidal fibrations
\[\xymatrix{
	\calHbp 
	\ar[rr]^{\Pi_1}
	\ar[dr]
	&&
	\calC
	\ar[dl]
	\\
	&
	\calU
}\]
To see that $\Pi_1$ preserves cartesian arrows,
one makes use of the fact that base change in $\calS$ preserves
both limits and colimits
(see \cite[Remark~2.1.9]{MaySigurdsson}).
Our next goal is to show that $\Pi_1$ in fact takes values in 
$\calG$.
To do that, we will exploit a local
triviality property enjoyed by families of h-graphs 
with basepoints over good base spaces.

\begin{definition}
\label{def:trivialization}
A \emph{trivialization} of a family of h-graphs with basepoint
$(X,P,u)$ over 
a base space $B$ consists of an h-graph with basepoints $(X_0,P,u_0)$
and a map $f\colon X_0 \times B \to X$ which fits into 
a commutative diagram
\begin{equation}
\label{diag:trivialization}
\vcenter{\xymatrix@!C@C-2em{
	& P\times B
	\ar[dl]_{u_0\times 1}
	\ar[dr]^u
	\\
	X_0\times B
	\ar[rr]^f
	\ar[dr]_{\pr}
	&&
	X
	\ar[dl]
	\\
	&
	B
}}
\end{equation}
and which for each $b\in B$ retricts to a homotopy equivalence
\[
	f|\colon X_0 \xto{\ \homot\ } X_b.
\]
\end{definition}

\begin{lemma}
\label{lm:localtriviso}
Suppose $B$ is a good base space. 
Then a trivialization of $(X,P,u)$ as above induces an isomorphism
\[
	\Pi_1(X_0,P) \times B \xto{\ \isom\ } \Pi_1(X,P)
\]
in $\calC_B$.
\end{lemma}
\begin{proof}
Let $X'_0$ and $X'$ be the mapping cylinders of the maps
$u_0\colon P \to X_0$  and  $u\colon P\times B \to X$, 
respectively. Then the maps $X'_0 \times B \to B$ and 
$X' \to B$ are fibrations, the latter by 
\cite[Proposition~1.3]{Clapp}; the inclusions of $P\times B$
to $X'_0 \times B$ and $X'$ are closed fibrewise cofibrations;
and $f$ induces a map $f'\colon X'_0 \times B \to X'$ which 
makes the diagram
\begin{equation}
\label{diag:localtriviso}
\vcenter{\xymatrix@!C@C-2em{
	& P\times B
	\ar[dl]
	\ar[dr]
	\\
	X'_0\times B
	\ar[rr]^{f'}
	\ar[dr]_{\pr}
	&&
	X'
	\ar[dl]
	\\
	&
	B
}}
\end{equation}
commutative and which restricts to a homotopy 
equivalence
\[
	f'| \colon X'_0 \xto{\ \homot\ } X'_b
\]
for each $b\in B$. 
As $B$ is locally contractible and paracompact, 
Lemma~\ref{lem:equivalence-over-and-under}
implies that the map $f'$ is a homotopy equivalence 
over $B$ and under $P\times B$. The claim now follows easily.
\end{proof}

\begin{lemma}
\label{lm:localtrivofhgraphs}
Suppose $(X,P,u)$ is a family of h-graphs with basepoints over 
a good base space $B$. Then each point $b\in B$
has an open neighborhood $U$ 
such that the restriction of $(X,P,u)$
to $U$ admits a trivialization.
\end{lemma}
\begin{proof}
Let $b\in B$. Since the space $B$ is locally contractible,
we can find an open neighborhood $U$ of $b$ such that 
the restriction $X|U$ of $X$ to $U$ admits a 
fibrewise homotopy equivalence over $U$
\[
	f\colon F\times U \longto X|U
\]
for some h-graph $F$, which we may assume to be a CW complex.
Using a fibrewise homotopy inverse of $f$, we can lift the map
\[
	u|\colon P\times U \longto X|U
\]
to a map
\[
	v\colon P\times U \longto F\times U
\]
such that $fv$ is fibrewise homotopic to $u$. Since $F$ is 
locally contractible, by shrinking $U$ and modifying $v$ by 
a fibrewise homotopy, we may assume that $v$ is of the form
$v = v_0 \times 1$ for some set of basepoints $v_0 \colon P\to F$
for $F$. Let 
\[ 
	H \colon P\times U \longto X|U
\] 
be a fibrewise homotopy connecting $f\circ(v_0\times 1)$ and $u$,
and let $F'$ denote the mapping cylinder of $v_0$. Then $f$ 
and $H$ together define a map $f'$ fitting into a commutative
diagram
\[\xymatrix@!C@C-2em{
	& P\times U
	\ar[dl]_{v'_0\times 1}
	\ar[dr]^{u|}
	\\
	F'\times U
	\ar[rr]^{f'}
	\ar[dr]_{\pr}
	&&
	X|U
	\ar[dl]
	\\
	&
	U
}\]
where $v'_0 \colon P \to F'$ is the inclusion of $P$ as a subset
of the mapping cylinder $F'$. Notice that $f'$ is a 
homotopy equivalence when restricted to fibres. 
Thus $(F',P,v'_0)$ and $f'$ give the desired trivialization.
\end{proof}

Lemmas \ref{lm:localtrivofhgraphs} and \ref{lm:localtriviso}
imply that $\Pi_1(X,P)$ lies in $\calG$
whenever $(X,P)$ is an object of $\calHbp$, giving us the desired
map of symmetric monoidal fibrations
\[\xymatrix{
	\calHbp
	\ar[rr]^{\Pi_1}
	\ar[dr]
	&&
	\calG
	\ar[dl]
	\\
	&
	\calU
}\]

\subsection{Comparing $X$ and $B\Pi_1(X,P)$}
\label{subsec:X-to-B-Pi-One-X}

Given a category $\sC$ internal to spaces over a base space $B$, 
the classifying space $B\sC$ of $\sC$
(considered as a category internal to spaces) is naturally a space 
over $B$, with the fibre over $b\in B$ given by the classifying space 
$B(\sC_b)$ of the fibre of $\sC$ over $b$.
One can verify that the classifying space
construction gives a morphism of symmetric monoidal fibrations 
with base $\calU$
\[\xymatrix{
	(\calC,\usqcup)
	\ar[rr]^B
	\ar[dr]
	&&
	(\calS,\usqcup)
	\ar[dl]
	\\
	&
	\calU
}\]
and this morphism restricts to give the functor 
$B \colon \calG \to \calS$ 
appearing in diagram~\eqref{diag:left-part}.
Our next goal is to complete the construction of 
diagram~\eqref{diag:left-part} by giving a proof of the 
following proposition.

\begin{proposition}
\label{prop:pi1cmp}
There is a zig-zag of symmetric monoidal 
transformations of fibrations over $\calU$
\[
	U \xot{\ \homot\ } WQ \xto{\ \homot\ } B\Pi_1
\]
between the composite functor 
$B\circ\Pi_1\colon \calHbp \to \calS$ 
and the forgetful functor $U\colon\calHbp\to\calS$.
The zig-zag consists of fibrewise homotopy equivalences.
\end{proposition}

The proof of Proposition~\ref{prop:pi1cmp} occupies the 
rest of this subsection. 
Let $(X,P,u)$ be a family of h-graphs with basepoints 
over a good base space $B$.
Then $B\Pi_1(X,P)$ is another family of h-graphs over $B$,
with a canonical choice of basepoints 
\[
	v\colon P\times B \longto B\Pi_1(X,P).
\]
(For the fibres of $B\Pi_1(X,P)$ are certainly h-graphs,
and the map $B\Pi_1(X,P) \to B$ is a fibration by the
local triviality of $\Pi_1(X,P)$ and the paracompactness of $B$.)
Moreover, there is a canonical map
of fibrewise fundamental groupoids
\[
	\iota \colon \Pi_1(X,P)\longto\Pi_1(B\Pi_1(X,P),P)
\]
that is the identity on objects,
and that sends a morphism in $\Pi_1(X,P)$ to the homotopy class
of paths arising from the corresponding $1$-simplex in $B\Pi_1(X,P)$.

\begin{lemma}
\label{lem:iota}
The functor $\iota$ is an isomorphism.
\end{lemma}
\begin{proof}
By the local triviality of fibrewise fundamental groupoids
over good base spaces, 
it will suffice to check
that $\iota$ restricts to give an isomorphism in each fibre.
Equivalently it suffices to check the claim when $B$ is a single point.
In this case $X$ is a single h-graph, which we can assume to be connected.
Then $\Pi_1(X,P)$ and $\Pi_1(B\Pi_1(X,P),P)$
are connected groupoids with objects identified by $\iota$,
so it suffices to check the claim after restricting $\iota$ to the
automorphism group of a single object $p_0\in P$.
But this restriction may be identified with the natural map
\[
	\pi_1(N\Pi_1(X,P),p_0)\longto \pi_1(|N\Pi_1(X,P)|,p_0),
\]
which is an isomorphism since $N\Pi_1(X,P)$ is Kan.
\end{proof}

\begin{definition}
\label{def:wxp}
Let $W(X,P)$ denote following space over $B$.
First form the space of tuples
\[(x,\gamma,p,\delta,y)\]
in which $x$, $p$ and $y$ are points of $X$, $P$ and $B\Pi_1(X,P)$
respectively, and $x$ and $y$ belong to fibres over 
the same point $b\in B$;
$\gamma$ is a homotopy class of paths in $X_b$ from $x$ to $u(p,b)$;
and $\delta$ is a homotopy class of paths in $B\Pi_1(X,P)_b$
from $y$ to $v(p,b)$.
Topologize this space as a quotient of a subspace of
\[
	\Map_B(I\times B,X)
		\times_X (P\times B) \times_{B\Pi_1(X,P)}
	\Map_B(I\times B,B\Pi_1(X,P)).
\]
Now we form a further quotient of this space by identifying
a point $(x,\gamma,p,\delta,y)$ over $b\in B$ with
\[(x,\lambda\cdot\gamma,q,\iota(\lambda)\cdot\delta,y)\]
whenever we have a morphism $\lambda\colon p\to q$ in $\Pi_1(X,P)_b$.
The resulting space is $W(X,P)$.
\end{definition}

A map of families of h-graphs with basepoints
$(X,P)\to (Y,Q)$ induces in an evident way a map $W(X,P)\to W(Y,Q)$,
giving us a functor
\[
	W \colon \calHbp \longto \calS
\]
which one can verify to be a map of symmetric monoidal fibrations 
with base category $\calU$. The fact that $W$ preserves cartesian 
morphisms again follows from the fact that base change in $\calS$ 
preserves limits and colimits. We make the following observation,
which in particular implies that $W$ turns homotopy equivalences
over $B$ and under $P\times B$ 
(see Definition~\ref{def:equivalence-over-and-under}) 
into fibrewise homotopy equivalences.
\begin{lemma}
Suppose $(X,P)$ and $(Y,P)$ are families of h-graphs over 
a good base space $B$, and suppose 
\[
	(f_0,1,1),(f_1,1,1)\colon (X,P) \to (Y,P)
\]
are maps of families of h-graphs with basepoints such that
$f_0$ and $f_1$ are homotopic through maps over $B$ and under $P\times B$.
Then the maps $W(f_0,1,1)$ and $W(f_1,1,1)$ 
are homotopic through maps over $B$. \qed
\end{lemma}

There are natural maps $\alpha\colon W(X,P)\to X$
and $\beta\colon W(X,P)\to B\Pi_1(X,P)$
sending an element $[x,\gamma,p,\delta,y]$
to $x$ and $y$, respectively, giving us a zig-zag
\[
	U \xot{\ \alpha\ } W \xto{\ \beta\ } B\Pi_1
\]
of symmetric monoidal transformations of fibrations over $\calU$. 
This zig-zag, however, is not quite the one of 
Proposition~\ref{prop:pi1cmp}. Let 
\[
	Q \colon \calHbp \longto \calHbp
\]
be the map of symmetric monoidal fibrations sending 
a family of h-graphs with basepoints $(X,P,u)$ over $B$
to the family $(X',P,u')$ over $B$, where $X'$ is the mapping cylinder
of $u$ and $u'$ is the inclusion of $P\times B$ into $X'$;
the fact that the map $X' \to B$ is a fibration, as required
from a family of h-graphs, follows from \cite[Proposition~1.3]{Clapp}.
Let 
\[
	\eta \colon Q \longto \Id
\]
be the symmetric monoidal transformation of fibrations 
given by the collapse maps $X' \to X$. We will prove 
Proposition~\ref{prop:pi1cmp} by showing that all arrows
in the diagram
\[
	U 
	\xot{\ U\eta\ } 
	UQ 
	\xot{\ \alpha Q\ }
	WQ 
	\xto{\ \beta Q\ }
	B\Pi_1 Q 
	\xto{\ B\Pi_1\eta\ } 
	B\Pi_1
\]
consist of fibrewise homotopy equivalences.

It is evident that $U\eta$ and $B\Pi_1 \eta$ 
consist of fibrewise homotopy equivalences; indeed, $\Pi_1\eta$
is even a natural isomorphism. It therefore remains to show that
the maps $\alpha$ and $\beta$ in 
\begin{equation}
\label{diag:wzigzag'}
	X' \xot{\ \alpha\ } W(X',P) \xto{\ \beta\ } B\Pi_1(X',P)
\end{equation}
are fibrewise homotopy equivalences when $(X',P) = Q(X,P)$
for some family of h-graphs with basepoints $(X,P)$ over 
a good base space $B$. By \cite[Theorem~6.3]{Dold},
a map of Dold fibrations
\[\xymatrix{
	E 
	\ar[rr]^f
	\ar[dr]
	&&
	E'
	\ar[dl]
	\\
	&
	B		
}\]
over a good base space $B$ 
is a fibrewise homotopy equivalence exactly when it
restricts to an ordinary homotopy equivalence $f|\colon E_b \to E'_b$ 
on fibres for all $b\in B$. Thus the following lemma
implies that it is enough to consider the case where the 
base space $B$ is a point.

\begin{lemma}
\label{lm:doldfib}
Suppose $(X,P)$ is a family of h-graphs with basepoints 
over a good base space $B$. Then the map $WQ(X,P) \to B$
is a Dold fibration.
\end{lemma}
\begin{proof}
Consider first the special case where $(X,P)$ admits a trivialization
as in Definition~\ref{def:trivialization}. As in the proof of 
Lemma~\ref{lm:localtriviso}, applying $Q$ to the
diagram~\eqref{diag:trivialization} gives us a diagram as in
\eqref{diag:localtriviso}, with the map $f'$ a homotopy equivalence
over $B$ and under $P\times B$. Applying $W$ to this diagram
then gives us a fibrewise homotopy equivalence from 
a trivial space over $B$ to $WQ(X,P)$, and \cite[Corollary 5.3]{Dold} 
implies that $WQ(X,P) \to B$ is a Dold fibration
in this special case. In view of Lemma~\ref{lm:localtrivofhgraphs} and 
the paracompactness of $B$, the claim now follows in the general case 
from \cite[Theorem 5.12]{Dold}.
\end{proof}

Suppose now $(X,P,u)$ is a family of h-graphs with basepoints over a point.
We can then find a homotopy equivalence $f$ from $X$ to a finite
CW complex $Y$ of dimension $\leq 1$. We equip $Y$ with the basepoints
$fu \colon P \to Y$. Let $(X',P) = Q(X,P)$ and $(Y',P) = Q(Y,P)$.
Then $f$ induces a homotopy equivalence $f'\colon X'\to Y'$
under $P$ (and over the point),
and we have the diagram
\[\xymatrix{
	X'
	\ar[d]_\homot
	&
	\ar[l]_-\alpha
	W(X',P)
	\ar[d]^\homot
	\ar[r]^-\beta
	&
	B\Pi_1(X',P)
	\ar[d]^\homot
	\\
	Y'
	&
	\ar[l]_-\alpha
	W(Y',P)
	\ar[r]^-\beta
	&
	B\Pi_1(Y',P)
}\]
where the vertical maps are homotopy equivalences induced by $f'$.
Thus $\alpha$ and $\beta$ in the top row are homotopy equivalences
precisely when $\alpha$ and $\beta$ in the bottom row are.
Observe that $Y'$ can again be given the structure of 
a CW complex of dimension $\leq 1$.
Thus we are left with verifying the following lemma.

\begin{lemma}
\label{lm:alphabetaheq}
The maps $\alpha$ and $\beta$ in the zig-zag \eqref{diag:wzigzag'}
are homotopy equivalences when $B$ is a point and $X'$
is a CW complex of dimension $\leq 1$.
\end{lemma}
\begin{proof}
We prove the claim for $\alpha$, the proof for $\beta$ being 
identical. Observe that $\alpha$ is a fibre bundle.
This is a direct consequence of the fact that $X'$ is 
locally contractible. We will show that the fibres of $\alpha$ are
contractible. Since $X'$ is paracompact the result will then follow,
for example by \cite[Corollary 3.2]{Dold}.

The fibre of $\alpha$ over a point $x_0 \in X'$ is the space of
equivalence classes of 
tuples $(\gamma,p,\delta,y)$ where $p\in P$, $y\in B\Pi_1(X',P)$,
$\gamma$ is a homotopy class of paths from $x_0$ to $u(p)$,
and $\delta$ is a homotopy class of paths from $y$ to $v(p)$.
Such a tuple $(\gamma,p,\delta,y)$ is identified with
$(\varepsilon\cdot \gamma,q,\iota(\varepsilon)\cdot \delta,y)$
whenever we have a morphism
$\varepsilon\colon p\to q$ in $\Pi_1(X',P)$. 
By choosing a homotopy class of paths from $x_0$ to $u(p_0)$
for some $p_0\in P$, we can identify the fibre with a quotient space
of tuples described in an identical way, except that $\gamma$ is now a
morphism $\gamma\colon p_0\to p$ in $\Pi_1(X',P)$.

The resulting space is homeomorphic to the space of pairs
$(\delta,y)$ in which $y\in B\Pi_1(X',P)$ and $\delta$ is a homotopy
class of paths from $y$ to $v(p_0)$. 
This is the universal cover of the component of $B\Pi_1(X',P)$ containing
$p_0$.
It is contractible, since the $p_0$-component of $B\Pi_1(X',P)$
is a CW-complex of type $K(\pi_1(X',u(p_0)),1)$. 
This completes the proof.
\end{proof}

The proof of Proposition~\ref{prop:pi1cmp} is now complete.

\subsection{From fibrewise groupoids to fibrewise manifolds}
\label{subsec:groupoids-to-manifolds}
The last subsection completed the construction of diagram~\eqref{diag:left-part},
and by taking fibrewise opposites we obtain the left-hand square
of diagram~\eqref{diagram-one}.
In this subsection we construct the symmetric monoidal functor 
\[
	B(G^{(-)}) \colon (\calG,\usqcup) \longto (\calM,\times)
\]
appearing in the
right-hand square of~\eqref{diagram-one}, the other functors in that square 
having already been constructed.

\begin{definition}
Let $\sP$ be a fibrewise finite free groupoid over a good base space $B$.
We denote by
\[
	G^\sP = \Fun_B(\sP,G\times B)
\]
the functor category (internal to spaces over $B$) whose fibre over a point $b\in B$
is the category (internal to spaces) of functors $\sP_b\to G$,
where we consider $G$ as a one-object category internal to spaces.
Regarding $\Ob(\sP)$ as a fibrewise finite free groupoid with only identity morphisms,
we obtain a restriction functor
\[
	G^\sP \longrightarrow G^{\Ob(\sP)}
\]
and by taking classifying spaces, we obtain a map of spaces
\[
	B(G^\sP) \longrightarrow B(G^{\Ob(\sP)}),
\]
which we regard as an object of $\widehat\calS$.
(In the cases of interest we have $\Ob(\sP)=P\times B$ for some finite set $P$,
and then $B(G^{\Ob(\sP)})=BG^P\times B$.)
The assignment $\sP \mapsto \bigl(B(G^\sP) \to B(G^{\Ob(\sP)})\bigr)$
extends naturally to a symmetric monoidal functor
\[
	B(G^{(-)})\colon (\calG^\fop,\usqcup) \longrightarrow (\widehat\calS,\times).
\]
Observe that this functor is not a morphism of fibrations, as it 
does not preserve base spaces.
\end{definition}

\begin{definition}
\label{def:rank-function}
If $\sP$ is a fibrewise finite free groupoid over a base $B$,
we write $r(\sP)\colon B\to\Z$ for the locally constant function sending
an element $b\in B$ to the rank of $\sP_b$.
Precomposing $r(\sP)$ with the projection 
$B(G^{\Ob(\sP)})\to B$, we obtain a locally
constant function $B(G^{\Ob(\sP)})\to\Z$ which we again denote $r(\sP)$.
\end{definition}

\begin{proposition}\label{prop:functors-manifolds}
The functor $B(G^{(-)})\colon (\calG^\fop,\usqcup) \to (\widehat\calS,\times)$
lifts to a  symmetric monoidal functor
\[
	B(G^{(-)})\colon (\calG^\fop,\usqcup) \longrightarrow (\calM,\times).
\]
The fibrewise dimension of $B(G^\sP) \to B(G^{\Ob(\sP)})$ 
is given by $|B(G^\sP)| = \dim(G)\cdot r(\sP)$.
\end{proposition}

Let $\sP$ be a finite free groupoid (not fibrewise)
and write $\fun(\sP,G)$ for the topological {space}
of functors from $\sP$ to $G$, or in other words the space of objects
in $G^\sP$.
Evaluation on a morphism $\gamma$ of $\sP$ determines
a map $\varepsilon_\gamma\colon\fun(\sP,G)\to G$.
In the proof of Proposition~\ref{prop:functors-manifolds}, 
we will make use of the following lemma.

\begin{lemma}
\label{lem:single-smooth-structure}
The space $\fun(\sP,G)$ is a closed topological manifold of dimension 
$\dim(G)\cdot r$, where $r$ is the rank of $\sP$.
It admits a unique smooth structure
with the property that a map into $\fun(\sP,G)$ is smooth
if and only if its composition with $\varepsilon_\gamma$
is smooth for every morphism $\gamma$ of $\sP$.
\end{lemma}

\begin{proof}
If $\Gamma$ is a basis of $\sP$, then by evaluating functors on the elements
of $\Gamma$ one obtains a homeomorphism $\fun(\sP,G)\homeom G^\Gamma$,
and the first claim follows.
Using this homeomorphism and the smooth structure of $G^\Gamma$,
we obtain a smooth structure on $\fun(\sP,G)$
characterised by the fact that a map into $\fun(\sP,G)$
is smooth if and only if its composition with $\varepsilon_\gamma$ is smooth
for every element $\gamma\in\Gamma$.
Since by Lemma~\ref{lm:basis-words} every morphism of $\sP$ 
can be written as a word in the elements of $\Gamma$ and their inverses,
the given characterisation of the smooth structure on $\fun(\sP,G)$
follows.  In particular the smooth structure on $\fun(\sP,G)$ is independent
of the choice of basis.
\end{proof}

\begin{proof}[Proof of Proposition~\ref{prop:functors-manifolds}]
We begin by constructing a lift
$\ell_\pt\colon\calG^\fop_\pt\to\calM$
of the restriction of $B(G^{(-)})$ to the fibre $\calG^\fop_\pt$
of $\calG^\fop\to\calU$ over $\pt$.

Let $\sP$ be a single finite free groupoid.
Then $G^{\Ob(\sP)}$ is the topological group of functions from $\Ob(\sP)$ into $G$,
and it acts on $\fun(\sP,G)$ according to the rule
\[
	(\delta\cdot f)(\gamma)
	=
	\delta(y)\cdot f(\gamma)\cdot\delta(x)^{-1}
\]
for $f\in \fun(\sP,G)$, $\delta\in G^{\Ob(\sP)}$, and $\gamma\colon x\to y$
a morphism of $\sP$.
With this action the category of functors $G^\sP$ may be identified with the
action-groupoid $\fun(\sP,G) / G^{\Ob(\sP)}$.
Consequently the map
\[
	B(G^\sP) \longto B(G^{\Ob(\sP)})
\]
may be identified with the map
\begin{equation}
\label{eq:borel-construction}
	E(G^{\Ob(\sP)}) \times_{G^{\Ob(\sP)}} \fun(\sP,G)
	\longto
	B(G^{\Ob(\sP)}).
\end{equation}
This identification is natural in $\sP$.

The map \eqref{eq:borel-construction} is a fibre bundle.
Its structure group $G^{\Ob(\sP)}$ is a compact Lie group,
and by Lemma~\ref{lem:single-smooth-structure}
its fibre $\fun(\sP,G)$ is a smooth closed manifold
on which  the structure group acts smoothly.
This observation equips $B(G^\sP)\to B(G^{\Ob(\sP)})$
with the structure of a fibrewise closed manifold.
A morphism of finite free groupoids induces a fibrewise smooth map,
again by Lemma~\ref{lem:single-smooth-structure}.
This completes the construction of $\ell_\pt$.

We next prove that $\ell_\pt$ extends to a lift 
$\ell_\mathrm{triv}\colon\calG_{\mathrm{triv}}^\fop\to\calM$
of the restriction of $B(G^{(-)})$ to the full subcategory
consisting of trivial
fibrewise finite free groupoids, that is, ones of the 
form $\sP \times B$ for some finite free groupoid $\sP$ and
good base space $B$.
The existence of $\ell_\mathrm{triv}$
is immediate from the case already proved, since 
if $\sP\times B$ is a trivial fibrewise finite free groupoid over a good base $B$,
then $B(G^{\sP\times B})\to B(G^{\Ob(\sP\times B)})$ may be identified with
$B(G^\sP)\times B\to B(G^{\Ob(\sP)})\times B$, while a functor
$F\colon \sP\times B\to \sQ\times C$ over $f\colon B\to C$
is locally the product of $f$ with a functor $\sP\to\sQ$.

Finally, we construct the full lift $\ell$.
If $\sP$ is an arbitrary fibrewise finite free groupoid over a good base $B$,
then we may locally make $B(G^\sP)\to B(G^{\Ob(\sP)})$ into a fibrewise smooth
manifold by choosing local trivialisations of $\sP$ and using $\ell_\mathrm{triv}$.
By functoriality of $\ell_\mathrm{triv}$
these local smooth structures patch together,
and this defines $\ell$ on objects.
That induced maps are fibrewise smooth
again follows from functoriality of $\ell_\mathrm{triv}$,
so that $\ell$ extends to morphisms.

It is easy to verify that $\ell$ is  symmetric monoidal.
The claim about the fibrewise dimension 
follows from the identification of $B(G^\sP) \to B(G^{\Ob(\sP)})$
with \eqref{eq:borel-construction} together with Lemma~\ref{lem:single-smooth-structure}.
\end{proof}

\subsection{From functors to mapping spaces}
\label{subsec:functors-to-mapping-spaces}
In this subsection, we complete the construction of 
diagram~\eqref{diagram-one}
by proving the following result.

\begin{proposition}
\label{prop:funmapcmp}
There is a symmetric monoidal natural transformation $\Theta$
fitting into the diagram
\begin{equation}
\label{diag:Theta}
\vcenter{\xymatrix@R+5ex{
	(\calG,\usqcup)^\fop
	\ar[d]_{B^\fop}
	\ar[rr]^{B(G^{(-)})}
	&&
	(\calM,\times)
	\ar[d]^{\forget}
	\ar@{=>}[dll]^{\homot}_\Theta
	\\
	(\calS,\usqcup)^\fop
	\ar[r]_(0.53){BG^{(-)}}
	&
	(\calS,\times)
	\ar[r]_{\inclusion}
	&
	(\widehat\calS,\times)
}}
\end{equation}
and consisting of homotopy equivalences of total spaces.
\end{proposition}

Let $\sP$ be a fibrewise finite free groupoid
over a good base $B$.  Then the composite 
\[
	\forget\circ B(G^{(-)})
\] 
in diagram~\eqref{diag:Theta} sends $\sP$ to the fibred space 
$B(G^\sP) \to B(G^{\Ob\sP})$, 
while the composite 
\[
	\inclusion \circ BG^{(-)}\circ B^\fop
\]
sends $\sP$ to the fibred space $BG^{B\sP} \to B$.
The component at $\sP$ of the required symmetric monoidal natural transformation 
$\Theta$ is then the commutative square (regarded as a morphism of
fibred spaces)
\[\xymatrix{
	B(G^\sP)
	\ar[r]^{\theta_\sP}\ar[d]
	&
	BG^{B\sP}
	\ar[d]
	\\
	B(G^{\Ob\sP})
	\ar[r]
	&
	B
}\]
where the upper map $\theta_\sP$ is defined as follows.
There is an evaluation functor $ G^\sP\times_B \sP \to G\times B$ 
of categories internal to spaces over $B$, inducing a map
$B(G^\sP\times_B \sP)\to B(G\times B)$ of classifying spaces.
Since geometric realisation of simplicial spaces commutes with fibre products,
this can be regarded as a map
$B(G^\sP)\times_B B\sP \to BG\times B$,
and we define $\theta_\sP$ to be its adjoint
$B(G^\sP)\to \Map_B(B\sP,BG\times B)=BG^{B\sP}$.
It is routine to check that these morphisms do indeed assemble to a symmetric
monoidal natural transformation $\Theta$.

It remains to show that each $\theta_\sP$ is a homotopy equivalence.
It will suffice to consider the case where $B$ is a single point.
For in the general case we may regard $\theta_P$ as a morphism between
spaces over $B$, which by local triviality of $\sP$ are fibre bundles,
and hence fibrations since $B$ is paracompact \cite[Theorem~4.8]{Dold}.
As in the discussion before Lemma~\ref{lm:doldfib},
it therefore suffices to show that $\theta_\sP$
is a homotopy equivalence in every fibre, or equivalently
that it is a homotopy equivalence when $B$ is a single point.
We will do so in Lemma~\ref{lm:theta-eq} below after some
preparatory work which we now commence.

Let $\sP$ be a single finite free groupoid and let $\Gamma$ be a basis for $\sP$.
Then we define $X_\Gamma$ 
to be the space obtained by starting with a copy of $\Ob\sP$
and adding arcs, one for each morphism in $\Gamma$,
attached at the points corresponding to the source
and target of the morphism.
There is a natural morphism
\[
	\alpha_{\sP,\Gamma}\colon\sP\longrightarrow \Pi_1(X_\Gamma,\Ob\sP)
\]
given by the identity on objects and sending an element
of $\Gamma$ to the homotopy class of the corresponding arc of $X_\Gamma$.

\begin{lemma}
The morphism $\alpha_{\sP,\Gamma}$ is an isomorphism.
\end{lemma}
\begin{proof}
We may assume that $\sP$ is connected.
It is enough to show that $\alpha_{\sP,\Gamma}$ is an equivalence,
since it is an isomorphism on objects.
If $\sP$ has more than one object, 
then we may find an element $\gamma\in\Gamma$
whose source and target are distinct.
Let $\sQ$ denote the finite free groupoid with two objects and basis
consisting of a single morphism from one object to the other,
and let $\bar\sP$ denote the pushout
\[\xymatrix{
	\sQ\ar[r]\ar[d] & \sP\ar[d]\\
	\ast\ar[r] & \bar\sP
}\]
where $\ast$ denotes the trivial one-object groupoid and 
the top horizontal map sends the basis element of $\sQ$ to $\gamma$.
Then $\bar\sP$ is finite free with basis $\bar\Gamma=\Gamma\setminus\{\gamma\}$,
the morphism $\sP\to\bar\sP$ is an equivalence, and the corresponding collapse map
$X_\Gamma\to X_{\bar\Gamma}$ is a homotopy equivalence.
Thus there is a commutative square of groupoids
\[\xymatrix{
	\sP
	\ar[r]^-{\alpha_{\sP,\Gamma}}
	\ar[d]
	&
	\Pi_1(X_\Gamma,\Ob\sP)
	\ar[d]
	\\
	\bar\sP
	\ar[r]_-{\alpha_{\bar\sP,\bar\Gamma}}
	&
	\Pi_1(X_{\bar\Gamma},\Ob\bar\sP)
}\]
in which the vertical arrows are equivalences, so that the upper arrow
is an equivalence if and only if the lower arrow is an equivalence.
Iterating this step, we see that it is enough to show that 
$\alpha_\sP$ is an equivalence when $\sP$ has a single object.
In this case the result is immediate since $\sP$ is the free
group on the elements of $\Gamma$ and $X_\Gamma$ is a wedge of $\Gamma$ circles.
\end{proof}

Now we consider the inclusion
\[
	\beta_{\sP,\Gamma}
	\colon
	X_\Gamma
	\longto
	B\sP
\]
that sends elements of $\Ob\sP$ to the corresponding vertices
of $B\sP$, and that sends the arc corresponding to an element of $\Gamma$
to the corresponding edge of $B\sP$.

\begin{lemma}
\label{lm:graphmodel}
The map $\beta_{\sP,\Gamma}$ is a homotopy equivalence.
\end{lemma}

\begin{proof}
Since the domain and range of $\beta_{\sP,\Gamma}$ are both
disjoint unions of spaces of the type $K(\pi,1)$, it will be enough
to show that the induced functor
\[(\beta_{\sP,\Gamma})_\ast\colon \Pi_1(X_\Gamma,\Ob\sP)\longto\Pi_1(B\sP,\Ob\sP)\]
is an isomorphism.
By the previous lemma it will suffice to show that the composite
\[
	(\beta_{\sP,\Gamma})_\ast\circ\alpha_{\sP,\Gamma}
	\colon
	\sP
	\longto
	\Pi_1(B\sP,\Ob\sP)
\]
is an isomorphism.
Now this composite is given by the identity on objects and sends a
morphism to the homotopy class of the corresponding edge of $B\sP$
(the last claim holds by definition for elements of $\Gamma$
and then follows for arbitrary morphisms).
That $(\beta_{\sP,\Gamma})_\ast$ is an isomorphism now follows by 
the same argument we used to prove Lemma~\ref{lem:iota}.
\end{proof}

\begin{lemma}
\label{lm:theta-eq}
Let $\sP$ be a single finite free groupoid.
Then the map $\theta_\sP$ is a homotopy equivalence.
\end{lemma}

\begin{proof}
Let $\Gamma$ be a basis of $\sP$.
The homotopy equivalence $\beta_{\sP,\Gamma}$ induces
a homotopy equivalence
$\beta_{\sP,\Gamma}^\ast\colon BG^{B\sP}\to BG^{X_\Gamma}$,
and so it will suffice to show that the composite
\[
	\phi_{\sP,\Gamma}
	\colon
	B(G^\sP)
	\longto
	BG^{X_\Gamma}
\]
of $\theta_\sP$ with $\beta_{\sP,\Gamma}^\ast$ is a homotopy equivalence,
or equivalently (since it is a map between fibrations over $BG^{\Ob\sP}$),
that it is a fibrewise homotopy equivalence over $BG^{\Ob\sP}$.

We again denote by $\sQ$ the groupoid with two objects and basis
$\Delta$ consisting of a single isomorphism from one object to the other.
Then we have two pushout diagrams, the first by the definition of basis
and the second by construction.
\[
\xymatrix{
	\bigsqcup_{\gamma\in\Gamma} \Ob\sQ
	\ar[r]\ar[d]
	&
	\bigsqcup_{\gamma\in\Gamma} \sQ
	\ar[d]
	\\
	\Ob\sP\ar[r]
	&
	\sP
}
\qquad\qquad
\xymatrix{
	\bigsqcup_{\gamma\in\Gamma} \Ob\sQ
	\ar[r]\ar[d]
	&
	\bigsqcup_{\gamma\in\Gamma} X_\Delta
	\ar[d]
	\\
	\Ob\sP\ar[r]
	&
	X_\Gamma
}
\]
It follows that $\phi_{\sP,\Gamma}$ is the pullback,
along the map $BG^{\Ob\sP}\to BG^{\bigsqcup\Ob\sQ}$,
of the direct product of $\Gamma$ copies of the map
$\phi_{\sQ,\Delta}$.
It will therefore suffice to show that $\phi_{\sQ,\Delta}$ is a fibrewise homotopy 
equivalence of spaces over $BG^{\Ob\sQ}$, 
or equivalently that it is a homotopy equivalence of total spaces.
Let $\ast$ denote the trivial groupoid with one object and empty basis.
The inclusion $\ast\to\sQ$ of a single object (which is an equivalence)
and the inclusion $X_\emptyset\hookrightarrow X_\Delta$ of a single point
(which is a homotopy equivalence)
induce a commutative square
\[\xymatrix{
	B(G^\sQ)
	\ar[r]^{\phi_{\sQ,\Delta}}
	\ar[d]
	&
	BG^{X_\Delta}
	\ar[d]
	\\
	B(G^\ast)
	\ar[r]_{\phi_{\ast,\emptyset}}
	&
	BG^{X_\emptyset}
}\]
in which the vertical maps are homotopy equivalences and the
lower map is an isomorphism by construction.
It follows that $\phi_{\sQ,\Delta}$ is a homotopy equivalence, as required.
\end{proof}

The proof of Proposition~\ref{prop:funmapcmp} is now complete.

\subsection{Double categories of special squares}
\label{subsec:special-squares}
Our aim now is to construct diagram~\eqref{diagram-two}
from diagram~\eqref{diagram-one}.
The symmetric monoidal double categories
$\bbS^d(\calH^\fop)$ and $\bbS^\ds(\calM)$ appearing there
have already been defined
(in Examples~\ref{ex:dc-special-squares-h-graphs}
and~\ref{ex:sm-dc-special-squares-h-graphs}
and Definition~\ref{def:sm-dc-special-squares-manifolds})
and our aim in this subsection is to define the remaining
symmetric monoidal double categories
$\bbS^d(\calHbp^\fop)$ and $\bbS^d(\calG^\fop)$.
We will employ a strategy similar to the one we used to construct
$\bbS^d(\calH^\fop)$, and so the reader may find it helpful 
recall the definition of $\bbS^d(\calH^\fop)$ from
Examples~\ref{ex:dc-special-squares-h-graphs}
and~\ref{ex:sm-dc-special-squares-h-graphs}.

We begin with the construction of $\bbS^d(\calHbp^\fop)$.
First observe that the fibre product $\calHbp^\fop \times_\calB \calB^\ds$
of $\calHbp^\fop$ with the category $\calB^\ds$ defined in 
Example~\ref{ex:dc-special-squares-h-graphs}
can be described as follows.
\begin{itemize}
\item The objects of $\calHbp^\fop \times_\calB \calB^\ds$ are quadruples
	$(X,P,u,m)$, where $(X,P,u)$ is a family of h-graphs with basepoints
	over a good base space $B$, and $m\colon B \to \Z$
	is a locally constant function.
\item A morphism in $\calHbp^\fop \times_\calB \calB^\ds$ 
	from $(X_1, P_1, u_1, m_1)$ to $(X_2, P_2, u_2, m_2)$
	is simply a morphism from $(X_1,P_1,u_1)$  to $(X_2,P_2,u_2)$
	in $\calHbp^\fop$.
\end{itemize}
Also observe that the symmetric monoidal structures on $\calHbp$ and $\calB^\ds$ 
induce on $\calHbp \times_\calB \calB^\ds$ a symmetric 
monoidal structure given by the product
\[
	(X_1, P_1, u_1, m_1)\otimes (X_2, P_2, u_2, m_2)
		=
	(X_1\usqcup X_2, P_1\sqcup P_2, u_1\usqcup u_2, m_1+m_2).
\]

\begin{definition}[The double category $\bbS^d(\calHbp^\fop)$]
\label{def:dc-special-squares-h-graphs}
The symmetric monoidal
\emph{double category of special squares in $\calHbp^\fop$ with degree-shifts}, denoted $\bbS^d(\calHbp^\fop)$, is the sub-double category of
$\sq(\calHbp^\fop\times_\calB\calB^\ds)$ defined as follows.
\begin{itemize}
	\item
	The objects of $\bbS^d(\calHbp^\fop)$ are the objects of
	$\sq(\calHbp^\fop\times_\calB\calB^\ds)$.
	\item
	A morphism
	$(X,P,u,m)\to(Y,Q,v,n)$
	in $\calHbp^\fop\times_\calB\calB^\ds$ qualifies as a vertical
	morphism in $\bbS^d(\calHbp^\fop)$ exactly when the underlying morphism
	$(X,m)\to(Y,n)$ in $\calH^\fop\times_\calB\calB^\ds$ qualifies
	as a vertical morphism in $\bbS^d(\calH^\fop)$.
	\item
	A morphism
	$(X,P,u,m)\to(Y,Q,v,n)$
	in $\calHbp^\fop\times_\calB\calB^\ds$ qualifies as a horizontal
	morphism in $\bbS^d(\calHbp^\fop)$ exactly if 
	the underlying morphism
	$(X,m)\to(Y,n)$ in $\calH^\fop\times_\calB\calB^\ds$
	qualifies as a horizontal morphism in $\bbS^d(\calH^\fop)$
	and the map $Q\to P$ is the identity map $P\to P$.
	\item
	A 2-cell in $\sq(\calHbp^\fop\times_\calB\calB^\ds)$ qualifies as a 2-cell in
	$\bbS^d(\calHbp^\fop)$ if its vertical and horizontal edges
	are respectively vertical and horizontal morphisms in $\bbS^d(\calHbp^\fop)$
	and the underlying 2-cell in $\sq(\calH^\fop\times_\calB\calB^\ds)$ qualifies
	as a 2-cell in $\bbS^d(\calH^\fop)$.
\end{itemize}
We equip $\bbS^d(\calHbp^\fop)$ with the symmetric monoidal structure
inherited from $\sq(\calHbp^\fop\times_\calB\calB^\ds)$.
\end{definition}

Before giving the definition of $\bbS^d(\calG^\fop)$, we need the following
notion, which is analogous to the notion of {h-embedding} for h-graphs.

\begin{definition}
We say that a morphism $F\colon\sP\to\sQ$ of finite free groupoids
has the \emph{basis extension property} if for any basis $\{\gamma_i\}_{i\in I}$
of $\sP$ one may extend the indexed set $\{F(\gamma_i)\}_{i\in I}$
to a basis $\{F(\gamma_i)\}_{i\in I}\cup\{\delta_j\}_{j\in J}$ of $\sQ$.
\end{definition}

To define $\bbS^d(\calG^\fop)$, we first observe that the fibre product
$\calG^\fop \times_\calB\calB^\ds$ of $\calG^\fop$ with the category
$\calB^\ds$ of Example~\ref{ex:dc-special-squares-h-graphs} can be 
described as follows.
\begin{itemize}
	\item
	An object of $\calG^\fop \times_\calB\calB^\ds$ is a pair
	$(\sP,m)$ consisting of a fibrewise finite free  groupoid
	$\sP$ over a good base $B$ and a locally constant function
	$m\colon B\to\bbZ$.
	\item
	A morphism in $\calG^\fop \times_\calB\calB^\ds$ from $(\sP,m)$
	to $(\sQ,n)$ is simply a morphism $\sP\to\sQ$ in $\calG^\fop$,
	or in other words an equivalence class of diagrams
	\begin{equation}
	\label{eq:calG-fop-morphism}
	\vcenter{\xymatrix@C-1em{
	\sP
	\ar@{|->}[d]
	&& 
	\sU
	\ar@{|->}[d]
	\ar[ll]_\alpha
	\ar[rr]^\beta
	&&
	\sQ
	\ar@{|->}[d]
	\\
	B
	&
	&
	B
	\ar[rr]^f
	\ar[ll]_=
	&&
	C
	}}
	\end{equation}
	with $\beta$ cartesian.
\end{itemize}
The symmetric monoidal structures on $\calG^\fop$
and $\calB^\ds$ induce a symmetric monoidal structure on 
$\calG^\fop\times_\calB\calB^\ds$ given by
\[
	(\sP_1,m_1)\otimes(\sP_2,m_2)
	=
	(\sP_1\usqcup\sP_2,m_1+m_2).
\]

\begin{definition}[The double category $\bbS^d(\calG^\fop)$]
The symmetric monoidal 
\emph{double category of special squares in $\calG^\fop$ with degree shifts},
denoted $\bbS^d(\calG^\fop)$,
is the sub-double category of $\sq(\calG^\fop\times_\calB\calB^\ds)$ defined as follows.
\begin{itemize}
	\item
	The objects of $\bbS^d(\calG^\fop)$ are the objects of
	$\sq(\calG^\fop\times_\calB\calB^\ds)$.
	\item
	A morphism
	$(\sP,m)\to(\sQ,n)$ in $\calG^\fop\times_\calB\calB^\ds$  represented by a 
	diagram \eqref{eq:calG-fop-morphism} qualifies as a vertical morphism
	in $\bbS^d(\calG^\fop)$ if the restriction of $\alpha$ to each fibre over $B$
	has the basis extension property,
	and $m$ and $n$ satisfy the equation $m=n\circ f$.
	\item
	A morphism $(\sP,m)\to(\sQ,n)$
	in $\calG^\fop\times_\calB\calB^\ds$ represented by a 
	diagram \eqref{eq:calG-fop-morphism} qualifies as a horizontal morphism
	in $\bbS^d(\calG^\fop)$ if $\beta$ and $f$ are identity maps,
	$\alpha$ is the identity map on objects,
	and $m$ and $n$ satisfy
	\[
		m-n = d(r(\sP)-r(\sQ)).
	\]
	Recall from Definition~\ref{def:rank-function} 
	that $r(\sP)\colon B \to \Z$ is the locally constant
	function whose value at a point $b \in B$ is the 
	the rank of $\sP_b$.
	\item
	Finally, a 2-cell
	\[\xymatrix{
		(\sP,m) 
		\ar[r]
		\ar[d]
		&
		(\sQ,n)
		\ar[d]
		\\
		(\sR,k)
		\ar[r]
		&
		(\sS,l)
	}\]
	in $\sq(\calG^\fop\times_\calB\calB^\ds)$ 
	whose vertical and horizontal morphisms satisfy the respective
	above conditions qualifies as a 2-cell of $\bbS^d(\calG^\fop)$
	if for each $b\in B$ the commutative square
	\[\xymatrix{
		\sP_b
		&
		\sQ_b	
		\ar[l]
		\\
		\sR_{f(b)}
		\ar[u]
		&
		\sS_{f(b)}
		\ar[u]
		\ar[l]
	}\]
	of finite free groupoids is a pushout.
\end{itemize}
We equip $\bbS^d(\calG^\fop)$ with the symmetric monoidal structure
inherited from $\sq(\calG^\fop\times_\calB\calB^\ds)$.
\end{definition}

\subsection{Double functors between double categories of special squares}
\label{subsec:double-functors}
We continue the task of constructing diagram~\eqref{diagram-two}.
By now we have defined all of the symmetric monoidal double categories
appearing there, and 
our next goal is to define the  symmetric monoidal double functors
\begin{align*}
	\bbS(\forget) &\colon \bbS^d(\calHbp^\fop) \longto \bbS^d(\calH^\fop), \\
	\bbS(\Pi_1) &\colon \bbS^d(\calHbp^\fop) \longto \bbS^d(\calG^\fop) \quad \text{and}\\
	\bbS(B(G^{(-)})) &\colon \bbS^d(\calG^\fop) \longto \bbS^\ds(\calM).
\end{align*}
\begin{definition}[The symmetric monoidal double functor $\bbS(\forget)$]
\label{def:bbS-forget}
The  symmetric monoidal functor
\[\forget\colon\calHbp\longrightarrow\calH,\]
which is a morphism of symmetric monoidal fibrations over $\calU$,
induces a fibrewise opposite functor
\[\forget^\fop\colon \calHbp^\fop \longrightarrow\calH^\fop,\]
a fibred product functor
\[\forget^\fop\times_\calB 1
\colon
\calHbp^\fop\times_\calB\calB^\ds
\longrightarrow
\calH^\fop\times_\calB\calB^\ds,\]
and a double functor
\[\sq(\forget^\fop\times_\calB 1)
\colon
\sq(\calHbp^\fop\times_\calB\calB^\ds)
\longrightarrow
\sq(\calH^\fop\times_\calB\calB^\ds),\]
all of them  symmetric monoidal.
The latter restricts to a symmetric monoidal double functor
\[
	\bbS(\forget)
	\colon
	\bbS^d(\calHbp^\fop)
	\longrightarrow
	\bbS^d(\calH^\fop)
\]
of sub-double categories. 

To see that $\sq(\forget^\fop\times_\calB 1)$ restricts to a double functor $\bbS(\forget)$
we must check that:
\begin{itemize}
	\item
	$\sq(\forget^\fop\times_\calB 1)$
	sends vertical morphisms of $\bbS^d(\calHbp^\fop)$
	to vertical morphisms of $\bbS^d(\calH^\fop)$;
	\item
	$\sq(\forget^\fop\times_\calB 1)$
	sends horizontal morphisms of $\bbS^d(\calHbp^\fop)$
	to horizontal morphisms of $\bbS^d(\calH^\fop)$;
	\item
	$\sq(\forget^\fop\times_\calB 1)$
	sends 2-cells of $\bbS^d(\calHbp^\fop)$
	to 2-cells of $\bbS^d(\calH^\fop)$.
\end{itemize}
It is trivial to check these conditions.

To see that the symmetric monoidal structure of $\sq(\forget^\fop\times_\calB 1)$
restricts to one for $\bbS(\forget)$, we must check that 
the unit and monoidality constraints for $\sq(\forget^\fop\times_\calB 1)$,
once restricted to $\bbS^d(\calHbp^\fop)$,
take values in $\bbS^d(\calH^\fop)$.
In more detail, we must check that:
\begin{itemize}
	\item
	the isomorphism
	$\sq(\forget^\fop\times_\calB 1)_{I,0}$
	and the components of the natural isomorphism
	$\sq(\forget^\fop\times_\calB 1)_{\otimes,0}$
	are vertical morphisms in $\bbS^d(\calH^\fop)$;
	\item
	the isomorphism 
	$\sq(\forget^\fop\times_\calB 1)_{I,1}$
	and the components of the natural isomorphim
	$\sq(\forget^\fop\times_\calB 1)_{\otimes,1}$
	at horizontal morphisms of $\bbS^d(\calHbp^\fop)$
	are 2-cells of $\bbS^d(\calH^\fop)$.
\end{itemize}
Again these conditions hold trivially.
\end{definition}

\begin{definition}[The symmetric monoidal double functor $\bbS(\Pi_1)$]
The symmetric monoidal double functor
\[\bbS(\Pi_1)\colon \bbS^d(\calHbp^\fop)\longrightarrow \bbS^d(\calG^\fop)\]
is obtained from the symmetric monoidal functor
$\Pi_1\colon \calHbp\to \calG$ exactly as $\bbS(\forget)$ was obtained
from the functor $\forget$ in Definition~\ref{def:bbS-forget}.
In order for this to be possible, we must verify that $\sq(\Pi_1^\fop\times_\calB 1)$ 
satisfies (the analogues of) the five conditions appearing 
in Definition~\ref{def:bbS-forget}.

To check that vertical morphisms are preserved, it suffices to check that if
$(X,P)\to(Y,Q)$ is a morphism of families of h-graphs with basepoints over the same base,
such that the underlying map $X\to Y$ is an h-embedding of families of h-graphs,
then the induced $\Pi_1(X,P)\to\Pi_1(Y,Q)$ has the basis extension property.
This is shown in Lemma~\ref{lem:h-embedding-basis-extension} below.

To check that horizontal morphisms are preserved, it suffices to check that
if $X$ and $Y$ are families of h-graphs over the same base $B$,
equipped with the same basepoints $P$,
and $m,n\colon B\to\Z$ are locally constant functions that satisfy the condition
\[
	m-n=-d\bigl(\chi(X)-\chi(Y)\bigr),
\]
then they also satisfy the condition
\[
	m-n=d\bigl(r(\Pi_1(X,P)- r(\Pi_1(Y,P))\bigr).
\]
Letting $p$ denote the cardinality of $P$, we have
$\chi(X)=p-r(\Pi_1(X,P))$ and $\chi(Y)=p-r(\Pi_1(Y,P))$,
so the second condition holds.

To check that 2-cells are preserved, it suffices to check that a 
commutative square of h-graphs with basepoints
	\[\xymatrix{
		(X,P)
		\ar[r]
		\ar[d]
		&
		(Y,P)
		\ar[d]
		\\
		(Z,Q)
		\ar[r]
		&
		(W,Q),
	}\]
in which the underlying square of h-graphs is a homotopy cofibre square
with its left and right edges h-embeddings,
is turned by $\Pi_1(-)$ into a pushout square of groupoids.
This is a simple consequence of an appropriate version of the 
van Kampen theorem \cite[Theorem 17']{Higgins}.

To check that the monoidality and unit constraints satisfy the required
conditions is trivial since isomorphisms of finite free groupoids
have the basis extension property, and since commutative squares of 
finite free groupoids whose vertical edges are isomorphisms are 
pushout squares.
\end{definition}
		
\begin{lemma}\label{lem:h-embedding-basis-extension}
Let $(X,P)\to (Y,Q)$ be a morphism of h-graphs with basepoints
whose underlying map of h-graphs is an h-embedding.
Then the induced morphism $\Pi_1(X,P)\to\Pi_1(Y,Q)$
has the basis extension property.
\end{lemma}
		
\begin{proof}
By discarding the components of $Y$ that do not meet the image of $X$
(which does not affect the property of being an h-embedding),
we may assume that the image of $X$ in $Y$ meets every component.
Then the morphism factors as $(X,P)\to (Y,P)\to (Y,Q)$, and 
$\Pi_1(Y,P)\to\Pi_1(Y,Q)$ clearly has the basis extension property
(to a basis of $\Pi_1(Y,P)$ attach a single morphism from each element of $Q\setminus P$
to some element of $P$).
So we may assume that $Q=P$.

Let us write $f\colon X\to Y$ for the map underlying the given morphism
$(X,P)\to (Y,P)$ and $u\colon P\to X$ for the basepoints of $X$.
Take a square
\[\xymatrix{
		A\ar[r]^{g}\ar[d]_{h} & X\ar[d]^f\\
		B\ar[r]_k & Y
}\]
that witnesses $f$ as an h-embedding,
so that $B$ is an h-graph and $A$ has the homotopy type of a finite set.
We may assume that $h$ is a closed cofibration and that
$g$ factors as $u\circ l$ for some map $l\colon A\to P$.
(To achieve this, we first choose a homotopy $F\colon A\times [0,1]\to X$
from $g$ to a map $g'$ that factors through $u$,
and then replace $B$ with the mapping cylinder of $h$,
the map $h$ with the standard inclusion,
the map $g$ with $g'$,
and the map $k$ by its extension by $f\circ F$.)
Now we may factorise the above square as
\[\xymatrix{
		A\ar[r]\ar[d]_{h} & P \ar[r]^u\ar[d] & X\ar[d]^f\\
		B\ar[r] & C\ar[r] & Y
}\]
where the left-hand square is a pushout along a closed cofibration 
and therefore a homotopy cofibre square.
Then the right-hand square is a homotopy cofibre square
by the two-out-of-three property dual to~\cite[13.3.15]{Hirschhorn},
and $C$ is an h-graph by Lemma~\ref{lm:pushout}.

Applying $\Pi_1(-,P)$ to the right-hand square above produces a pushout
square of finite free groupoids.
(This follows from an appropriate version
of the van Kampen theorem \cite[Theorem 17']{Higgins}.)
Now the morphism $\Pi_1(X,P)\to\Pi_1(Y,P)$ is a pushout of the morphism
$\Pi_1(P,P)\to\Pi_1(B,P)$.  Since the latter evidently has the basis extension
property, it follows that the former does too, and the lemma is proved.
\end{proof}

\begin{definition}[The symmetric monoidal double functor $\bbS(B(G^{(-)}))$]
\label{def:bbS-B-G-blank}
We define
\[
		B(G^{(-)})\times'_\calB 1
		\colon
		\calG^\fop\times_\calB\calB^\ds
		\longrightarrow
		\calM\times_\calB\calB^\ds
\]
to be the  symmetric monoidal functor
that sends a pair $(\sP,m)$ consisting of a fibrewise finite groupoid $\sP$
over a base $B$ and a locally constant function $m\colon B\to\Z$ to the pair
$(B(G^{\sP}),m\circ\pi_\sP)$,
where $\pi_\sP\colon B(G^{\Ob(\sP)})\to B$ is the projection map.
(The construction of a simpler functor $B(G^{(-)})\times_\calB 1$ is precluded
by the fact that $B(G^{(-)})$ is not a functor over $\calB$.)
This induces a symmetric monoidal double functor 
\[
	\sq(B(G^{(-)})\times'_\calB 1)\colon
	\sq(\calG^\fop\times_\calB\calB^\ds)
	\longrightarrow
	\sq(\calM\times_\calB\calB^\ds),
\]
and by restricting the domain and range, we obtain
the symmetric monoidal double functor
\[\bbS(B(G^{(-)}))\colon \bbS^d(\calG^\fop)\longrightarrow\bbS^\ds(\calM).\]
For this last step to be possible we must verify that $\sq(B(G^{(-)})\times'_\calB 1)$
satisfies (the analogues of) the five conditions appearing 
in Definition~\ref{def:bbS-forget}.
Unravelling the definitions,
and in particular using the proof of Proposition~\ref{prop:functors-manifolds}
for the definition of the fibrewise smooth structure on the spaces $B(G^\sP)$,
we see that these checks amount to the following.

For vertical morphisms to be preserved, one must check that
if a morphism $f\colon \sP\to\sQ$ between finite free groupoids
has the basis extension property, then the map of smooth manifolds
$\fun(\sQ,G)\to \fun(\sP,G)$ is a smooth submersion.
But we may choose a basis $\Gamma$ for $\sP$ and extend its image $f\Gamma$
to a basis $f\Gamma\sqcup\Delta$ for $\sQ$, and then the induced map
$\fun(\sQ,G)\to \fun(\sP,G)$ may be identified with the map
$G^{f\Gamma\sqcup\Delta}\to G^\Gamma$, which is a smooth submersion
(it is a projection map from one product of copies of $G$ to a product of fewer copies
of $G$).

For horizontal morphisms to be preserved, we must check the following.
Let $\sP$ and $\sQ$ be fibrewise finite free groupoids over the same base $B$,
with the same spaces of objects, and let $m,n\colon B\to\Z$ be locally constant
functions.
We must show that if the fibrewise ranks of $\sP$ and $\sQ$ satisfy the equation
\[
	m-n=d(r(\sP)-r(\sQ))
\]
then the fibrewise dimensions of $B(G^\sP)$ and $B(G^\sQ)$ satisfy the equation
\[
	m+|B(G^{\sP})| = n+|B(G^{\sQ})|.
\]
But according to Proposition~\ref{prop:functors-manifolds},
$|B(G^{\sP})|=\dim(G)\cdot r(\sP)$ and $|B(G^{\sQ})|=\dim(G)\cdot r(\sQ)$,
while we have $d=-\dim(G)$, and so the required condition holds.

To check that 2-cells are preserved, we must check that
if the square on the left is a pushout of finite free groupoids
\[\xymatrix{
	\sP\ar[r]\ar[d]
	&
	\sQ\ar[d]
	\\
	\sR\ar[r]
	&
	\sS
}
\qquad\qquad
\xymatrix{
	\fun(\sS,G)\ar[r]\ar[d]
	&
	\fun(\sR,G)\ar[d]
	\\
	\fun(\sQ,G)\ar[r]
	&
	\fun(\sP,G)
}
\]
then the square on the right is a pullback square of spaces.
This is immediate.

That the unit and monoidality constraints satisfy the required
conditions is again immediate, since an isomorphism of fibrewise
manifolds is a submersion on fibres, and since a commutative
square of manifolds whose vertical edges are isomorphisms
is a pullback square.
\end{definition}

\subsection{Constructing $\widetilde U^G$ and $U^G$}
\label{subsec:UG}
Now we have constructed the left and upper edges of the diagram~\eqref{diagram-two}.
In this subsection we will construct the remaining part of that diagram,
thereby completing the programme outlined at the start of 
this section, and hence the construction of our HHGFT.

We begin by constructing the intermediate functor $\widetilde U^G$ and the natural
transformation $U_\mfld\circ\bbS(B(G^{(-)}))\circ\bbS(\Pi_1)\Rightarrow \widetilde U^G$. 
We require that the vertical part of $\widetilde U^G$ will be exactly the functor
$(X,P,m)\mapsto H_{\ast-m}(BG^X)$.
We start by examining the vertical part of 
$U_\mfld\circ\bbS(B(G^{(-)}))\circ\bbS(\Pi_1)$.

\begin{lemma}
\label{lm:vertpartiso}
The vertical part of the composite 
$U_\mfld\circ\bbS(B(G^{(-)}))\circ\bbS(\Pi_1)$ is symmetric
monoidally naturally isomorphic to the functor that sends a triple
$(X,P,m)$, consisting of a family of h-graphs with basepoints $(X,P)$ 
over a good base space
$B$ and a locally constant function $m\colon B\to\Z$,
to the homology $H_{\ast-m}(BG^X)$.
\end{lemma}

\begin{proof}
Let us write $\calB^\ds_0$ for the symmetric monoidal subcategory of
$\calB^\ds$ with the same objects, but containing only those morphisms
$f\colon (B,m)\to (C,n)$ for which $n\circ f = m$.
Then diagram~\eqref{diagram-one} extends to the following diagram.
\[
\xymatrix@+3em{
	\calHbp^\fop\times_\calB\calB^\ds_0
	\ar[d]_{\forget\times_\calB 1}
	\ar[r]^{\Pi_1^\fop\times_\calB 1}
	&
	\calG^\fop\times_\calB\calB^\ds_0
	\ar[r]^{B(G^{(-)})\times'_\calB 1}
	\ar[d]^{B^\fop\times_\calB 1}
	\ar@{<=>}[dl]_\simeq
	&
	\calM\times_\calB\calB^\ds_0
	\ar[d]^{\forget\times_\calB 1}
	\ar@{=>}[dl]_\simeq
	\ar[r]^{(U_\mfld)_0}
	&
	\grmod
	\ar@{=}[d]
	\\
	\calH^\fop\times_\calB\calB^\ds_0
	\ar[r]_{\inclusion\times_\calB 1}
	&
	\calS^\fop\times_\calB\calB^\ds_0
	\ar[r]_{(\inclusion\circ BG^{(-)})\times_\calB 1}
	&
	\widehat\calS\times_\calB\calB^\ds_0
	\ar[r]_H
	&
	\grmod
}\]
The left hand square of this diagram is obtained from the left hand square of
\eqref{diagram-one} by taking fibre product over $\calB$ with $\calB^\ds_0$.
This is possible since the categories, functors and transformations in 
the left hand square in \eqref{diagram-one} are over $\calU$, and hence over $\calB$.
The bottom and right hand edges of the middle square are obtained in a similar way.  
The top functor in the middle square cannot be obtained in this way since the functor
$B(G^{(-)})\colon\calG^\fop\to\calM$ in \eqref{diagram-one} is not over $\calB$.
Instead it is defined in the same way as the analogous functor
in Definition~\ref{def:bbS-B-G-blank}.
The natural transformation in the middle square is obtained directly from
the natural transformation in the right hand square of \eqref{diagram-one}.
The functor $H$ is the functor
sending a pair $(S,m)$ consisting of a space $S$ over
a base space $B$ and a locally constant function 
$m\colon B\to\Z$ to $H_{\ast-m}(S)$.

Observe that the vertical parts of the double categories in \eqref{diagram-two} are all
subcategories of the categories on the left and top edges of this new diagram,
and that the vertical part of $U_\mfld\circ\bbS(B(G^{(-)}))\circ\bbS(\Pi_1)$
is the restriction
to these subcategories of the composite along the top edge.
Prolonging the natural transformations in the left-hand square
(which consist of fibrewise homotopy equivalences)
by $BG^{(-)}\times_\calB 1$ produces a zig-zag of natural transformations
consisting of homotopy equivalences on total spaces.
Composing this new zig-zag with the natural transformation in the middle square
(which consists of homotopy equivalences on total spaces)
produces another zig-zag of natural transformations
consisting of homotopy equivalences on total spaces.
Prolonging this zig-zag by the functor $H$
(which sends homotopy equivalences of total spaces to isomorphisms)
produces a zig-zag of natural \emph{isomorphisms}
between the vertical part of $U_\mfld\circ\bbS(B(G^{(-)}))\circ\bbS(\Pi_1)$
and the restriction of the composite
$H\circ(BG^{(-)}\times_\calB 1)\circ(\inclusion\times_\calB 1)\circ (\forget\times_\calB 1)$.
to the vertical part of $\bbS^d(\calHbp^\fop)$.
But the latter functor is simply the assignment
$(X,P,m)\mapsto H_{\ast-m}(BG^X)$.
\end{proof}

The (entirely formal) proof of the next lemma is left to the reader.

\begin{lemma}
\label{lem:moddfun}
Let $\calC$ be a symmetric monoidal category, let $\bbD$
be a symmetric monoidal double category, and let 
$F \colon \bbD \to \sq(\calC)$ be a  symmetric monoidal 
double functor. Suppose $G_0 \colon \bbD_0 \to \calC$ 
is a  symmetric monoidal functor and 
$\eta_0\colon F_0 \to G_0$
is a symmetric monoidal natural isomorphism. 
Then $G_0$ and $\eta_0$ extend in a unique way 
to a  symmetric monoidal double functor 
$G \colon \bbD \to \sq(\calC)$ and a symmetric monoidal
isomorphism $\eta \colon F \to G$. \qed
\end{lemma}

It follows from Lemmas~\ref{lm:vertpartiso} and~\ref{lem:moddfun}
that the composite $U_\mfld\circ\bbS(B(G^{(-)}))\circ\bbS(\Pi_1)$
is symmetric monoidally naturally isomorphic to a functor
$\widetilde U^G\colon \bbS^d(\calHbp^\fop)\to\sq(\grmod)^\hop$
whose vertical part is exactly $(X,P,m)\mapsto H_{\ast-m}(BG^X)$.
In particular, the vertical part of $\widetilde U^G$ factors 
as a composite
\[
	\bbS^d(\calHbp^\fop)_0
	\xrightarrow{\ \bbS(\forget)_0\ }
	\bbS^d(\calH^\fop)_0
	\longrightarrow
	\sq(\grmod)^\hop_0
\]
where the second functor is precisely
$(X,m)\mapsto H_{\ast-m}(BG^X)$.
It remains to show that $\widetilde U^G$ itself factors in 
an analogous way as indicated in diagram~\eqref{diagram-two}.

\begin{lemma}
Let
$\widetilde V\colon \bbS^d(\calHbp^\fop)\to\sq(\grmod)^\hop$
be a  symmetric monoidal double functor whose vertical part factors as
\[
	\bbS^d(\calHbp^\fop)_0
	\xrightarrow{\ \bbS(\forget)_0\ }
	\bbS^d(\calH^\fop)_0
	\xrightarrow{\ V_0\ }
	\sq(\grmod)^\hop_0
\]
for some  symmetric monoidal functor $V_0$.
Then $\widetilde V$ itself factors as
\[
	\bbS^d(\calHbp^\fop)
	\xrightarrow{\ \bbS(\forget)\ }
	\bbS^d(\calH^\fop)
	\xrightarrow{\ V\ }
	\sq(\grmod)^\hop_0
\]
for some uniquely determined  symmetric monoidal double functor $V$
whose vertical part is $V_0$.
\end{lemma}

\begin{proof}
Let us construct the required $V$.
We begin by defining $V$ on horizontal morphisms.
Let $f\colon (X,m)\to (Y,n)$ be a horizontal morphism of $\bbS^d(\calH^\fop)$
between objects over $B$.
We may choose basepoints $P\times B\to Y$
(for all objects of $\calH$ admit basepoints by assumption)
and since $f$ is positive there is a unique choice of basepoints
$P\times B\to X$ such that $f$ lifts to a horizontal morphism 
$f^P\colon (X,P,m)\to (Y,P,n)$ in $\bbS^d(\calHbp^\fop)$.
Then $\widetilde V_1(f^P)$ is a horizontal morphism
$V_0 (X,m)\to V_0 (Y,n)$ in $\sq(\grmod)^\hop$, 
and we define $V_1(f)=\widetilde V_1(f^P)$.
To see that $V_1(f)$ is independent of the choice of basepoints for $Y$,
let $Q\times B\to Y$ be a second choice of basepoints.
Then we obtain a third choice $(P\sqcup Q)\times B\to Y$ and a 2-cell
\[\xymatrix{
	(X,P\sqcup Q,m)
	\ar[r]^{f^{P\sqcup Q}}
	\ar[d]
	&
	(Y,P\sqcup Q,n)
	\ar[d]
	\\
	(X,P,m)
	\ar[r]_{f^P}
	&
	(Y,P,n)
}\]
in $\bbS^d(\calHbp^\fop)$
whose vertical edges lie over identity morphisms in $\bbS^d(\calH^\fop)$.
It follows that $\widetilde V_1(f^P)=\widetilde V_1(f^{P\sqcup Q})$,
and similarly $\widetilde V_1(f^Q)=\widetilde V_1(f^{P\sqcup Q})$.
Thus $V_1(f)$ is well-defined.

The last paragraph defined the constituent functor $V_1$ 
of the symmetric monoidal double functor $V$ on objects. 
The data left to specify to complete the definition of $V$
are the values of $V_1$ on morphisms and the unit and 
monoidality constraints $V_{I,1}$ and $V_{\otimes,1}$. 
But these data consist of 2-cells in $\sq(\grmod)^\hop$ which,
since they amount to commutative squares in $\grmod$, are determined
by their boundary 1-morphisms; and the boundary 1-morphisms
of the requisite 2-cells are determined by the parts of $V$ 
that have already been specified, namely the symmetric monoidal functor 
$V_0$ and the values of $V$ on horizontal morphisms.
Thus to complete the construction of $V$, we must check that the
resulting squares in $\grmod$ do indeed commute,
and that the axioms for a symmetric monoidal double functor 
are satisfied. All these verifications are easily performed. 
Moreover, it is clear from the definition that the composite
$V\circ \bbS(\forget)$ is $\widetilde V$.

It remains to demonstrate that $V$ is uniquely determined.
By the discussion above, any choice for $V$ is determined 
by the given functor $V_0$ and the values $V$ takes on
horizontal 1-morphisms.
But the equation $V_1(f)=\widetilde V_1(f^P)$ 
determining the value of $V$ on a horizontal
1-morphism $f$ must clearly hold if $\widetilde V$ is to 
factor as $V\circ\bbS(\forget)$.
\end{proof}

Applying the last lemma to $\tilde U^G$
provides us with a symmetric monoidal double functor
$U^G\colon\bbS^d(\calH^\fop)\to\sq(\grmod)^\hop$ whose vertical part is
$(X,n)\mapsto H_{\ast-n}(BG^X)$.
This completes our task for this subsection, and indeed completes the construction
of our homological h-graph field theory.

%% file: comparison.tex

\section{Comparison with Chataur and Menichi's theory}
\label{sec:ChataurMenichi}

Let $G$ be a compact Lie group for which Chataur and Menichi's 
HCFT $\phi^\mathrm{CM}$ is defined, that is, 
either a finite group or a connected compact Lie group.
The aim of this section is to show that the HHGFT $\Phi^G$
constructed in the proof of Theorem~\ref{thm:main} is an 
extension of Chataur and Menichi's HCFT 
$\phi^\mathrm{CM}$ \cite{ChataurMenichi} as claimed
after the statement of Theorem~\ref{thm:main}.
The two theories do agree on 1-manifolds,
both sending a closed 1-manifold $X$ to the homology $H_\ast(BG^X)$.
Thus we are left to show that the two operations
\[
	\Phi^G(\Sigma), \phi^\mathrm{CM}(\Sigma) 
		\colon
	H_{\ast+\dim(G)\cdot\chi(\Sigma,X)}(B\Diff(\Sigma))
		\tensor
	H_\ast(BG^X)
		\longto
	H_\ast(BG^Y)
\]
agree when $\Sigma$ is a closed cobordism from $X$ to $Y$
admissible in the theory $\phi^\mathrm{CM}$. 
(We remind the reader that 
we are working over a fixed field $\bbF$ of characteristic 2,
so the determinant twisting of \cite[appendix C]{ChataurMenichi}
reduces to a degree shift.)
Recall that in the case of a finite group $G$, $\Sigma$ is admissible 
if the inclusion $X\incl \Sigma$ is surjective on $\pi_0$,
while in the case of a compact connected Lie group $G$,
both $X \incl \Sigma$ and $Y \incl \Sigma$
must be surjective on $\pi_0$. 

The operations $\Phi^G(\Sigma)$
and $\phi^\mathrm{CM}(\Sigma)$ both arise
from a push-pull construction that considers the diagram
\[
	B\Diff(\Sigma) \times BG^X 
		\xleftarrow{\ \ p_0\ \ }
	BG^{U\Diff(\Sigma)}
		\xrightarrow{\ \ p_1\ \ }
	B\Diff(\Sigma) \times BG^Y,
\]
where $U\Diff(\Sigma)=E\Diff(\Sigma) \times_{\Diff(\Sigma)} \Sigma$ is the 
universal family of cobordisms from $X$ to $Y$ over $B\Diff(\Sigma)$.
They are defined as composites
\[\xymatrix@!0@C=4em{
	&
	{H_{\ast+\dim(G)\cdot\chi(\Sigma,X)}(B\Diff(\Sigma))\otimes H_\ast(BG^X)}
	\\
	\ar[r]^-{\times}
	&
	*+[r]{H_{\ast+\dim(G)\cdot\chi(\Sigma,X)}(B\Diff(\Sigma) \times BG^X)}
	\\
	\ar[r]^-{p_0^\natural}
	&
	*+[r]{H_{\ast}(BG^{U\Diff(\Sigma)})}
	\\
	\ar[r]^-{(p_1)\ast}
	&
	*+[r]{H_\ast(B\Diff(\Sigma) \times BG^Y)}
	\\
	\ar[r]^-{\pr_\ast}
	&
	*+[r]{H_\ast(BG^Y)}
}\]
where $p_0^\natural$ is an umkehr map induced by $p_0$.
See subsection~\ref{subsec:push-pull-construction}
and the proof of \cite[Theorem 9]{ChataurMenichi}.
The point of divergence in the construction of the operations
$\Phi^G(\Sigma)$ and $\phi^\mathrm{CM}(\Sigma)$ is in the 
construction of this umkehr map $p_0^\natural$. Thus to prove that 
the two operations agree, it is enough to verify the
following proposition.
\begin{proposition}
\label{prop:umkehrs-agree}
The umkehr maps $p_0^\natural$ employed in the constructions of $\Phi^G(\Sigma)$ 
and $\phi^\mathrm{CM}(\Sigma)$ agree.
\end{proposition}

In the case of a finite group $G$, Chataur and Menichi use 
transfer maps as the umkehr maps in the 
construction of $\phi^\mathrm{CM}$, while
in the case of a connected group $G$, they obtain
the umkehr maps from the Serre spectral sequence in the way we now recall.

\begin{construction}
\label{cons:serre-ss-umkehr}
Let $F \to X \xto{p} B$ be a Serre fibration in which the base space $B$ is connected,
the fibre $F$ is homotopy equivalent to a closed $d$-dimensional manifold, 
and the action of $\pi_1(B)$ on $H_d(F)$ is trivial. 
Then there is an induced umkehr map
\[
	p^\sharp \colon H_\ast(B) \longto H_{\ast+d}(X)
\]
defined as the composite
\[
	H_\ast(B) 
		\longto 
	H_\ast(B;\,H_d(F)) = E^2_{\ast,d} 
		\longto 
	E^\infty_{\ast,d}
		\longto
	H_{\ast+d}(X)
\]
where the first map is induced by the fundamental class of $F$;
$E^2$ and $E^\infty$ refer to the respective pages in the 
Serre spectral sequence of the fibration $p$; the penultimate map
is the projection onto a quotient; and the last map is the monomorphism
given by the identification of $E^\infty_{\ast,d}$ as the first 
stage in a filtration of $H_{\ast+d}(X)$.
\end{construction}

On the other hand, 
the umkehr maps used in the construction of $\Phi^G$ are obtained, 
roughly speaking, by first replacing the domain and the target 
of the map by homotopy equivalent
fibrewise closed manifolds and the map itself by a fibrewise
smooth map, and by then taking the umkehr map constructed
in subsection~\ref{subsec:umkehr-maps-for-fibrewise-manifolds}
using a fibrewise Pontryagin--Thom construction. 
In the case of the map $p_0$,
the replacement for $p_0$ can be obtained by first choosing 
a set of basepoints $P \to X$ for $X$, and is then given by 
the map of fibrewise manifolds over $BG^P \times B\Diff(\Sigma)$
\begin{equation}
\label{diag:p0replacement}
\vcenter{\xymatrix@C-2em{
	B(G^{\Pi_1(U\Diff(\Sigma),P)})
	\ar[dr]
	\ar[rr]^-{\tilde p_0}
	&&
	B(G^{\Pi_1(X,P)}) \times B\Diff(\Sigma)
	\ar[dl]
	\\
	&
	BG^P\times B\Diff(\Sigma)
}}
\end{equation}
induced by the inclusion $X\times B\Diff(\Sigma) \incl U\Diff(\Sigma)$.
Alternatively, we may identify the above diagram with the one obtained
from the diagram
\[\xymatrix{
	\fun(\Pi_1(\Sigma,P),G)
	\ar[rr]
	\ar[dr]
	&&
	\fun(\Pi_1(X,P),G)
	\ar[dl]
	\\
	&
	\pt
}\]
of $G^P\times \Diff(\Sigma)$-spaces by the Borel construction.
(See the discussion preceding Lemma~\ref{lem:single-smooth-structure}
and the proof of Proposition~\ref{prop:functors-manifolds}.)

Let us now focus on the case where $G$ is connected.
Since in this case 
every component of $\Sigma$ is required to have non-empty outgoing boundary, 
it follows as in Example~\ref{ex:occobordism}
that the inclusion $X \incl \Sigma$ is an h-embedding, and hence
the map
\[
	\fun(\Pi_1(\Sigma,P),G)
		\longto
	\fun(\Pi_1(X,P),G)
\]
can be identified with a projection from a direct product 
of a number of copies of $G$ onto some of its factors. 
See Lemma~\ref{lem:h-embedding-basis-extension} and 
the discussion regarding vertical morphisms in 
Definition~\ref{def:bbS-B-G-blank}. Thus the map $\tilde p_0$
is a fibrewise bundle in the sense of the following definition.

\begin{definition}
Let $p\colon M \to N$ be a map of fibrewise closed manifolds
over a base space $B$. We call $p$ a \emph{fibrewise bundle}
if each point $b\in B$ has a neighbourhood $U$ over which $M$ and $N$
admit local trivializations $M|U \isom M'\times U$ and 
$N|U \isom N'\times U$ under which the map $p$ corresponds to a map
\[
	p' \times 1 \colon M'\times U \longto N' \times U
\]
where the map $p' \colon M' \to N'$ is a smooth fibre bundle. 
\end{definition}

Since the umkehr map obtained from the Serre spectral
sequence is compatible with homotopy equivalences, to prove 
Proposition~\ref{prop:umkehrs-agree} in the case of a 
connected compact Lie group $G$, it suffices to verify 
the following lemma.
\begin{lemma}
\label{lm:umkehrs-for-fw-bundles}
Let $M$ and $N$ be fibrewise manifolds over a base space $B$, 
and let $p \colon M \to N$ be a fibrewise bundle. Then the two maps
\[
	p^\sharp, p^! \colon H_\ast(N) \longto H_{\ast+|M|-|N|} (M)
\]
agree, where $p^\sharp$ denotes the umkehr map of
Construction~\ref{cons:serre-ss-umkehr} and $p^!$ the umkehr map
constructed in subsection~\ref{subsec:umkehr-maps-for-fibrewise-manifolds}.
\end{lemma}

Since homology is compactly supported and both kinds of umkehr maps
under consideration are compatible with pullbacks, to prove 
Lemma~\ref{lm:umkehrs-for-fw-bundles}, it is enough to 
consider the special case where the base space $B$ is a finite CW complex
(and hence in particular a compact ENR). 
Then $N$ is also a compact ENR, as follows for example from 
\cite[Proposition IV.8.10]{DoldLAT} and the assumption that 
$N$ is a fibrewise smooth fibre bundle over $B$ with fibre
a closed manifold. Let $q$ denote the map $N \to B$, and
consider the  commutative diagram
\[\xymatrix@!0@C=2.2em@R=6ex{
	&
	M 
	\ar `l[ldd]`[dddd]_\id [dddd]
	\ar[dd]
	\ar[drr]^p
	\ar[rrrr]^p
	&&&&
	N
	\ar `r[rdd]`[dddd]^\id [dddd]
	\ar[dd]
	\ar[dll]_\id
	\\
	&&&
	N
	\ar[dd]_(0.35)q|!{[dr];[dl]}{\hole}
	\\
	&
	q^\ast M
	\ar[urr]
	\ar[dd]
	\ar[rrrr]
	&&&&
	q^\ast N
	\ar[ull]
	\ar[dd]
	&
	\\
	&&&
	B
	\\
	&
	M
	\ar[urr]^{qp}
	\ar[rrrr]^p
	&&&&
	\ar[ull]_q
	N
}\]
Observe that the assumption that $p$ is a fibrewise bundle in particular
makes $M$ into a fibrewise closed manifold over $N$, and that 
the top square of the above diagram is a fibrewise transverse
pullback square. Now \cite[Proposition II.12.16]{CrabbJames} 
together with the fibrewise version of 
\cite[Proposition II.12.11]{CrabbJames} applied to the top square 
of the above diagram 
lead to the conclusion that the umkehr maps $p^!$ for $p$
considered as a map of fibrewise manifolds over $B$ on one hand
and as a map of fibrewise manifolds over $N$ on the other hand 
coincide. Thus to prove Lemma~\ref{lm:umkehrs-for-fw-bundles}, 
it suffices to verify the following lemma.
\begin{lemma}
\label{lm:umkehrs-for-fw-mfld}
Let $p\colon M \to B$ be a fibrewise closed manifold over 
a base space $B$, and consider $B$ as a fibrewise closed manifold
over itself via the identity map.
Then the two umkehr maps
\[
	p^\sharp, p^! \colon H_\ast(B) \longto H_{\ast+|M|} (M)
\]
agree.
\end{lemma}

Again, to prove Lemma~\ref{lm:umkehrs-for-fw-mfld}, it is enough
to consider the case where $B$ is a finite CW complex. 
We can now employ the strategy outlined in the proof of 
\cite[Lemma~6.22]{BoardmanNotesCh5}
to verify that the two umkehr maps agree.
Both types of umkehr maps admit relative versions
in the sense that there are induced maps
\[
	p^\sharp, p^! \colon H_\ast(B,B_0) \longto H_\ast(M,M_0)
\]
associated to a map $p \colon (M,M_0) \to (B,B_0)$
where $p \colon M \to B$ is a fibrewise closed manifold over 
a finite CW complex $B$, $B_0 \subset B$ is a subcomplex,
and $M_0$ is the restriction of $M$ to $B_0$.
Moreover, these umkehr maps are natural in the sense that if 
$q \colon (N,N_0) \to (C,C_0)$ is another such map
and
\[\xymatrix{
	(N,N_0)
	\ar[r]^g 
	\ar[d]_q
	&
	(M,M_0)
	\ar[d]^p
	\\
	(C,C_0) 
	\ar[r]^f 
	&
	(B,B_0)
}\]
is a pullback square, then the square
\[\xymatrix{
	H_{\ast+|N|}(N,N_0) 
	\ar[r]^{g_\ast}
	&
	H_{\ast+|M|}(M,M_0)
	\\
	H_\ast(C,C_0)
	\ar[r]^{f_\ast}
	\ar[u]^{q^\sharp}
	&
	H_\ast(B,B_0)	
	\ar[u]_{p^\sharp}
}\]
commutes, and likewise for $p^!$ and $q^!$.
To construct the relative $p^\sharp$, one can simply 
use the relative Serre spectral sequence of the map 
$p \colon (M,M_0) \to (B,B_0)$; and to construct the relative $p^!$,
one uses a relativized form
\[
	p^\umk \colon \suspension^\infty B/B_0 \longto M^{-\tau_M}/ M_0^{-\tau_M}
\]
of the stable map \eqref{eq:gysinmap} (with $N=B$ and $\beta=0$)
together with a relative form of the Thom isomorphism.

Let $B^{(n)}$ denote the $n$-skeleton of $B$, and let 
$M^{(n)} = p^{-1}( B^{(n)} ) \subset M$. 
Since the map $H_n(B^{(n)}) \to H_n(B)$ is an epimorphism,
to show that the maps
\[
	p^\sharp, p^! \colon H_\ast(B) \longto H_{\ast+|M|} (M)
\]
agree, it is enough to show that the maps
\[
	p^\sharp, p^! \colon H_\ast(B^{(n)}) \longto H_{\ast+|M|} (M^{(n)})
\]
do. Since the map 
$H_{n+|M|} (M^{(n)}) \to H_{n+|M|}(M^{(n)},M^{(n-1)})$
is a monomorphism, it is enough to show that 
the maps
\[
	p^\sharp, p^! 
	\colon
	H_n(B^{(n)},B^{(n-1)})
	\longto
	H_{n+|M|}(M^{(n)},M^{(n-1)})
\]
coincide. Since the group
$H_n(B^{(n)},B^{(n-1)})$ is generated by the images of 
the maps
\[
	f_* \colon H_n(D^n,S^{n-1}) \longto H_n(B^{(n)},B^{(n-1)})
\]
as $f$ runs through the characteristic maps for the $n$-cells
of $B$, we see that it is enough to show that 
$q^\sharp = q^!$ when $q$ is a projection 
\[
	q \colon F\times (D^n,S^{n-1}) \longto (D^n,S^{n-1})
\]
with $F$ a closed manifold. This verification is straightforward. 
This concludes the proof of Lemma~\ref{lm:umkehrs-for-fw-mfld}
and therewith the proof that the operations $\Phi^G(\Sigma)$ 
and $\phi^\mathrm{CM}(\Sigma)$ agree when $G$ is a connected
compact Lie group.

Let us now assume that $G$ is a finite group. In this case, 
the domain and target of the map $\tilde{p}_0$ of 
diagram~\eqref{diag:p0replacement} are
(possibly disconnected) finite covering spaces of 
$BG^P \times B\Diff(\Sigma)$. Since the
transfer map is compatible with homotopy equivalences, 
to prove Proposition~\ref{prop:umkehrs-agree} in the case of a 
finite group $G$, it therefore suffices to verify the following lemma.
\begin{lemma}
\label{lm:umkehrs-for-maps-bw-coverings}
Let $M$ and $N$ be (possibly disconnected) finite covering spaces
of a path-connected base space $B$, and let 
$p \colon M \to N$ be a map over $B$.
Then the umkehr map 
\[
	p^! \colon H_\ast(N) \longto H_{\ast} (M)
\]
constructed in subsection~\ref{subsec:umkehr-maps-for-fibrewise-manifolds}
agrees with the transfer map induced by $p$. 
\end{lemma}

Observe that if $N_\alpha$ is a component of $N$, 
then  $p$ restricts to give a finite 
(possibly disconnected and possibly empty) covering space
\[
	p_\alpha \colon  p^{-1}(N_\alpha) \longto N_\alpha
\]
of $N_\alpha$. Both $p^!$ and the transfer map associated
with $p$ decompose as direct sums of the corresponding maps
associated with the various $p_\alpha$, and 
if $p^{-1}(N_\alpha)$ is empty, 
then $p_\alpha^!$ and the transfer associated to $p_\alpha$ 
are both zero.
Thus it is enough to prove Lemma~\ref{lm:umkehrs-for-maps-bw-coverings}
when $N$ is connected and $M$ is a non-empty covering space 
of $N$. A similar argument as before now shows that we may 
further reduce to the case where $N = B$ and $B$ is a 
finite CW complex.
In this case, the claim can be proven by observing 
that the stable map
\[
	p^\umk \colon B_+ \longto M_+
\]
of equation~\eqref{eq:gysinmap} (with $N=B$ and $\beta=0$)
underlying $p^!$ agrees with Becker and Gottlieb's \cite{BeckerGottlieb} 
stable map underlying the transfer map associated to $p$.
This concludes the proof that the operations $\Phi^G(\Sigma)$ 
and $\phi^\mathrm{CM}(\Sigma)$ agree also when $G$ is a finite
group.

%% file: higher-operations.tex

\section{An example of non-trivial higher operations}
\label{sec:higher-operations}

The aim of this section is to prove Theorem~\ref{thm:calc}
computing the operation
\begin{equation}
\label{eq:op-to-compute}
	\Phi^{\Z/2}(S/B(\Z/2))
	\colon 
	H_\ast B(\Z/2) \tensor H_\ast B(\Z/2) 
	\longto 
	H_\ast B(\Z/2) 
\end{equation}
associated to a certain family of h-graph cobordisms
$S/B(\Z/2) \colon \pt \hto \pt$.
The family $S/B(\Z/2)$ is defined in
Definition~\ref{def:family-s} below;
its fibres are modelled on the h-graph cobordism 
$S_0\colon \{p\} \hto \{q\}$
depicted below.
\begin{equation}
\label{eq:modelcob}
\begin{tikzpicture}[
		baseline=(current  bounding  box.center),
		scale=0.03,
		->-/.style={	decoration={	markings,
		  	mark=at position 0.5 with {\arrow{>}}
			},postaction={decorate}}
	]
	\clip (-15, -20) rectangle (110, 60);
	\path[ARC, fill=black] (0,0) circle (3);
	\node [below] at (0,0) {$p$};
	\path[ARC, fill=black] (100,0) circle (3);
	\node [below] at (100,0) {$q$};
	\path[ARC,->-] (0,0) -- (100,0) node [midway, above] {$c$};
	\path[ARC,->-] (0,0) .. controls (70,30) and (-30,70) 
						.. (0,0) node [midway,above] {$l$};
\end{tikzpicture}
\end{equation}
Moreover, we explain how the calculation of the operation 
\eqref{eq:op-to-compute} leads to 
the construction of non-trivial elements in the homology 
of the holomorph $\Hol(F_n) = F_n \rtimes \Aut(F_n)$ of the 
free group on $n$ generators. See Corollary~\ref{cor:hol-application}.
While it is possible to construct these elements of $H_\ast(B\Hol(F_n))$ 
by other means,
the proof of Corollary~\ref{cor:hol-application} illustrates what
we hope will be a fruitful pattern for 
constructing non-trivial elements in the difficult-to-understand
unstable homology of $\Hol(F_n)$, $\Aut(F_n)$, mapping 
class groups of surfaces, and other interesting groups 
whose homology groups parameterize operations in HHGFTs.

Before starting our work, we make the following remark
which expands on the comment we made after the statement of 
Theorem~\ref{thm:calc}.

\begin{remark}
\label{rk:calc}
The Pontryagin product on $H_\ast B(\Z/2)$ featuring 
in Theorem~\ref{thm:calc} can be computed
by dualizing the Hopf algebra structure on the cohomology
$H^\ast B(\Z/2) \isom \F[u]$ where the coproduct is induced
by the addition map of $\Z/2$. Explicitly, the homology $H_\ast B(\Z/2)$
is isomorphic as an algebra to the divided polynomial algebra
$\Gamma_\F[x]$ on a generator of degree 1, or what is the same,
the exterior algebra $\Lambda_\F (y_0,y_1,y_2,\ldots)$
where $y_i$ has degree $2^i$. 
See for example \cite[Example 3C.11 and p.~286]{Hatcher}.

As a consequence of the above calculation, we see that for every
$a\in H_i B(\Z/2)$, $i>0$, Theorem~\ref{thm:calc}
gives us a non-zero string topology operation 
$H_\ast B(\Z/2) \to H_{\ast+i} B(\Z/2)$,
as claimed following the statement of Theorem~\ref{thm:calc}. 
At least when $i$ is larger than $1$, 
this operation does not correspond to any HCFT operation.
To see this, observe that the only open-closed cobordisms
equivalent to $S_0$ are the cobordisms $U$ and $V$ of
Example~\ref{ex:whistlepunchcard}. But by Example~\ref{ex:diff-vs-haut}
neither $B\Diff(U)$ nor $B\Diff(V)$ has homology in dimensions
larger than $1$.
\end{remark}

Our goal now is to construct the family $S/B(\Z/2)$ and prove
Theorem~\ref{thm:calc}. 
We will obtain the family $S/B(\Z/2)$ from a diagram of 
fibrewise finite free groupoids over $B(\Z/2)$. We start 
with a series of lemmas.

\begin{lemma}
\label{lm:cobs-from-groupoids}
Let $\sX$ and $\sY$ be finite free groupoids, and let 
\[
	\sX\times B \xto{\ i\ } \sP \xot{\ j\ } \sY\times B
\]
be a diagram of fibrewise finite free groupoids over a 
good base space $B$ such that on every fibre, the map
$i$ is essentially surjective on objects and the map
$j$ has the basis extension property. 
Let $(B\sP)'$ denote the mapping cone of the map 
$(B\sX \sqcup B\sY)\times B \to B\sP$.
Then the diagram
\[
	B\sX \times B \longto (B\sP)' \longot B\sY \times B
\]
of spaces over $B$ is a family of h-graph cobordisms
$(B\sP)'/ B \colon B\sX \longhto B\sY$.
\end{lemma}
\begin{proof}
For any finite free groupoid $\sQ$, the classifying space $B\sQ$
is an h-graph. The map $B\sP \to B$ is a fibre bundle over 
a paracompact space and hence a fibration; 
by \cite[Proposition~1.3]{Clapp}, the map $(B\sP)' \to B$ 
is then also a fibration. The assumptions on $i$ and $j$ 
ensure that the maps from $B\sX \times B$ and $B\sY \times B$
into $B\sP$ are a positive map and an h-embedding, respectively; 
to prove the latter claim, one can make use of Lemma~\ref{lm:graphmodel}. 
The analogous statements are then true for the maps from these spaces
into $(B\sP)'$.  Finally, our replacement of $B\sP$ by $(B\sP)'$ 
ensures that the map $(B\sX \sqcup B\sY)\times B \to (B\sP)'$ 
is a closed fibrewise cofibration as required.
\end{proof}

Let $G$ be a compact Lie group. 
In the situation of Lemma~\ref{lm:cobs-from-groupoids}, let us assume 
further that the map $i$ is injective on objects in each fibre.
We can then form the composite map

\begin{equation}
\label{eq:ffgpdop}
\newcommand{\ffgpdopdomain}{H_{\ast+\mathrm{shift}}(B) \tensor H_\ast\left(BG^{B\sX}\right)}
\vcenter{\xymatrix@R=1ex@C=-4em{
 *+[l]	{\ffgpdopdomain}
 \ar[r]^-{\mathrm{(a)}}_\isom
 &
 *+[r]{H_{\ast+\mathrm{shift}}(B) \tensor H_\ast\left(B(G^{\sX})\right)}
 \\
 *+[l]{\phantom{\ffgpdopdomain}}
 \ar[r]^\times_\isom
 &
 *+[r]{H_{\ast+\mathrm{shift}}\left(B \times B(G^{\sX})\right)}
 \\
 *+[l]{\phantom{\ffgpdopdomain}}
 \ar[r]^{\mathrm{(b)}}_\isom
 &
 *+[r]{H_{\ast+\mathrm{shift}}\left(B(G^{\sX\times B})\right)}
 \\
 *+[l]{\phantom{\ffgpdopdomain}}
 \ar[r]^{\mathrm{!}}
 &
 *+[r]{H_\ast\left(B(G^{\sP|{\Ob \sX}})\right)}
 \\
 *+[l]{\phantom{\ffgpdopdomain}}
 \ar[r]^{\mathrm{(c)}}_\isom
 &
 *+[r]{H_\ast\left(B(G^{\sP|({\Ob\sX \cup \Ob\sY)}})\right)}
 \\
 *+[l]{\phantom{\ffgpdopdomain}}
 \ar[r]^{\mathrm{(d)}}
 &
 *+[r]{H_\ast\left(B(G^{\sY \times B})\right)}
 \\
 *+[l]{\phantom{\ffgpdopdomain}}
 \ar[r]^{\mathrm{(e)}}_\isom
 &
 *+[r]{H_\ast\left(B \times B(G^{\sY})\right)}
 \\
 *+[l]{\phantom{\ffgpdopdomain}}
 \ar[r]^{\mathrm{\pr_\ast}}
 &
 *+[r]{H_\ast\left(B(G^{\sY})\right)}
 \\
 *+[l]{\phantom{\ffgpdopdomain}}
 \ar[r]^{\mathrm{(f)}}_\isom
 &
 *+[r]{H_\ast\left(BG^{B\sY}\right).}
}}
\end{equation}
Here `$\mathrm{shift}$'
denotes $\dim(G)\cdot\chi((B\sP)',B\sX)$, while
$\sP|\Ob X$ (resp.~$\sP|(\Ob X \cup \Ob Y)$) 
denotes the full fibrewise subgroupoid of $\sP$ 
spanned by the objects in the image of $i$ 
(resp.~the objects in the image of $i$ or $j$);
(a) and (f) are the maps induced by the natural homotopy equivalence of 
Proposition~\ref{prop:funmapcmp}; 
(b) and (e) are the canonical isomorphisms;
the map labelled `$!$' is the umkehr map of 
subsection~\ref{subsec:umkehr-maps-for-fibrewise-manifolds}
associated to the map of fibrewise manifolds
\[\xymatrix{
	B(G^{\sP|\Ob X}) 
	\ar[rr]
	\ar[dr]
	&&
	B(G^{\sX \times B})
	\ar[dl]
	\\
	&
	B \times B(G^{\Ob X})
}\]
induced by $i$; 
(c) is the inverse of the isomorphism induced
by the inclusion of $\sP| \Ob \sX$ into $\sP|(\Ob\sX \cup \Ob\sY)$,
which is an equivalence; 
and (d) is the map induced by the inclusion of 
$\sY \times B$ into $\sP|(\Ob\sX \cup \Ob\sY)$.
We now have the following lemma.

\begin{lemma}
\label{lm:gpdopcalc}
Suppose that in the situation of Lemma~\ref{lm:cobs-from-groupoids},
the map $i$ is injective on objects in each fibre. Then the 
operation 
\[
	\Phi^G((B\sP)'/B)
	\colon 
	H_{\ast + \mathrm{shift}}(B) 
			\tensor 
	H_\ast(BG^{B\sX})
	\to
	H_\ast(BG^{B\sY})
\]
agrees with the composite~\eqref{eq:ffgpdop}.
\end{lemma}
\begin{proof}
The argument of Lemma~\ref{lem:iota} shows that 
for any finite free groupoid $\sQ$, there is a natural 
isomorphism
$\sQ \to \Pi_1(B\sQ,\Ob \sQ)$
which is given by the identity map on objects and 
which sends a morphism to the 
homotopy class of paths represented by the corresponding 1-simplex 
in $B\sQ$. The claim now follows by tracing 
through the definition of the operation $\Phi^G((B\sP)'/B)$.
\end{proof}

Let $\sQ$ be a finite free groupoid. Recall from 
subsection~\ref{subsec:groupoids-to-manifolds} that 
$\fun(\sQ,G) = \Ob(G^\sQ)$
denotes the space of functors from $\sQ$ to $G$, and that 
the category $G^\sQ$ can be identified with the action groupoid
of the action of $G^{\Ob \sQ}$ on $\fun(\sQ,G)$ given by
\[
	(\delta\cdot f)(\alpha)
	=
	\delta(y)\cdot f(\alpha)\cdot\delta(x)^{-1}
\]
for $f\in \fun(\sQ,G)$, $\delta\in G^{\Ob(\sQ)}$, and 
$\alpha\colon x\to y$ a morphism of $\sQ$. We will also 
use the notation $\fun(\sQ,G)$ for $\Ob(G^\sQ)$ when $\sQ$ is a 
fibrewise finite free groupoid. Observe that the aforementioned
identification of $G^\sQ$ with the action groupoid of 
an action of $G^{\Ob\sQ}$ on $\fun(\sQ,G)$ generalizes 
to this fibrewise situation. Finally, 
recall that for a space $X$ equipped with an action
of a group $\Gamma$, we use $X \sslash \Gamma$ to denote 
the homotopy orbit space $E\Gamma\times_\Gamma X$.

\begin{lemma}
\label{lm:bgpiso}
Let $\sP_0$ be a finite free groupoid equipped with 
an action of a discrete group $\Gamma$ which fixes all objects of $\sP_0$.
Let $\sP$ denote the fibrewise finite free groupoid 
$E\Gamma \times_\Gamma \sP_0$ over $B\Gamma$. Then there is a natural
isomorphism of fibrewise manifolds
\[\xymatrix@!C@C-4em{
	B(G^\sP) 
	\ar[rr]^\isom
	\ar[dr]
	&&
	\fun(\sP_0,G)\sslash (\Gamma \times G^{\Ob \sP_0})
	\ar[dl]
	\\
	&
	B\Gamma \times BG^{\Ob \sP_0} 
}\]
where the action of $\Gamma \times G^{\Ob \sP_0}$ on 
$\fun(\sP_0,G)$ is given by 
\[
	((\gamma,\delta)\cdot f)(\alpha)
	 =
	\delta(y)\cdot f(\gamma^{-1}\cdot\alpha)\cdot\delta(x)^{-1}
\]
for $f\in \fun(\sP_0,G)$, $(\gamma,\delta)\in \Gamma \times G^{\Ob(\sP_0)}$, 
and $\alpha\colon x\to y$ a morphism of $\sP_0$.
\end{lemma}
\begin{proof}
The aforementioned identification of $G^\sP$ with an action groupoid
gives an isomorphism 
\begin{equation}
\label{eq:1stiso}
	B(G^\sP) 
	\isom 
	EG^{\Ob \sP_0} \times_{G^{\Ob \sP_0}} \fun(\sP,G)
\end{equation}
over $BG^{\Ob\sP_0} \times B\Gamma$.
By the construction of $\sP$, we have a $G^{\Ob\sP_0}$-equivariant
isomorphism 
\begin{equation}
\label{eq:2ndiso}
	\fun(\sP,G) \isom E\Gamma \times_\Gamma \fun(\sP_0,G)
 \end{equation}
over $B\Gamma$
where $\Gamma$ acts on $\fun(\sP_0,G)$ by 
\[
	(\gamma\cdot f)(\alpha) = f(\gamma^{-1}\cdot \alpha)
\]
for $f\in \fun(\sP_0,G)$, $\gamma \in \Gamma$, 
and $\alpha$ a morphism of $\sP_0$. Finally, we have an isomorphism
\begin{equation}
\label{eq:3rdiso}
	EG^{\Ob \sP_0} 
		\times_{G^{\Ob \sP_0}} 
	(E\Gamma \times_\Gamma \fun(\sP_0,G))
		\isom
	E(\Gamma \times G^{\Ob \sP_0})
		\times_{\Gamma \times G^{\Ob \sP_0}}
	\fun(\sP_0,G) 
\end{equation}
over $BG^{\Ob \sP_0}\times B\Gamma$. The desired isomorphism is now
given by combining 
\eqref{eq:1stiso}, \eqref{eq:2ndiso} and \eqref{eq:3rdiso}.
\end{proof}

Let us now construct the family of h-graph cobordisms 
$S/B(\Z/2) \colon \pt \hto \pt$.
Let $\sP_0 = \Pi_1(S_0,\{p,q\})$ where $S_0\colon \{p\} \to \{q\}$
is the h-graph cobordism of \eqref{eq:modelcob}. 
The paths $l$ and $c$ depicted in \eqref{eq:modelcob}  give
a basis for $\sP_0$, and the automorphism of $\sP_0$ sending $l$
to $l^{-1}$ and $c$ to $c\cdot l^{-1}$ defines an action of 
$\Z/2$ on $\sP_0$. Let $\sP = E(\Z/2)\times_{\Z/2} \sP_0$, and 
let $\sT_p$ denote the trivial finite free groupoid with one object $p$,
and likewise for $\sT_q$. Applying $E(\Z/2)\times_{\Z/2}(-)$
to the diagram of inclusions
\[
	\sT_p \longto \sP_0 \longot \sT_q
\]
gives the following diagram of fibrewise finite free groupoids 
over $B(\Z/2)$:
\begin{equation}
\label{diag:preszigzag}
	\sT_p \times B(\Z/2) \longto \sP \longot \sT_q \times B(\Z/2).
\end{equation}
\begin{definition}
\label{def:family-s}
We define
$S/B(\Z/2)$ to be the family $(B\sP)'/B(\Z/2)$
obtained from the zigzag~\eqref{diag:preszigzag}
by the procedure of Lemma~\ref{lm:cobs-from-groupoids}.
\end{definition}

That the fibres
of $S/B(\Z/2)$ are indeed modelled on the h-graph cobordism $S_0$
as claimed follows from Proposition~\ref{prop:pi1cmp}.

Let us now restrict to the case where $G$ is a finite group.
Interpret $\Z/2 = \{\pm 1\}$, and let $\Z/2 \times G$ 
act on $G$ by
\begin{equation}
\label{eq:action-formula}
	(\varepsilon,g_p)\cdot g_l = g_p g_l^\varepsilon g_p^{-1}
\end{equation}
for $(\varepsilon,g_p)\in \Z/2 \times G$ and $g_l\in G$.
Let $C^e_g$ denote the stabilizer of $g\in G$ in this action,
and for $\varepsilon \in \Z/2$, define 
\[
	[\varepsilon] = 
		\begin{cases}
		0 & \text{if } \varepsilon = 1\\
		1 & \text{if } \varepsilon = -1.
		\end{cases}
\]
Consider the  diagram
\begin{equation}
\label{diag:final-push-pull}
\vcenter{\xymatrix@!0@C=8em@R=8ex{
	&
	\bigsqcup_{[g]_e} \pt\sslash C_g^e
	\ar[dl]
	\ar[dr]
	\\
	\pt\sslash (\Z/2\times G)
	&&
	\pt\sslash G
}}
\end{equation}
where the disjoint union runs over a set of representatives for the 
orbits of $G$ in the above $(\Z/2\times G)$-action,
the left-hand map is given by the inclusions of $C_g^e$ into 
$\Z/2\times G$, and the right-hand map is given by the homomorphisms
\[
	C_g^e \longto G,\quad
	(\varepsilon,g_p) \longmapsto g_p g^{[\varepsilon]}.
\]
We then have the following proposition.
\begin{proposition}
\label{prop:sg-push-pull}
The operation 
\begin{equation}
\label{eq:sg-op}
	\Phi^G(S/B(\Z/2)) 
	\colon
	H_\ast(B\Z/2) \times H_\ast(BG) 
	\longto
	H_\ast(BG)
\end{equation}
agrees with the composite
\[
	\newcommand{\sopdomain}{H_\ast(B\Z/2) \tensor H_\ast(BG)}
	\xymatrix@R=1ex@C-1em{
	 *+[l]{\sopdomain}
	 \ar[r]^-\times
	 &
	 \rlap{$H_\ast(B\Z/2 \times BG)$}\phantom{H_\ast(B\Z/2 \times BG)}
	 \\
	 *+[l]{\phantom{\sopdomain}}
	 \ar[r]^-{!}
	 &
	 \rlap{$H_\ast\big(\bigsqcup_{[g]_e} BC_g^e\big)$}\phantom{H_\ast(B\Z/2 \times BG)}
	 \\
	 *+[l]{\phantom{\sopdomain}}
	 \ar[r]
	 &
	 \rlap{$H_\ast (BG)$}\phantom{H_\ast(B\Z/2 \times BG)}
	}
\]
where the middle map is the transfer map associated to the 
left-hand map in \eqref{diag:final-push-pull} and the last
map is the map induced by the right-hand map in 
\eqref{diag:final-push-pull}.
\end{proposition}
\begin{proof}
By Lemmas~\ref{lm:gpdopcalc} and \ref{lm:bgpiso},
the operation \eqref{eq:sg-op}
can be computed by a push-pull construction in the diagram
\begin{equation}
\label{diag:pushpull1}
\vcenter{\xymatrix@!0@C=1.85em@R=5ex{
	&&
	\fun(\sP_0|\{p\},G)\sslash (\Z/2 \times G^{\{p\}})
	\ar[ddll]_{!}
	&&&&&&&&
	\ar[llllllll]_-\homot
	\fun(\sP_0,G)\sslash (\Z/2 \times G^{\{p,q\}})
	\ar[ddrr]
	\\
	\\
	\pt \sslash (\Z/2 \times G^{\{p\}})
	&&
	&&&&&&&&
	&&
	\pt \sslash G^{\{q\}}
}}
\end{equation}
Here $\sP_0|\{p\}$ denotes the full subgroupoid of $\sP_0$ spanned 
by the object $p$; the top horizontal map is induced by the inclusion
of $\sP_0|\{p\}$ into $\sP$ and the projection of $G^{\{p,q\}}$ onto
$G^{\{p\}}$; the oblique arrows are projection maps; and `$!$' 
denotes the map of which an umkehr map should be taken. 
By Lemma~\ref{lm:umkehrs-for-maps-bw-coverings},
in our context of a finite group $G$, this umkehr map is simply
the transfer map.
Evaluation against $l$ gives an isomorphism
\[
	\fun(\sP_0|\{p\},G) \xto{\ \isom\ } G
\]
under which the action of $\Z/2 \times G^{\{p\}}$ on 
$\fun(\sP_0|\{p\},G)$ corresponds to the action on $G$
given by the formula \eqref{eq:action-formula}.
Similarly, evaluation against $l$ and $c$ give
an isomorphism
\[
	\fun(\sP_0,G) \xto{\ \isom\ } G^2
\]
under which the action of  $\Z/2 \times G^{\{p,q\}}$ on 
$\fun(\sP_0,G)$ corresponds to the action on $G^2$ given by
\[
	(\varepsilon,g_p,g_q)\cdot(g_l,g_c)
	 = 
	(g_p g_l^\varepsilon g_p^{-1}, 
	   g_q g_c g_l^{-[\varepsilon]} g_p^{-1})
\]
for $(\varepsilon,g_p,g_q)\in \Z/2 \times G^{\{p,q\}}$ and 
$(g_l,g_c) \in G^2$.
Using the above formulas
to decompose $\fun(\sP_0|\{p\},G) \isom G$ and 
$\fun(\sP_0,G) \isom G^2$ into orbits under the respective
actions, we obtain the following commutative diagram of 
homeomorphisms and homotopy equivalences.
\begin{equation}
\label{diag:pushpullaux}
\vcenter{
\xymatrix@C-2em{
&
\bigsqcup_{[g]_e}\pt \sslash C_g^e
\ar[dl]_\homot
\ar[dr]^\homot
\\
\bigsqcup_{[g]_e} 
\calO(g) \sslash (\Z/2 \times G^{\{p\}})
\ar[d]_\isom
&&
\bigsqcup_{[g]_e} 
\calO(g,e) \sslash (\Z/2 \times G^{\{p,q\}})
\ar[d]^\isom
\\
G \sslash (\Z/2 \times G^{\{p\}})
&&
\ar[ll]_\homot
G^2\sslash (\Z/2 \times G^{\{p,q\}})
\\
\ar[u]^\isom
\fun(\sP_0|\{p\},G) \sslash (\Z/2 \times G^{\{p\}})
&&
\ar[u]_\isom
\ar[ll]_\homot
\fun(\sP_0,G) \sslash (\Z/2 \times G^{\{p,q\}})
}}
\end{equation}
Here the bottom horizontal arrow is the one appearing in 
diagram~\eqref{diag:pushpull1} and the upper horizontal 
arrow is given by the projection of $G^2$ onto its first coordinate
and the projection of $\Z/2\times G^{\{p,q\}}$ onto $\Z/2\times G^{\{p\}}$;
$\calO(g)$ denotes the orbit of the element $g\in G$ in the 
$\Z/2\times G^{\{p\}}$-action on $G$,
and likewise $\calO(g,e)$ denotes the 
orbit of the element $(g,e)\in G^2$ in the 
$\Z/2\times G^{\{p,q\}}$-action on $G^2$;
the disjoint unions in the top pentagon run over a set of 
representatives for the orbits of the 
$\Z/2\times G^{\{p\}}$-action on $G$;
the upper vertical maps are induced by the inclusions of $\calO(g)$
into $G$ and $\calO(g,e)$ into $G^2$;
the top oblique arrow on the left is given by the inclusions
of $C_g^e$ into $\Z/2\times G^{\{p\}}$ and the maps that send
$\pt$ to $g\in \calO(g)$; and finally the top oblique arrow on the 
right is given is given by the homomorphisms
\[
	C_g^e \longto \Z/2 \times G^{\{p,q\}},\quad
	(\varepsilon,g_p) \longmapsto (\varepsilon, g_p, g_p g^{[\varepsilon]})
\]
and the maps that send $\pt$ to the element $(g,e)\in \calO(g,e)$.
The claim now follows by combining diagrams 
\eqref{diag:pushpull1} and \eqref{diag:pushpullaux}.
\end{proof}

Theorem~\ref{thm:calc} now follows easily from 
Proposition~\ref{prop:sg-push-pull}.

\begin{proof}[Proof of Theorem~\ref{thm:calc}]
Specializing to $G=\Z/2$ in Proposition~\ref{prop:sg-push-pull},
we obtain the following diagram for computing
the operation $\Phi^{\Z/2}(S/B(\Z/2))$ where `!' denotes the map
of which a transfer map should be taken:
\begin{equation*}
\vcenter{\xymatrix@!0@C=8em@R=8ex{
	&
	\pt\sslash (\Z/2 \times \Z/2) \sqcup \pt\sslash (\Z/2 \times \Z/2)
	\ar[dl]_{!}
	\ar[dr]
	\\
	\pt\sslash (\Z/2\times \Z/2)
	&&
	\pt\sslash (\Z/2) 
}}
\end{equation*}
Here the left-hand map is the identity map on both summands
and the right-hand map is induced by the projection map
\[
	\Z/2\times\Z/2 \longto \Z/2,\quad
	 (\varepsilon,g) \longmapsto g
\]
and the multiplication map
\[
	 \Z/2\times\Z/2 \longto \Z/2,\quad
	 (\varepsilon,g) \longmapsto \varepsilon g.
\]
The claim follows.
\end{proof}

Having proved Theorem~\ref{thm:calc},
we now turn to the promised application to the homology of $\Hol(F_n)$.
Composing the family $S/B(\Z/2)\colon \pt \hto \pt$ with itself $n$ times,
we obtain a family $S^{\ucirc n}/B(\Z/2)^n \colon \pt \hto \pt$.
Its fibres are modelled on the $n$-fold composite of the h-graph
cobordism $S_0$ with itself, or what is the same, on the 
h-graph cobordism $T_0 \colon \{p\} \hto \{q\}$ 
pictured below:
\begin{equation}
\label{eq:t0cob}
\begin{tikzpicture}[
		baseline=(current  bounding  box.center),
		scale=0.03,
		->-/.style={	decoration={	markings,
		  	mark=at position 0.5 with {\arrow{>}}
			},postaction={decorate}},
		->--/.style={	decoration={	markings,
		  	mark=at position 0.46 with {\arrow{>}}
			},postaction={decorate}},
		-->-/.style={	decoration={	markings,
		  	mark=at position 0.58 with {\arrow{>}}
			},postaction={decorate}}
	]
	\clip (-50, -20) rectangle (130, 70);
	\path[ARC, fill=black] (0,0) circle (3);
	\path[ARC, fill=black] (120,0) circle (3);
	\path[ARC,->-] (0,0) -- (120,0) node [midway, above] {$c$};
	\node [below] at (0,0) {$p$};
	\node [below] at (120,0) {$q$}; 
	\path[ARC,->--] (0,0) .. controls (-20,80) and (-80,45) 
						.. (0,0) node [midway, above] {$l_1$};
	\node at (1.5,35) {$\cdots$};
	\path[ARC,-->-] (0,0) .. controls (80,45) and (20,80)
						.. (0,0)  node [midway, above] {$l_n$};
\end{tikzpicture}
\end{equation}
By the gluing axiom of HHGFTs, Theorem~\ref{thm:calc} has the following 
immediate corollary.
\begin{corollary}
\label{cor:compositeop}
The operation
\[
	\Phi^{\Z/2}(S^{\ucirc n}/B(\Z/2)^n)
	\colon 
	H_\ast (B(\Z/2)^n) \tensor H_\ast B(\Z/2) 
	\longto 
	H_\ast B(\Z/2)
\]
is the map given by
\[
	(a_1 \times \cdots \times a_n)\tensor b \longmapsto
		\begin{cases}
		 a_1 \cdot \ldots\cdot a_n \cdot b 
		 	&\text{if the degree of $a_i$ is positive for all $i$}\\
		 0 &\text{otherwise}
		\end{cases}
\]
for homogeneous elements $a_1,\ldots,a_n,b \in H_\ast B(\Z/2)$.\qed
\end{corollary}

By Corollary~\ref{cor:universal} and Remark~\ref{rk:whisker},
the family $S^{\ucirc n}/B(\Z/2)^n$ admits a 2-cell
\[
	\varphi\colon S^{\ucirc n}/B(\Z/2)^n
	\Longrightarrow 
	U\hAut(T_0) / B\hAut(T_0),
\]
and by Proposition~\ref{pr:cc}, the projection 
\[
	\pi\colon \hAut(T_0) \longto \pi_0(\hAut(T_0))
\] 
is a homotopy equivalence.
The component group
$\pi_0(\hAut(T_0))$ can be shown to be isomorphic to the group 
$F_n \rtimes \Aut(F_n) = \Hol(F_n)$.
(It also agrees with the group $A^2_{n,0}$ of 
Example~\ref{ex:free-groups-with-boundary}.)
Concretely, an isomorphism between these groups is given by the map
\[
	\alpha\colon\Hol(F_n) \longto \pi_0\hAut(T_0)
\]
sending an element $(w,\theta)\in F_n\rtimes \Aut(F_n)$
to the component of the map $T_0 \to T_0$ whose restriction 
to the edge $l_i$ in \eqref{eq:t0cob} is the path
$\theta(x_i)|_{l_1,\ldots,l_n}$
and whose restriction to the edge $c$ in \eqref{eq:t0cob} 
is the path  $c \cdot w|_{l_1,\ldots,l_n}$.
Here $x_i$ denotes the $i$-th basis element of $F_n$, and 
$v|_{l_1,\ldots,l_n}$ for $v\in F_n$ denotes the result of 
substituting $l_i$ for $x_i$ in $v$ for every $i=1,\ldots,n$.
The base change axiom of HHGFTs now implies the following 
consequence for Corollary~\ref{cor:compositeop}.
\begin{corollary}
\label{cor:hol-application}
Let $a_1,\ldots, a_n \in H_{\ast}B(\Z/2)$ be homogeneous elements
of positive degree such that the product $a_1\cdot \ldots \cdot a_n$
is non-zero. Then the composite map
\[\xymatrix@1@C+1em{
	H_\ast(B(\Z/2)^n) 
	\ar[r]^-{(\varphi_B)_\ast}
	&
	H_\ast B\hAut(T_0)
	\ar[r]^-{(B\pi)_\ast}_-\isom
	&
	H_\ast B\pi_0(\hAut(T_0))
	\ar[r]^-{(B\alpha^{-1})_\ast}_-\isom
	&
	H_\ast B\Hol(F_n)
}\]
sends $a_1\times \cdots \times a_n \in 	H_\ast(B(\Z/2)^n)$
to a non-zero element of $H_\ast B\Hol(F_n)$. \qed
\end{corollary}

%% file: app-fibredspaces.tex

\section{Fibred spaces}\label{appendix:fibred-spaces}
This appendix recalls various notions
from the theory of fibred spaces, and gives proofs of some
results which we could not locate in the literature.
Our aim here is to provide the reader with background
for the work done elsewhere in the paper.
For fuller discussions we refer the reader to
the monographs \cite{CrabbJames} and \cite{MaySigurdsson}.
A very useful compilation of relevant results together with further pointers to the 
literature can be found in Chapter~IV, section~1 of \cite{Rudyak}.
We remind the reader that we work in the category of $k$-spaces throughout,
and that in addition all base spaces are assumed to be weak Hausdorff
(see subsection~\ref{subsec:notation-conventions}).

Let $B$ be a base space.  A \emph{space fibred over $B$}, or just
\emph{a space over $B$}, is simply a map $X\to B$, which we 
think of as a collection of spaces parameterized by the points of $B$ and
bound together by the topology of $X$. If $X\to B$ and $Y\to B$
are spaces over $B$, a \emph{fibrewise map} or a \emph{map over} $B$
from $X$ to $Y$ is a map $f\colon X \to Y$ such that the triangle
\[\xymatrix{
	X 
	\ar[rr]^f
	\ar[dr]
	&&
	Y
	\ar[dl]
	\\
	&
	B
}\]
commutes.
We write $X_b$ for the fibre of $X$ over a point $b\in B$.

If $B$ is a weak Hausdorff space, the category of spaces over $B$ is 
cartesian closed, so that there exist
\emph{fibrewise mapping spaces} $\Map_B(X,Y) \to B$ with
the property that fibrewise maps
\[
 X\times_B Y \longto Z
\]
are in natural bijection with fibrewise maps
\[
 X \longto \Map_B(Y,Z)
\]
for all spaces $X,Y$ and $Z$ over $B$. 
See \cite[section~1.3]{MaySigurdsson}.
The fibre of $\Map_B(X,Y)$ over a point $b \in B$ is simply
the space of maps $\Map(X_b,Y_b)$, and the fibrewise
mapping space $\Map_B(X,Y)$ for the trivial fibred spaces
$X = X_1 \times B$ and $Y = Y_1\times B$ is simply the trivial 
fibred space $\Map(X_1,Y_1) \times B$.

There are evident notions of \emph{fibrewise homotopies} of fibrewise maps,
and of \emph{fibrewise homotopy equivalences} between spaces over $B$.
There are also fibrewise analogues of fibrations and cofibrations.
A map $p\colon X\to Y$ of spaces fibred over $B$ is a \emph{fibrewise fibration}
if, given a space $A$ over $B$ and maps over $B$ making
the solid diagram on the left commute,
a dotted map can be found making the whole diagram commute.
\[\xymatrix{
	A\ar[r]\ar[d] & X\ar[d]^p\\
	A\times I\ar[r]\ar@{-->}[ur] & Y
}
\qquad\qquad\qquad
\xymatrix{
	U\ar[r]\ar[d]_i & \Map_B(B\times I,E)\ar[d]\\
	V\ar[r]\ar@{-->}[ur] & E
}\]
Similarly, a map $i\colon U\to V$ over $B$
is a \emph{fibrewise cofibration} if, given a space $E$
over $B$ and maps over $B$ making the solid diagram on the right commute,
there exists a dotted arrow making the whole diagram commute.
(The unlabelled vertical maps in the two diagrams are both
induced by the inclusion $\{0\}\hookrightarrow I$.)
A \emph{closed} fibrewise cofibration
is a fibrewise cofibration that is in addition a closed inclusion.

(The notions of fibrewise fibrations, fibrewise cofibrations and 
closed fibrewise cofibrations as defined above agree with what 
May and Sigurdsson call $f$-fibrations, $f$-cofibrations and 
$\bar{f}$-cofibrations, respectively. 
See Definition~1.3.2, Definition~5.1.7 and Theorem 5.2.8.(i) of \cite{MaySigurdsson}.
Notice that May and Sigurdsson write $A\times_B I$ and 
$\Map_B(I,E)$ in place of $A\times I$ and 
$\Map_B(B \times I, E)$, respectively.)

Most of the fibred spaces appearing in this paper are fibrations,
but at certain points we will require the following more general notion.
A map $p\colon X\to B$ is a \emph{Dold fibration} if,
given a homotopy $f\colon A\times I\to B$ and a map
$g_0\colon A\times\{0\}\to X$ satisfying $p\circ g_0=f|(A\times\{0\})$,
there is a homotopy $g\colon A\times I\to X$ such that $p\circ g=f$,
and such that $g|(A\times\{0\})$ is homotopic to $g_0$ through
maps over $B$.

In this paper we make significant use of results from Dold's classic paper~\cite{Dold}.
Many of these results state that if a certain fibrewise property
(of a space or map) holds over every subset in a numerable cover of the base,
then it also holds globally.
Recall that a cover $\{V_\lambda\}_{\lambda\in\Lambda}$ of $B$
is \emph{numerable} if there is a locally-finite partition of unity
on $B$ whose supports refine $\{V_\lambda\}_{\lambda\in\Lambda}$.
We refer to the results of \cite{Dold} explicitly whenever we use them.

The following result is of a similar flavour to those from \cite{Dold},
but we were unable to find a reference and so give the proof here.
It is used to show that an open-closed cobordism $\Sigma$
determines a family of h-graph cobordisms over $B\Diff(\Sigma)$.

\begin{proposition}\label{pr:locality}
Let $f\colon X\to Y$ be a fibrewise map between spaces over $B$.
If $f$ is a fibrewise (co)fibration when restricted to every set in a
numerable open cover of $B$, then it is a fibrewise (co)fibration.
\end{proposition}

\begin{proof}
Appropriate adjunctions show that $U\to V$ is a fibrewise cofibration
if and only if $\Map_B(V,E)\to \Map_B(U,E)$ 
is a fibrewise fibration for all spaces $E$ over $B$.
Thus the result for fibrewise cofibrations follows from the one for fibrewise fibrations.
To prove the result for fibrewise fibrations, one can modify the arguments of
\cite[4.1,4.2,4.5]{Dold} to show that $X\to Y$ is a fibrewise fibration
if and only if for each solid square on the left above, the space 
\[
	B\times_{\Map_B(A,X)}\Map_B(A\times I,X)\times_{\Map_B(A\times I,Y)}B
\]
over $B$
has the section extension property.  The result now follows from the
Section Extension Theorem~\cite[2.7]{Dold}.
\end{proof}

In the next appendix we discuss the classification of a certain
type of relative fibration.  The following definition and lemma
will be useful there.

\begin{definition}
\label{def:equivalence-over-and-under}
Suppose given spaces $W$, $X$ and $Y$ over a base $B$,
and fibrewise maps $W\to X$ and $W\to Y$.
A map
\[
	f\colon X\longto Y
\]
over $B$ and under $W$ is called a
\emph{homotopy equivalence over $B$ and under $W$}
if it admits a homotopy inverse that is again a map over $B$ and under $W$,
and where all the homotopies are through maps over $B$ and under $W$.
\end{definition}

\begin{lemma}
\label{lem:equivalence-over-and-under}
Suppose given spaces $W$, $X$ and $Y$ over a base space $B$,
maps $W\to X$ and $W\to Y$ over $B$,
and a map $f\colon X\to Y$ over $B$ and under $W$.
Suppose that:
\begin{itemize}
	\item
	For each $b\in B$, the restriction $f|\colon X_b\to Y_b$
	of $f$ to the fibres over $b$ is a homotopy equivalence.	
	\item
	The base $B$ admits a numerable cover $\{V_\lambda\}_\lambda$
	for which the inclusions $V_\lambda\hookrightarrow B$
	are all nullhomotopic.
	\item
	The maps $X\to B$ and $Y\to B$ are Dold fibrations.
	\item
	The maps $W\to X$ and $W\to Y$ are fibrewise cofibrations.
\end{itemize}
Then $f$ is a homotopy equivalence over $B$ and under $W$.
\end{lemma}

The second condition holds whenever $B$ is locally contractible and 
paracompact.  In particular it holds if $B$ is a CW-complex or,
more generally, one of the good base spaces introduced in Definition~\ref{def:calB-calU}.

\begin{proof}
By~\cite[Theorem~6.3]{Dold} the map $f$ is a fibrewise
homotopy equivalence. The claim now follows from the
evident fibrewise generalization of the first proposition in section 6.5 of
\cite{MayConcise}.
\end{proof}

%% file: app-Fd-fibrations.tex

\section{$(F,\partial)$-fibrations}\label{app:Fd-fibrations}

Let $F$ be a space and let $\partial\subset F$ be a subspace such that
both $F$ and $\partial$ have the homotopy type of finite CW complexes
and such that the inclusion $\partial \incl F$ is a closed cofibration.
Intuitively, an $(F,\partial)$-fibration is a fibration where each
fibre is modeled by the pair $(F,\partial)$ and where it is 
required that the subspaces corresponding to $\partial$ 
stay constant as we move from fibre to fibre. The following
definition makes this idea precise. The examples we are
interested in are given by families of h-graph cobordisms,
with $\partial$ given by the disjoint union of the incoming and
outgoing boundaries.

\begin{definition}
Suppose $B$ is a base space.  An $(F,\partial)$-\emph{fibration} $E\to B$ 
over $B$ consists of a fibration $E\to B$ 
 together with a closed fibrewise cofibration $\partial \times B \to E$ 
 over $B$ such that for each point $b \in B$, there exists a homotopy 
 equivalence $F \to E_b$ making the diagram
 \[\xymatrix{
 	&
	\partial
	\ar[dl]
	\ar[dr]
	\\
 	F 
	\ar[rr]
	&&
	E_b
 }\]
commute.
\end{definition}

The main purpose of this section is to state a classification
theorem for $(F,\partial)$-fibrations. We will first define the 
notion up to which $(F,\partial)$-fibrations are classified.

\begin{definition}
Suppose $D\to B$ and $E \to B$ are $(F,\partial)$-fibrations over 
the same base space. An $(F,\partial)$-\emph{map} from $D\to B$
to $E \to B$ is a map $D\to E$ making the diagram
\[\xymatrix{
	& \partial \times B 
	\ar[dl]
	\ar[dr]
	\\
	D 
	\ar[rr]
	\ar[dr]
	&&
	E
	\ar[dl]
	\\
	&
	B
}\]
commutative and having the property that for each $b\in B$, the induced
map $D_b \to E_b$ between fibres is a homotopy equivalence.
Two $(F,\partial)$-fibrations over the same space are \emph{equivalent}
if they can be connected by a zigzag of $(F,\partial)$-maps.
We denote the collection of equivalence classes of 
$(F,\partial)$-fibrations over a space $B$ by 
$(F,\partial)$-$\mathrm{FIB}(B)$.
\end{definition}

\begin{remark}
\label{rk:fd-equiv-rel}
Observe that any homotopy equivalence over $B$ and under $\partial \times B$ (see Definition~\ref{def:equivalence-over-and-under})
between $(F,\partial)$-fibrations is an $(F,\partial)$-map. 
Conversely, if $B$ is well-behaved in the sense that
it satisfies the second condition of Lemma~\ref{lem:equivalence-over-and-under},
for example if it is paracompact and locally contractible,
then by Lemma~\ref{lem:equivalence-over-and-under}
every $(F,\partial)$-map is a homotopy equivalence over $B$ and under
$\partial \times B$.  It follows that over such well-behaved base spaces,
our equivalence relation on $(F,\partial)$-fibrations 
simply amounts to homotopy equivalence over $B$ and under $\partial \times B$.
\end{remark}

Let $\hAut(F,\partial)$ denote the topological 
monoid of homotopy equivalences from $F$ to itself fixing 
the subspace $\partial$ pointwise,
and let $\hAut^w(F,\partial) = \hAut(F,\partial) \cup_{\id} I$ be the topological monoid obtained from $\hAut(F,\partial)$ 
by growing a whisker at the identity element of the monoid
as in \cite[A.8]{MayGILS}.
The arguments presented in \cite{MayClassifying}
in the special case $\partial=\pt$ now generalize in 
a straightforward way
to prove the following classification theorem. 

\begin{theorem}
\label{thm:fd-classification}
There exists a \emph{universal} $(F,\partial)$-fibration 
\[U\hAut^w(F,\partial) \longto B\hAut^w(F,\partial)\]
such that the map
\[
	[B,B\hAut^w(F,\partial)] \longto (F,\partial)\text{-}\mathrm{FIB}(B),\qquad
	[f] \longmapsto [f^*U\hAut^w(F,\partial) ]
\]
is a bijection for every base space $B$ having the homotopy type 
of a CW complex.\qed
\end{theorem}

Explicitly, the universal $(F,\partial)$-fibration can be constructed
as follows. Let $F'$ denote the mapping cylinder of the inclusion 
$\partial \incl F$, and for brevity write $H$ and $H'$ for 
$\hAut(F,\partial)$ and $\hAut^w(F,\partial)$, respectively.
The monoid $H$ acts on $F'$
by acting on $F$ in the evident way while keeping the image of the 
cylinder $\partial \times I$ fixed. 
Via the monoid homomorphism 
$H' \to H$,
we obtain an action of $H'$ on $F'$,
making it possible to 
form the bar construction $B(\pt,H',F')$.
To rectify the defect that the map 
\[
	B(\pt,H',F') \longto BH'
\]
is in general only a quasifibration, we replace it by 
a fibration
\[
	\Gamma B(\pt,H',F')) \longto BH'
\]
in the usual way, except that, following May, we use
Moore paths in $BH'$ in place of 
standard paths. See \cite[Definition 3.2]{MayClassifying}.
To rectify the problem that the composite map
\[
	BH' \times \partial
  	= 
	B(\pt,H',\partial)
	\longto
	B(\pt,H', F')
	\longto
	\Gamma 	B(\pt,H', F')
\]
(where the first arrow is induced by the inclusion $\partial \incl F'$)
is not necessarily a closed fibrewise cofibration, we then replace 
$\Gamma B(\pt,H',F')$ by the mapping cylinder
$\Gamma' B(\pt,H',F')$ 
of this composite map. See \cite[Definition 5.3]{MayClassifying}.
The composite map
\begin{equation}
 \label{map:univproj}
	\Gamma' B(\pt,H',F')
	\longto
	\Gamma B(\pt,H',F')
	\longto
 	BH'
\end{equation}
together with the inclusion
\[
	BH' \times \partial 
	\longincl 
	\Gamma' B(\pt,H',F')
\]
then give the desired universal $(F,\partial)$-fibration.
The composite map \eqref{map:univproj} is indeed a fibration
by \cite[Proposition~1.3]{Clapp}.

\begin{remark}
\label{rk:whisker}
The only reason why Theorem~\ref{thm:fd-classification}
features $\hAut^w(F,\partial)$ instead of the simpler monoid
$\hAut(F,\partial)$ is the need to ensure that the 
identity element of the monoid used in the bar construction
is a strongly nondegenerate basepoint, in the sense of 
\cite[section 7]{MayClassifying}.
This technical requirement is needed for various constructions
in \cite{MayClassifying}. 
In the case where the identity element of $\hAut(F,\partial)$ 
is already a strongly non-degenerate basepoint 
(as happens for example when
$F$ is a finite simplicial complex and $\partial \subset F$
is a simplicial subcomplex), the  homomorphism
$\hAut^w(F,\partial) \to \hAut(F,\partial)$ induces
a homotopy equivalence 
$B\hAut^w(F,\partial) \to B\hAut(F,\partial)$
between classifying spaces, and the 
above construction with $H$ in place of $H'$  
provides a universal $(F,\partial)$-fibration 
\[
	U\hAut(F,\partial) \longto B\hAut(F,\partial).
\]
\end{remark}

\begin{proposition}\label{pr:cc}
Let $S_0\colon X\hto Y$ be an h-graph cobordism in which
$X\sqcup Y$ meets every path-component of $S_0$.
Then the components of the homotopy automorphism monoid 
$\hAut(S_0)=\hAut(S_0,X\sqcup Y)$
are contractible.
\end{proposition}

\begin{proof}
Let $Z_0 \incl Z$ be a cofibration between h-graphs 
such that $Z_0$ meets every path-component of $Z$,
let $f_0\colon Z_0\to W$ be a fixed map into a third h-graph,
and write $M(Z,Z_0,f_0)$ for the space of maps $f\colon Z\to W$
extending $f_0$.
We will prove that the components of $M(Z,Z_0,f_0)$ are contractible.
The theorem then follows by taking $Z_0= X\sqcup Y$, $Z=W=S_0$, 
and $f_0\colon X\sqcup Y\hookrightarrow S_0$.
The claim will be proved in increasingly general special cases.
\begin{enumerate}
\item\label{pr:cc:case-one}
$(Z,Z_0)=(I,\partial I)$.
Here $M(Z,Z_0,f_0)$ is either empty or is homotopy equivalent to
$\Omega_cW$ for some $c\in W$.  The components of the latter are
contractible since $W$ is an h-graph.

\item\label{pr:cc:case-two}
$(Z,Z_0)=(\Gamma,V)$ where $\Gamma$ is a finite graph with vertex
set $V$.
Here $M(Z,Z_0,f_0)$ is a product of spaces of the kind considered
in case \eqref{pr:cc:case-one}.

\item\label{pr:cc:case-three}
$Z_0$ is finite.
Here we may find a finite graph $\Gamma$ with vertex set $V=Z_0$
and a homotopy equivalence $\Gamma\to Z$ respecting $Z_0$.
Then $M(Z,Z_0,f_0)\simeq M(\Gamma,V,f_0)$ and the claim
follows from case \eqref{pr:cc:case-two}.

\item
The general case.
Here we may find a finite $P\subset Z_0$ meeting every component.
Without loss the inclusion is a cofibration.
Writing $f_P=f_0|P$ we have a fibration sequence
\[
M(Z,Z_0,f_0)\longto M(Z,P,f_P)\longto M(Z_0,P,f_P)
\]
in which, by case \eqref{pr:cc:case-three},
the components of the second and third spaces are contractible.
The same therefore holds for the fibre.
\qedhere
\end{enumerate}
\end{proof}

%% file: master-document.bbl
\begin{thebibliography}{10}

\bibitem{MooreSegal}
Paul~S. Aspinwall, Tom Bridgeland, Alastair Craw, Michael~R. Douglas, Mark
  Gross, Anton Kapustin, Gregory~W. Moore, Graeme Segal, Bal{\'a}zs
  Szendr{\H{o}}i, and P.~M.~H. Wilson.
\newblock {\em Dirichlet branes and mirror symmetry}, volume~4 of {\em Clay
  Mathematics Monographs}.
\newblock American Mathematical Society, Providence, RI, 2009.

\bibitem{BeckerGottlieb}
J.~C. Becker and D.~H. Gottlieb.
\newblock The transfer map and fiber bundles.
\newblock {\em Topology}, 14:1--12, 1975.

\bibitem{BGNX}
Kai Behrend, Gr{\'e}gory Ginot, Behrang Noohi, and Ping Xu.
\newblock String topology for stacks.
\newblock {\em Ast\'erisque}, (343):xiv+169, 2012.

\bibitem{BoardmanNotesCh5}
J.M. Boardman.
\newblock Stable homotopy theory, {C}hapter {V}: {D}uality and {T}hom spectra.
\newblock Mimeographed Notes, Warwick University, 1966.

\bibitem{Borceux}
Francis Borceux.
\newblock {\em Handbook of categorical algebra. 2}, volume~51 of {\em
  Encyclopedia of Mathematics and its Applications}.
\newblock Cambridge University Press, Cambridge, 1994.
\newblock Categories and structures.

\bibitem{ChasSullivan}
Moira Chas and Dennis Sullivan.
\newblock String topology.
\newblock {\tt arXiv:math/9911159}, 1999.

\bibitem{ChataurMenichiPreprint}
David Chataur and Luc Menichi.
\newblock String topology of classifying spaces.
\newblock {\tt arXiv:0801.0174v3}, 2007.

\bibitem{ChataurMenichi}
David Chataur and Luc Menichi.
\newblock String topology of classifying spaces.
\newblock {\em J. Reine Angew. Math.}, 669:1--45, 2012.

\bibitem{Clapp}
M{\'o}nica Clapp.
\newblock Duality and transfer for parametrized spectra.
\newblock {\em Arch. Math. (Basel)}, 37(5):462--472, 1981.

\bibitem{CohenGodin}
Ralph~L. Cohen and V{\'e}ronique Godin.
\newblock A polarized view of string topology.
\newblock In {\em Topology, geometry and quantum field theory}, volume 308 of
  {\em London Math. Soc. Lecture Note Ser.}, pages 127--154. Cambridge Univ.
  Press, Cambridge, 2004.

\bibitem{CohenJones}
Ralph~L. Cohen and John D.~S. Jones.
\newblock A homotopy theoretic realization of string topology.
\newblock {\em Math. Ann.}, 324(4):773--798, 2002.

\bibitem{Costello}
Kevin Costello.
\newblock Topological conformal field theories and {C}alabi-{Y}au categories.
\newblock {\em Adv. Math.}, 210(1):165--214, 2007.

\bibitem{CrabbJames}
Michael Crabb and Ioan James.
\newblock {\em Fibrewise homotopy theory}.
\newblock Springer Monographs in Mathematics. Springer-Verlag London Ltd.,
  London, 1998.

\bibitem{Dold}
Albrecht Dold.
\newblock Partitions of unity in the theory of fibrations.
\newblock {\em Ann. of Math. (2)}, 78:223--255, 1963.

\bibitem{DoldLAT}
Albrecht Dold.
\newblock {\em Lectures on algebraic topology}, volume 200 of {\em Grundlehren
  der Mathematischen Wissenschaften [Fundamental Principles of Mathematical
  Sciences]}.
\newblock Springer-Verlag, Berlin, second edition, 1980.

\bibitem{EarleSchatz}
C.~J. Earle and A.~Schatz.
\newblock Teichm\"uller theory for surfaces with boundary.
\newblock {\em J. Differential Geometry}, 4:169--185, 1970.

\bibitem{FelixThomas}
Yves F{\'e}lix and Jean-Claude Thomas.
\newblock String topology on {G}orenstein spaces.
\newblock {\em Math. Ann.}, 345(2):417--452, 2009.

\bibitem{FelixThomasVigue}
Yves F{\'e}lix, Jean-Claude Thomas, and Micheline Vigu{\'e}-Poirrier.
\newblock Rational string topology.
\newblock {\em J. Eur. Math. Soc. (JEMS)}, 9(1):123--156, 2007.

\bibitem{FreedHopkinsTeleman}
Daniel~S. Freed, Michael~J. Hopkins, and Constantin Teleman.
\newblock Consistent orientation of moduli spaces.
\newblock In {\em The many facets of geometry}, pages 395--419. Oxford Univ.
  Press, Oxford, 2010.

\bibitem{Godin}
V\'eronique Godin.
\newblock Higher string topology operations.
\newblock {\tt arXiv:0711.4859v2}, 2008.

\bibitem{Guldberg}
Casper Guldberg.
\newblock Labelled string topology for classfying spaces of compact {L}ie
  groups.
\newblock Master's thesis, University of {C}openhagen, 2011.

\bibitem{Hatcher}
Allen Hatcher.
\newblock {\em Algebraic topology}.
\newblock Cambridge University Press, Cambridge, 2002.

\bibitem{HatcherWahl}
Allen Hatcher and Nathalie Wahl.
\newblock Stabilization for the automorphisms of free groups with boundaries.
\newblock {\em Geom. Topol.}, 9:1295--1336 (electronic), 2005.

\bibitem{Higgins}
P.~J. Higgins.
\newblock Categories and groupoids.
\newblock {\em Repr. Theory Appl. Categ.}, (7):1--178, 2005.
\newblock Reprint of the 1971 original [Notes on categories and groupoids, Van
  Nostrand Reinhold, London; MR0327946] with a new preface by the author.

\bibitem{Hirschhorn}
Philip~S. Hirschhorn.
\newblock {\em Model categories and their localizations}, volume~99 of {\em
  Mathematical Surveys and Monographs}.
\newblock American Mathematical Society, Providence, RI, 2003.

\bibitem{JensenWahl}
Craig~A. Jensen and Nathalie Wahl.
\newblock Automorphisms of free groups with boundaries.
\newblock {\em Algebr. Geom. Topol.}, 4:543--569, 2004.

\bibitem{LundellWeingram}
A.T. Lundell and S.~Weingram.
\newblock {\em {The topology of CW complexes.}}
\newblock {The University Series in Higher Mathematics. New York etc.: Van
  Nostrand Reinhold Company. VIII, 216 p. }, 1969.

\bibitem{OrbifoldStringTopology}
Ernesto Lupercio, Bernardo Uribe, and Miguel~A. Xicotencatl.
\newblock Orbifold string topology.
\newblock {\em Geom. Topol.}, 12(4):2203--2247, 2008.

\bibitem{MacLane}
Saunders Mac~Lane.
\newblock {\em Categories for the working mathematician}, volume~5 of {\em
  Graduate Texts in Mathematics}.
\newblock Springer-Verlag, New York, second edition, 1998.

\bibitem{MayGILS}
J.~P. May.
\newblock {\em The geometry of iterated loop spaces}.
\newblock Springer-Verlag, Berlin, 1972.
\newblock Lectures Notes in Mathematics, Vol. 271.

\bibitem{MayConcise}
J.~P. May.
\newblock {\em A concise course in algebraic topology}.
\newblock Chicago Lectures in Mathematics. University of Chicago Press,
  Chicago, IL, 1999.

\bibitem{MaySigurdsson}
J.~P. May and J.~Sigurdsson.
\newblock {\em Parametrized homotopy theory}, volume 132 of {\em Mathematical
  Surveys and Monographs}.
\newblock American Mathematical Society, Providence, RI, 2006.

\bibitem{MayClassifying}
J.~Peter May.
\newblock Classifying spaces and fibrations.
\newblock {\em Mem. Amer. Math. Soc.}, 1(1, 155):xiii+98, 1975.

\bibitem{PoirierRounds}
Kate Poirier and Nathaniel Rounds.
\newblock Compactifying string topology.
\newblock {\tt arXiv:1111.3635}, 2011.

\bibitem{Rudyak}
Yuli~B. Rudyak.
\newblock {\em On {T}hom spectra, orientability, and cobordism}.
\newblock Springer Monographs in Mathematics. Springer-Verlag, Berlin, 1998.
\newblock With a foreword by Haynes Miller.

\bibitem{SGA1}
{\em Rev\^etements \'etales et groupe fondamental ({SGA} 1)}.
\newblock Documents Math\'ematiques (Paris), 3. Soci\'et\'e Math\'ematique de
  France, Paris, 2003.
\newblock S{\'e}minaire de g{\'e}om{\'e}trie alg{\'e}brique du Bois Marie
  1960--61., Directed by A. Grothendieck, With two papers by M. Raynaud,
  Updated and annotated reprint of the 1971 original [Lecture Notes in Math.,
  224, Springer, Berlin; MR0354651 (50 \#7129)].

\bibitem{Shulman}
Michael Shulman.
\newblock Framed bicategories and monoidal fibrations.
\newblock {\em Theory Appl. Categ.}, 20:No. 18, 650--738, 2008.

\bibitem{ShulmanSMBicat}
Michael Shulman.
\newblock Constructing symmetric monoidal bicategories.
\newblock {\tt arXiv:1004.0993}, 2010.

\bibitem{TamanoiStable}
Hirotaka Tamanoi.
\newblock Stable string operations are trivial.
\newblock {\em Int. Math. Res. Not. IMRN}, (24):4642--4685, 2009.

\bibitem{TradlerZeinalianOne}
Thomas Tradler and Mahmoud Zeinalian.
\newblock On the cyclic {D}eligne conjecture.
\newblock {\em J. Pure Appl. Algebra}, 204(2):280--299, 2006.

\bibitem{TradlerZeinalianTwo}
Thomas Tradler and Mahmoud Zeinalian.
\newblock Infinity structure of {P}oincar\'e duality spaces.
\newblock {\em Algebr. Geom. Topol.}, 7:233--260, 2007.
\newblock Appendix A by Dennis Sullivan.

\bibitem{WahlAutomorphisms}
Nathalie Wahl.
\newblock From mapping class groups to automorphism groups of free groups.
\newblock {\em J. London Math. Soc. (2)}, 72(2):510--524, 2005.

\bibitem{Wahl}
Nathalie Wahl.
\newblock Universal operations in {H}ochschild homology.
\newblock {\tt arXiv:1212.6498}. To appear in J. Reine Angew. Math., 2012.

\bibitem{WahlWesterland}
Nathalie Wahl and Craig Westerland.
\newblock Hochschild homology of structured algebras.
\newblock {\tt arXiv:1110.0651}, 2011.

\end{thebibliography}
